\documentclass{article}

\usepackage[utf8]{inputenc}
\usepackage{dsfont}

\usepackage{braket}

\usepackage[shortlabels]{enumitem}
\usepackage{arxiv}
\usepackage{bbold}
\usepackage{graphicx}
\usepackage{amsmath,amsthm,amssymb}
\usepackage{mathtools,mathdots,mathrsfs}

\usepackage{verbatim}
\usepackage{bm}
\usepackage{subcaption}
\graphicspath{{figures/}}
\usepackage[numbers]{natbib}

\usepackage[dvipsnames,table]{xcolor}
\usepackage{tikz}
\usetikzlibrary{arrows.meta, 
  bending,     
  patterns     
}
\usepackage{hyperref}
\usepackage{cleveref}

\newcommand{\RR}{\mathbb{R}}
\newcommand{\HH}{\mathscr{H}}
\newcommand{\PP}{\mathscr{P}}
\newcommand{\ZZ}{\mathbb{Z}}

\newcommand{\norm}[1]{\ensuremath{\left\| #1\right\|}}

\newcommand{\vect}[1]{\boldsymbol{#1}}
\newcommand{\matr}[1]{\boldsymbol{#1}}

\newcommand{\E}{\mathrm{E}} 
\newcommand{\Cov}{\mathrm{Cov}} 
\newcommand{\Var}{\mathrm{Var}} 

\newcommand{\im}{\iota} 

\newcommand{\Ds}{\matr{\Delta(\varepsilon)}}

\newcommand{\bg}{\mathbf{g}}       
\newcommand{\N}{\mathcal{N}}       
\newcommand{\Ly}{\mathcal{L}_{\by}}       

\newcommand{\ba}{\bm{\alpha}}
\newcommand{\ff}{\bm{f}}

\newcommand{\bb}{\bm{\beta}}
\newcommand{\bga}{\bm{\gamma}}
\newcommand{\brho}{\bm{\rho}}

\newcommand{\by}{\bm{y}}

\newcommand{\va}{\vect{\alpha}}
\newcommand{\vb}{\vect{\beta}}
\newcommand{\vg}{\vect{\gamma}}
\newcommand{\vo}{\vect{\omega}}

\newcommand{\bM}{\matr{M}}

\newcommand{\bA}{\matr{A}}       
\newcommand{\bH}{\matr{H}}       
\newcommand{\bQ}{\matr{Q}}       
\newcommand{\bP}{\matr{P}}       
\newcommand{\bF}{\matr{F}}       

\newcommand{\bB}{\matr{B}}       
\newcommand{\bC}{\matr{C}}       

\newcommand{\bZ}{\matr{Z}}       
\newcommand{\bU}{\matr{U}}

\newcommand{\bI}{\matr{I}}       
\newcommand{\bV}{\matr{V}}       
\newcommand{\bW}{\matr{W}}       
\newcommand{\bK}{\matr{K}}       
\newcommand{\bKe}{\matr{K}_{\varepsilon}}       

\newcommand{\bR}{\matr{R}}       
\newcommand{\bD}{\matr{D}}       
\newcommand{\bS}{\matr{S}}       
\newcommand{\bL}{\matr{L}}
\newcommand{\R}{\mathbb{R}}       
\newcommand{\tbL}{\widetilde{\matr{L}}}
\newcommand{\tbU}{\widetilde{\matr{U}}}
\newcommand{\tbLam}{\widetilde{\bm{\Lambda}}}

\newcommand{\X}{\mathcal{X}}

\newcommand{\ft}{\tilde{f}}

\newcommand{\Y}{\mathcal{Y}}

\newcommand{\flatlim}{\varepsilon\rightarrow0}

\renewcommand{\O}{\mathcal{O}}       

\DeclareMathOperator{\mspan}{span}

\DeclareMathOperator{\orth}{orth}
\DeclareMathOperator{\argmin}{argmin}

\DeclareMathOperator{\diag}{diag}
\DeclareMathOperator{\Tr}{Tr}

\definecolor{darkgreen}{rgb}{0,0.6,0}

\newcommand{\NNP}[2]{\begin{pmatrix} #1 ; #2  \end{pmatrix}}

\newcommand{\SPM}[2]{\left(#1;#2 \right)}
\newcommand{\V}{\mathcal{V}}
\newcommand{\M}{\mathcal{M}}

\newcommand{\ones}{\vect{\mathbf{1}}}

\newcommand{\lt}{\tilde{\lambda}}
\newcommand{\gt}{\tilde{\gamma}}

\newcommand{\fX}{\vect{f}_{\X}}
\newcommand{\fXh}{\vect{\hat{f}}_{\X}}
\newcommand{\bMth}{\bM_{\vect{\theta}}}

\newcommand{\saddle}[2]{\begin{pmatrix} #1 & #2 \\ #2^\top & \vect{0} 
  \end{pmatrix}}

\newcommand{\bk}{\vect{k}}
\newcommand{\bx}{\vect{x}}

\newcommand{\dn}{\vect{\delta}_n}

\newcommand{\fleq}{\underset{\flatlim}{\propto}}

\DeclarePairedDelimiter\floor{\lfloor}{\rfloor}

\newtheorem{definition}{Definition}[section]
\newtheorem{theorem}[definition]{Theorem}
\newtheorem{proposition}[definition]{Proposition}
\newtheorem{lemma}[definition]{Lemma}
\newtheorem{corollary}[definition]{Corollary}
\newtheorem{remark}[definition]{Remark}

\newtheorem{example}{Example}[section]

\newcommand\numberthis{\addtocounter{equation}{1}\tag{\theequation}}
\newcommand{\Qort}{\matr{Q}_{\bot}}

\newcommand\ostwo{\frac{1}{\sqrt{2}}}

\newcommand{\mm}{\vect{m}}
\newcommand{\bv}{\vect{v}}

\begin{document}
\title{Gaussian Process Regression in the Flat Limit}
\author{Simon Barthelm\'e, Pierre-Olivier Amblard, Nicolas Tremblay, Konstantin
  Usevich}

\maketitle
  \begin{abstract}
  Gaussian process (GP) regression is a fundamental tool in Bayesian statistics. It is also known as kriging and is the Bayesian counterpart to the frequentist kernel ridge regression. Most of the theoretical work on GP regression has focused on a large-$n$ asymptotics, characterising the behaviour of GP regression as the amount of data increases. Fixed-sample analysis is much more difficult outside of simple cases, such as locations on a regular grid.

  In this work we perform a fixed-sample analysis that was first studied in the
  context of approximation theory by Driscoll \& Fornberg (2002), called the ``flat limit''. In flat-limit asymptotics, the goal is to characterise kernel methods as the length-scale of the kernel function tends to infinity, so that kernels appear flat over the range of the data. Surprisingly, this limit is well-defined, and displays interesting behaviour: Driscoll \& Fornberg showed that radial basis interpolation converges in the flat limit to polynomial interpolation, if the kernel is Gaussian. Subsequent work showed that this holds true in the multivariate setting as well, but that kernels other than the Gaussian may have (polyharmonic) splines as the limit interpolant.
  
  Leveraging recent results on the spectral behaviour of kernel matrices in the flat limit, we study the flat limit of Gaussian process regression. Results show that Gaussian process regression tends in the flat limit to (multivariate) polynomial regression, or (polyharmonic) spline regression, depending on the kernel. Importantly, this holds for both the predictive mean and the predictive variance, so that the posterior predictive distributions become equivalent.
  
  For the proof, we introduce the notion of prediction-equivalence of semi-parametric models, which lets us state flat-limit results in a compact and unified manner. Our results have practical consequences: for instance, they show that optimal GP predictions in the sense of leave-one-out loss may occur at very large length-scales, which would be invisible to current implementations because of numerical difficulties.
  \end{abstract}

Gaussian processes are a cornerstone of modern Bayesian methods, used almost
wherever one may require nonparametric priors. Quite naturally, the theory of
Gaussian Process methods is well-developed. Aside from limited special cases in
which Fourier analysis is applicable, GP-based methods have mostly been
studied under large-$n$ asymptotics (see, {\it e.g.}
\cite{stein1999interpolation,sollich2004using,rousseau2016frequentist,van2008rates,scheuerer2013interpolation}), which apply when
the number of measurements is high. In this paper we
report intriguing theoretical results obtained under a different asymptotic, one
that treats the data as fixed, rather than random, with fixed sample size. The
limit we look at is the so-called ``flat limit'', pioneered by
\cite{driscoll2002interpolation} in 2002. The flat limit consists in letting the
spatial width of the kernel function go to infinity, which results in the
covariance function becoming flat over the range of the data.

Studying Gaussian processes under the flat limit may seem at first sight to be
entirely pointless - does that not correspond to a prior that contains only flat
functions? The answer is no, because covariance functions have a second
hyperparameter that sets the vertical scale (pointwise variance). When one lets
pointwise variance grow as the covariance becomes wider, the actual function
space spanned by Gaussian processes remains interesting and useful. In the cases
studied here, they are (multivariate) polynomials and (polyharmonic) splines. 

A first hint that such would be the case was obtained in
\cite{driscoll2002interpolation}, where Driscoll \& Fornberg examined Radial
Basis Function interpolation in the flat limit, a popular method in
approximation theory that corresponds to noiseless GP regression. Driscoll \&
Fornberg found that under certain conditions, the RBF interpolant tends to the
Lagrange polynomial interpolant in the flat limit. The result is very
surprising, since the RBF interpolation problem may seem at first sight to
become ill-defined in the flat limit. Subsequent papers generalised this result
to multivariate interpolation \cite{larsson2005theoretical}, and finitely-smooth
kernels \cite{song2012multivariate}. Our contribution can be seen as an
extension of these results, since RBF interpolation features as a special case.

Further evidence that GPs may be interesting in the flat limit comes from the
study of the spectrum and eigenvectors of kernel matrices performed in
\cite{BarthelmeUsevich:KernelsFlatLimit}. The full story is
complicated, but the key phenomenon is that the eigenvectors of kernel
matrices tend in the flat limit to orthogonal multivariate polynomials or
polyharmonic spline bases. Based on these results, we have been able to study
the flat limit of determinantal point processes (DPPs), a type of point process
that is in some sense dual to Gaussian processes
\cite{barthelme2023determinantal}. 

To give some highlights, we show the following: 
\begin{enumerate}
\item GP regression tends in the flat limit to either polynomial regression or
  (polyharmonic) spline regression.
\item Which it is depends on the smoothness
  of the kernel, and on the amount of regularisation enforced by the prior
\item The specific kernel has only a minor influence on the limit, influencing
  only the part of the function space that is most heavily regularised.
\item There is nothing \emph{in theory} that prevents the optimal GP model
  (according to hyperparameter selection criteria) from being arbitrarily close
  to the flat limit. Such solutions are invisible in practice because of
  numerical issues, or because they are obscured by a nugget term.
\item In some cases, we show empirically that the flat limit is a good
  approximation for GP regression \emph{even when the actual kernel is far from flat}.
\end{enumerate}

Gaussian processes are used throughout machine learning and statistics, for many
other tasks than just regression. For instance, they are also used for
classification \cite{williams2006gaussian}, density estimation
\cite{murray2008gaussian}, certain numerical methods
\cite{cockayne2019bayesian}, as emulators in inverse problems
\cite{teckentrup2020convergence} and Bayesian optimisation (starting with
\cite{jones2001taxonomy}), etc. We show in the appendices that our results can
be extended to the general setting of so-called ``latent Gaussian models'', with
non-Gaussian likelihoods and observations that are arbitrary linear functionals.
The
message is the same: a GP prior in a statistical model turns (effectively) in
the flat limit into a polynomial or a spline.

There are some important practical consequences of our results, that we discuss
in detail in section \ref{sec:deg-freedom-flatlim}. In particular, if the data
 presents strong polynomial trends, then caution should be observed when
performing hyperparameter selection. Overall our results argue in favour of
using polyharmonic (eg., thin plate,
\cite{duchon1977splines,meinguet1979multivariate}) splines as priors, at least
for functions in spaces of dimension $\leq 3$. These splines occur as flat
limits of GP models, eliminating a bothersome hyperparameter, and efficient
tools are available (for instance the highly-popular R package \emph{mgcv},
\cite{wood2006generalized}). 

We shall now introduce our results informally, by way of a few pictures. Fig.
\ref{fig:illus-fits} shows a synthetic dataset, fitted using various methods. A
very classical way of fitting such data is to use polynomials, which results in
the curves on panel (a). Another classical way is to use smoothing splines,
which results in panel (c), where the different curves correspond to different
values of the regression parameter. A more modern way of producing a fit is to
use a Gaussian process, which results in a Bayesian version of Radial Basis
Function interpolation \cite{scheuerer2013interpolation}. Gaussian process
regression requires a covariance function, which determines the behaviour of the
fit. The ever-popular Gaussian or ``squared-exponential'' covariance function is
\begin{equation}
  \label{eq:gaussian-kernel}
  \kappa_{\varepsilon}(\vect{x},\vect{y}) = \gamma \exp \left( - \varepsilon^2 \norm{\vect{x}-\vect{y}}^2 \right)
\end{equation}
where $\vect{x}$ and $\vect{y}$ are two points in $\R^d$. In this definition $\gamma$ and $\varepsilon$ are hyperparameters. $\varepsilon$
sets the horizontal scale (the width of the Gaussian), and $\gamma$ sets the
vertical scale (its height). For a fixed value of $\varepsilon$, changing the
value of $\gamma$ produces different fits. Increasing $\gamma$ increases the
amount of variation allowed in the fitted function (the \emph{effective degrees of freedom}), resulting in a tighter fit to the data. Decreasing $\gamma$
decreases the degrees of freedom, as $\gamma \rightarrow 0$ the fit goes to a
horizontal line at 0. Panel (b) shows the fits for a few different values of
$\gamma$, for a fixed value of $\varepsilon=0.5$. Panel (d) is the same, but for
a different covariance function, specifically one in the Matérn family (see eq.~\eqref{eq:matern-kernel} for the definition), which
has the property of being only once differentiable at 0. For reasons that cannot
be succintly explained, differentiability of the kernel function plays a large
role -- see \cite{BarthelmeUsevich:KernelsFlatLimit}. 

\begin{figure}
  \centering
  \includegraphics[width=14cm]{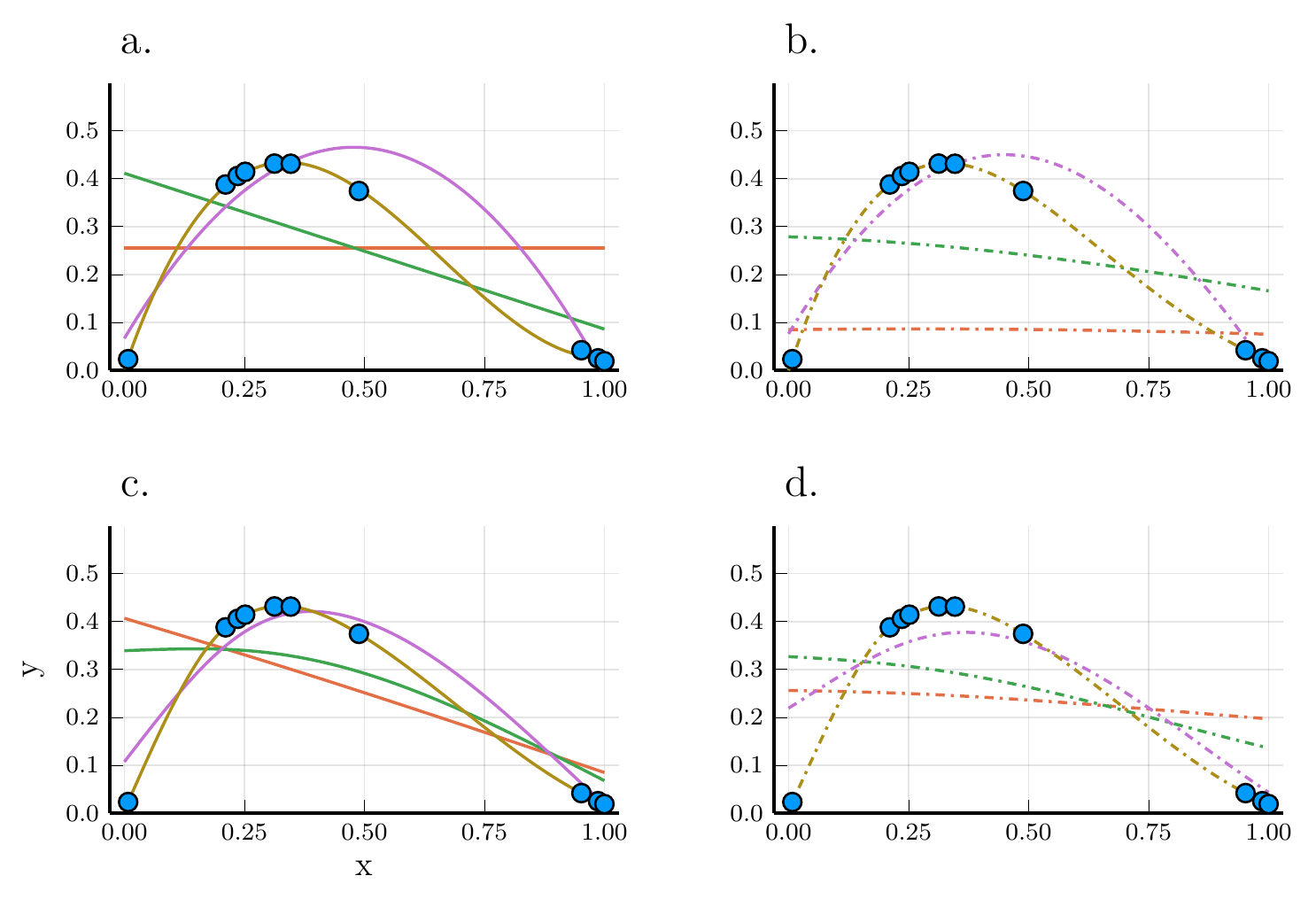}
  \caption{Various fits of the same data. a. Polynomials of degree 0 to 3. b.
    Gaussian process regression, with a Gaussian kernel, $\varepsilon=0.5$ and
    different values of $\gamma$ (see text). c. Quadratic smoothing splines. d.
    GP regression, with a Matérn kernel (see text). 
  }
  \label{fig:illus-fits}
\end{figure}

Our analysis consists in letting $\flatlim$, 
making the covariance functions ``flat'' over the range of the data. One outcome
is that in the flat limit, to put things very roughly, panel (a) $\rightarrow$ panel
(b) and panel(c) $\rightarrow$ panel (d). 

The next set of figures should explain
this a bit better. For a given value of $\varepsilon$ and $\gamma$, the fit
produced by a Gaussian process is a function from (in this case) $[0,1]$ to
$\R$. Call it $\hat{f}_{\varepsilon,\gamma}(x)$. If we leave $\varepsilon$ fixed
and vary $\gamma$, we obtain a family of functions $\mathcal{F}_\epsilon =
\left\{  \hat{f}_{\varepsilon,\gamma} \vert \gamma \in \R^+  \right\}$. Panel
(b) of fig.\ref{fig:illus-fits} shows a few elements from $\mathcal{F}_\epsilon$ for the Gaussian kernel.
A polynomial fit is another
function from $[0,1]$ to $\R$, this time parametrised by the degree of the
polynomial. Call $\hat{p}_r(x)$ the polynomial fit of degree $r$. An implication
of our results (theorem \ref{thm:equivalent-kernels-1d}) is that as $\flatlim$
the set $\mathcal{F}_\epsilon$ intersects (goes through) the polynomial fits.
This is best understood graphically. We cannot plot $\mathcal{F}_\epsilon$, but
we can plot the following: we choose two locations on the x-axis (at $x_a=0.2$ and
$x_b=0.8$), and plot the value of the fit at these locations. For fixed $\varepsilon$, we can think of the pair
$(\hat{f}_{\varepsilon,\gamma}(x_a),\hat{f}_{\varepsilon,\gamma}(x_b))$ as a
parametric curve in $\R^2$, parameterised by $\gamma$. The curves on fig. \ref{fig:pred-curve-gaussian} are
two such parametric curves, for the Gaussian kernel and two different values of
$\varepsilon$. The predictions of the polynomial fits at $x_a$ and $x_b$ are
just a set of points in $\R^2$. What theorem \ref{thm:equivalent-kernels-1d}
implies is that as $\flatlim$, the parametric curves will go through each
polynomial fit (and more than that, interpolate linearly between these points).

For the Matérn kernel, following again theorem \ref{thm:equivalent-kernels-1d},
the comparison should be to another parametric curve, corresponding to the
smoothing splines for all possible values of the regularisation parameter.
Barring very high values of $\gamma$ (for which the behaviour of the GP becomes
polynomial), the GP fit should behave like a spline and therefore tend to the
parametric curve produced by the smoothing splines. That is exactly the
behaviour observed on fig. \ref{fig:pred-curve-matern}. 

Theorem \ref{thm:equivalent-kernels-1d} is actually a bit more informative than
what we have just shown, since it deals for instance with the predictive
variance as $\flatlim$. Theorem \ref{thm:equivalent-kernels-nd} generalises the
result to fits in $\R^d$ with $d>1$. It is stated abstractly in terms of
asymptotically equivalent models, but figs. \ref{fig:pred-curve-gaussian} and
\ref{fig:pred-curve-matern} are useful to keep in mind to visualise what happens
to GP fits as $\flatlim$. 

\begin{figure}
  \centering
  \includegraphics[width=14cm]{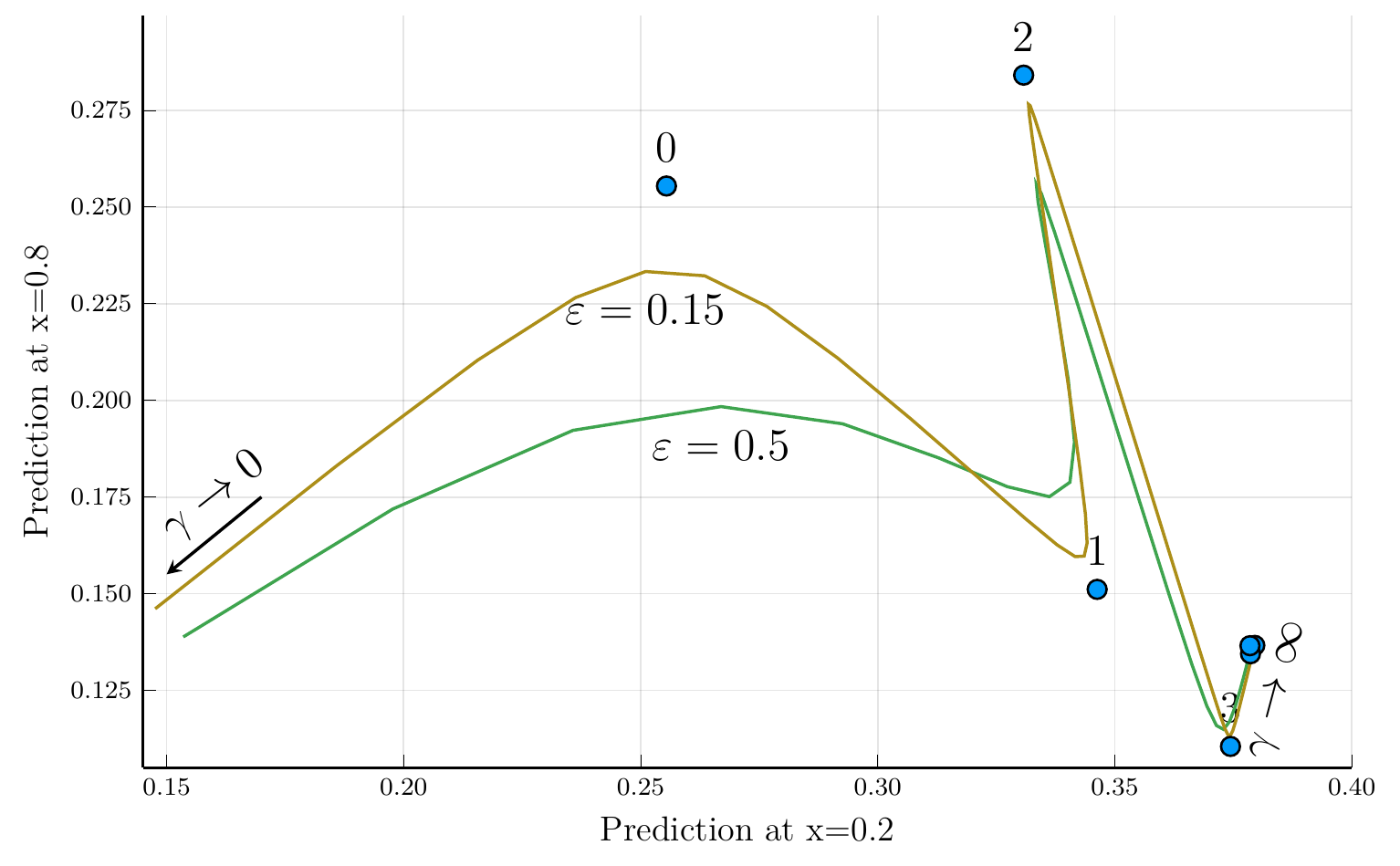}
  \caption{Predictions at $x=0.2$ and $x=0.8$ from different models. The
    numbered dots correspond to the polynomial fits with degree $0$ to $6$. The
    continuous curves are the predictions from GP regression with Gaussian
    kernel, fixing $\varepsilon$ but letting $\gamma$ vary. We show the two
    limits $\gamma \to 0$ and $\gamma \to \infty$, which correspond respectively
    to a maximally penalised fit and to a minimally constrained one (an interpolation).
    The two individual
    curves are for two different values of $\varepsilon$. Theorem
    \ref{thm:equivalent-kernels-1d} states that as $\flatlim$ the continuous
    curve goes through the blue dots in a piecewise linear manner. 
  }
  \label{fig:pred-curve-gaussian}
\end{figure}

\begin{figure}
  \centering
  \includegraphics[width=14cm]{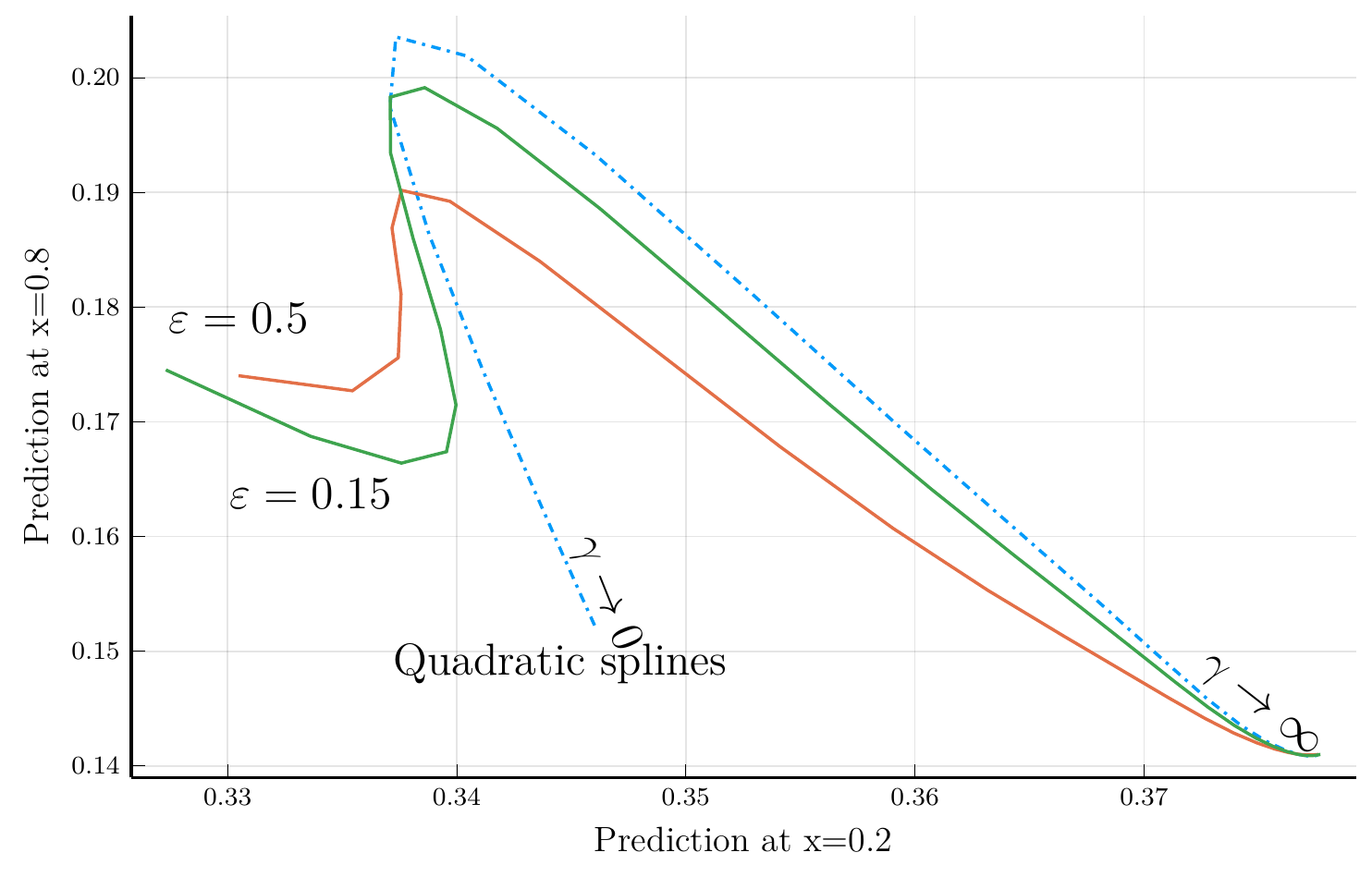}
  \caption{Predictions at $x=0.2$ and $x=0.8$ from different models. The
    dashed line corresponds to smoothing spline fits for all possible values of
    the regularisation parameter. The
    continuous curves are the predictions from GP regression with a once-differentiable Matérn
    kernel, fixing $\varepsilon$ but letting $\gamma$ vary. The two individual
    curves are for two different values of $\varepsilon$. Theorem
    \ref{thm:equivalent-kernels-1d} states that as $\varepsilon$ the continuous
    curve tends to the dashed curve. 
  }
  \label{fig:pred-curve-matern}
\end{figure}

\subsubsection*{Organisation of the paper}

Section \ref{sec:background} introduces GP regression, the main formulas and
notation. Taking GPs to the flat limit produces improper, \emph{semi-parametric}
GP models which have a penalised and an unpenalised part. Section
\ref{sec:semi-parametric} introduces some notation and useful facts on
semi-parametric models for GP regression. Section
\ref{sec:prediction-equivalence} sets the main theoretical framework, where we
develop an equivalence relation on semi-parametric models from the point of view
of prediction-equivalence. Roughly, two models are prediction-equivalent if they lead to the same predictive
distributions. Subsequent results are stated in terms of prediction equivalence.

Section \ref{sec:results} contains our core results on limits of GP regressions in the univariate case,
which are generalised in section \ref{sec:multivariate} to the multivariate case. Section \ref{sec:deg-freedom-flatlim}
contains additional results on hyperparameter selection and degrees of freedom.
Section \ref{sec:towards-practical} shows how to construct \emph{matched flat-limit approximations} to GP
regressions, and includes numerical results. The discussion in section
\ref{sec:conclusion} highlights some limitations and directions for future work. 
The appendices contain some deferred proofs, various results on Wronskian
matrices, and the outline of two proofs extending our results to general linear
observations and non-Gaussian likelihoods.

\section{Background on GP regression and related methods}
\label{sec:background}

Gaussian processes are used in a variety of statistical models, but the simplest
and most elegant is GP regression, also known as kriging. GP regression is a Bayesian procedure for
non-parametric regression, in which we assume that a function $f(x)$ has been
measured (with noise) at $n$ locations $x_1,\ldots,x_n$ and the goal is to infer
$f$ given these measurements and some vague prior knowledge, for instance that
$f$ is smooth. The procedure is called non-parametric because we do not assume
that $f$ has some parametric form. Instead, a Gaussian process prior is used to
capture some basic prior knowledge about $f$, for instance its smoothness or its
periodicity. 

A Gaussian process is a random function $f(x) : \Omega \rightarrow \R$ that has multivariate Gaussian
marginals. For simplicity we take $\Omega \subset \R$ in this introduction, but
higher dimensions are dealt with further down. We note $f \sim GP(\mu,\kappa)$,
where $\mu(x)  : \Omega \rightarrow \R$ is a mean function and $\kappa(x,y)  : \Omega^2 \rightarrow \R$ is a covariance function, if
for all finite subsets $\X = \{x_1 \ldots x_n \}$ of $\Omega$ , the random
vector $\vect{f}_\X = \left[ f(x_1) \ldots f(x_n) \right]^\top$ has a multivariate Gaussian
distribution, specifically $\vect{f}_\X \sim N(\vect{\mu}_\X,\matr{K}_\X)$, where
$\mu_\X= \left[ \mu(x_1) \ldots \mu(x_n) \right]^\top$ and $\matr{K}_\X =
[\kappa(x_i,x_j)]_{i,j=1}^n$. In most cases, the covariance $\kappa(x,y)$ is a
decreasing function of the distance between $x$ and $y$, which ensures that
$f(x_i)$ and $f(x_j)$ have similar values if $x_i$ is close to $x_j$ and so that
the random function is smooth (the smoothness order of $f$ is in fact a function of
the smoothness order of $\kappa$).
In GP regression the GP plays the role of a prior. The assumption is that $f \sim GP(0,\kappa)$
and in addition that the measurements are i.i.d. and Gaussian, namely
\begin{equation}
  \label{eq:likelihood}
  y_i \sim N(f(x_i),\sigma^2)    
\end{equation}
for $i$ in $1, \ldots, n$\footnote{A different derivation of the same estimator
  is used in the kriging literature, via minimum variance arguments \cite{scheuerer2013interpolation}. Note that
  here we are just stating the usual assumptions used to derive classical GP
  regression. In our analysis we make no assumptions on the true function $f$ or
  the distribution of the measurements $\by$. We describe what happens to the GP
  estimator for a given, fixed dataset.   }. Then the posterior distribution $f
| \vect{y}$ is also a Gaussian process. This can be verified by writing the
joint distribution of $\vect{y}$ and $\vect{f}_{\X'}$ for any finite set $\X' \subset \Omega$. By hypothesis,
that joint distribution is a multivariate Gaussian, and so the (posterior) conditional
$\vect{f}_{\X'}|\vect{y}$ is Gaussian as well. The posterior mean and covariance functions
can be easily derived by applying the usual Gaussian conditioning formulas, and
read:
\begin{equation}
  \label{eq:posterior-mean}
  \E(f(x)|\vect{y}) = \vect{k}_{x,\X}\left( \matr{K}_\X + \sigma^2 \matr{I} \right)^{-1} \by 
\end{equation}
with $\vect{k}_{x,\X}= [\kappa(x,x_1) \ldots \kappa(x,x_n)]$, and
\begin{equation}
  \label{eq:posterior-mean}
  \Cov(f(x),f(x') | \vect{y}) = \kappa(x,x') - \vect{k}_{x,\X}\left( \matr{K}_\X + \sigma^2 \matr{I}
  \right)^{-1}\vect{k}_{\X,x'}
\end{equation}

The posterior expectation, $\ft(x) = \E(f(x)|\vect{y})$ is naturally used as an estimator for
$f(x)$. Two remarks are in order:
\begin{enumerate}
\item when the hyperparameters are fixed, $\ft$ is a linear function of
  $\vect{y}$, which makes GP regression a member of the family of \emph{linear
    smoothers} studied by, {\it e.g.}, \cite{buja1989linear} \footnote{When the
    hyperparameters are themselves selected based on the data, the final fit is
    generally a non-linear function of $\vect{y}$, however.}.
\item $\ft$ can be written as $\ft(x) = \sum
  \kappa(x,x_i) \alpha_i$, so that $\ft$ belongs to the reproducing kernel
  Hilbert space generated by $\kappa$, which relates GP regression to classical
  kernel methods \cite{scholkopf2002learning}
\end{enumerate}
In fact, most of the results given below apply (with appropriate modifications) to related methods like kernel
ridge regression or support vector regression.

\subsection{Covariance functions, and the problem of hyperparameters}
\label{sec:hyperparameters}

So far, we have not defined our covariance function. We shall focus on
radial-basis kernel functions, meaning that $\kappa(\vect{x},\vect{y}) =
\gamma \psi(\norm{\vect{x}-\vect{y}}_2)$ for
some function $\psi$, {\it i.e.} the covariance only depends on the (Euclidean) distance between
$\vect{x}$ and $\vect{y}$. For $\kappa$ to be a valid covariance function, it
needs to be positive definite, and for stationary covariance functions this is equivalent (by
Bochner's theorem) to requiring that $\psi$ be the Fourier transform of a
non-negative measure. 

The prototypical choice in 
machine learning is to use the squared-exponential (also known as Gaussian) covariance
function, eq. \eqref{eq:gaussian-kernel}.
In this formulation $\varepsilon$ acts like an inverse \emph{horizontal} scale
(an inverse bandwidth) while $\gamma$ acts like an inverse \emph{vertical} scale
(a gain parameter). These parameters are usually unknown and must be estimated
from the data, using one of the methods outlined below in section \ref{sec:hyperpar-selec}. In
addition, the noise variance $\sigma^2$ (see eq. \eqref{eq:likelihood}) may not be known either, in which case it needs to
be estimated too, bringing the number of hyperparameters to three: $\varepsilon,\gamma,\sigma^2$.

The goal here is to characterise the flat limit of GP regression, which is the
regime where $\flatlim$. This essentially fixes $\varepsilon$ and leaves only
two hyperparameters to be estimated. We describe later (section \ref{sec:deg-freedom-flatlim}) how to understand this
limit in the context of hyperparameter selection. 

While in machine learning the squared-exponential kernel is the most popular, in spatial
statistics the Matérn class of kernels \cite{stein1999interpolation} is very often preferred. These kernels feature an
additional hyperparameter $\nu \in \R^+$ which determines regularity, and have a somewhat
unwieldy expression (\cite{williams2006gaussian}, p. 83): 
\begin{equation}
  \label{eq:matern-kernel}
  \kappa_{\varepsilon}(\vect{x},\vect{y}) = \gamma \frac{2^{1-\nu}}{\Gamma(\nu)} \left(\sqrt{2\nu}(\varepsilon \norm{\vect{x}-\vect{y}})\right)^{\nu} K_\nu  \left(\sqrt{2\nu}(\varepsilon \norm{\vect{x}-\vect{y}})\right)
\end{equation}
where $K_\nu$ is a modified Bessel function. The expression simplifies when
$2\nu$ is an integer. For instance, with $\nu=1/2$ we obtain the exponential
kernel:
\begin{equation}
  \label{eq:exp-kernel}
  \kappa_{\varepsilon}(\vect{x},\vect{y}) = \gamma \exp \left( -\varepsilon \norm{\vect{x}-\vect{y}} \right)
\end{equation}
The value of $\nu$ determines the regularity of $f$ in the sense that if $f$ is
drawn from a Matérn kernel with parameter $\nu$, $f$ is $\lfloor \nu \rfloor $
times mean-square (m.s.) differentiable. This implies for instance that if $f$ is drawn
from a exponential kernel it is continuous but nowhere differentiable. 

More generally, and beyond Matérn kernels, the regularity of $f$ is determined
by the differentiability of the covariance $\kappa(x,y)$ in both variables.
Heuristically, it is easy to see that $f$ is m.s. differentiable if and
only if $\kappa(x,y)$ is differentiable in both $x$ and $y$. M.s.
differentiability is equivalent to requiring that
$\frac{f(x+\delta)-f(x)}{\delta}$ has finite variance in the limit $\delta
\rightarrow 0$. By writing the covariance of $f(x+\delta)$ and $f(x)$ one can
check that $\Var(\frac{f(x+\delta)-f(x)}{\delta}) =
\delta^{-2}(\kappa(x+\delta,x+\delta)+\kappa(x,x)- 2 \kappa (x+\delta,x))$ which has a finite limit if and only
if $\kappa(x,y)$ is differentiable in both variables at $y=x$. 
Repeating the argument we see that m.s. differentiability of order $s$ requires
that the kernel be $s$ times differentiable in both variables at $y=x$.

The regularity of the kernel essentially determines its flat limit behaviour. We use the following definition:
\begin{definition}[Regularity parameter]
  We say a kernel $k(\vect{x},\vect{y})$ has regularity parameter $r$ if it is
  $(r-1)$-times differentiable in both $\vect{x}$ and $\vect{y}$ at $\vect{x}=0,\vect{y}=0$, but not $r$-times differentiable.
\end{definition}
\begin{example}
  The exponential kernel (eq. \eqref{eq:exp-kernel}) is not differentiable at
  $y=x$, because the distance function is not.  It therefore has regularity
  parameter $r=1$.
  In dimension $d=1$, one can check directly that the Matérn kernel
  \[ k(x,y) = (1+|x-y|)\exp(|x-y|)\]
  is once-differentiable and therefore has $r=2$. 
  The Gaussian kernel is infinitely differentiable in both variables, and
  therefore has $r=\infty$. 
\end{example}

For stationary kernels, the regularity parameter is easy to determine based on
the power spectral density, see \cite{stein1999interpolation} or appendix
\ref{sec:wronskian-spectral}. 

\subsection{Other linear smoothers}
\label{sec:linear-smoothers}

In this section we introduce other linear smoothers which are related to GP
regression in the flat limit. The first is polynomial regression, which is very
simple when $d=1$ (the multivariate case introduces some complexity, dealt with
in section \ref{sec:multivariate}). Here we assume $f$ is a polynomial of degree $s$, noted $f \in
\PP_s$, {\it i.e.} $f(x) = \sum_{i=0}^s \alpha_i x^i$, with $\vect{\alpha} \in \R^{s+1}$. $f$ is estimated by maximum
likelihood, {\it i.e.} via least-squares:

\begin{equation}
  \label{eq:polyreg}
  \hat{f}_s(.) = \argmin_{f \in \PP_s} \sum (y_i - f(x_i))^2 = \vect{v}_{\leq s}(.)^\top (\bV_{\leq s}^\top \bV_{\leq s})^{-1} \bV_{\leq s}^\top \vect{y}
\end{equation}
where $\vect{v}_{\leq s}(x) = [x^0\  x^1\ \ldots \ x^{s}]^\top $ is the
column vector in  $\R^{s+1}$ that contains the first $s$ order monomials at $x$, and  where the Vandermonde matrix 
$ \bV_{\leq s} = [ \vect{v}_{\leq s}(x_1), \ldots,  \vect{v}_{\leq
  s}(x_n)]^\top$ of $\R^{n \times (s+1)}$ collects these vectors at the
locations $x_1,\ldots,x_n$. This matrix and the space it spans are   fundamental
objects in the paper. Eq. \eqref{eq:polyreg} shows that $\hat{f}_s$ depends
linearly on $\vect{y}$, so that polynomial regression is a linear smoother.

The only hyperparameter in this case is the degree $s$. The polynomial
regression fit can also be thought of as a Bayesian {\it a posteriori}
estimate, specifically $\E(f|\vect{y})$ under a (improper), flat prior over the
coefficients $\vect{\alpha}$. In this case the posterior variance equals:
\begin{equation}
  \label{eq:polyreg-variance}
  \Var(f|\vect{y}) = \sigma^2 \vect{v}_{\leq s}(x)^\top (\bV_{\leq s}^\top \bV_{\leq s})^{-1} \vect{v}_{\leq s}(x).
\end{equation}

A different and very popular family of smoothers are the smoothing splines,
which generalise to the polyharmonic splines when $d>1$. Here $f$ is only assumed to
be $p$-times differentiable, and estimated using penalised maximum likelihood.
The penalty equals the energy of the $p$-th derivative of $f$:
\begin{equation}
  \label{eq:smoothing-splines}
  \hat{f} = \argmin_{f \in C_p(\Omega)} \sum (y_i - f(x_i))^2  + \eta \int_{\Omega} (f^{(p)}(x))^2 dx
\end{equation}
where the optimisation is over $C_p(\Omega)$, the space of $p$-times
differentiable functions on $\Omega$.

An important feature of the $\int_{\Omega} (f^{(p)}(x))^2 dx$ regulariser is
that it has a null space, since any polynomial of degree $(p-1)$ has zero
penalty.

A famous result known as the ``representer theorem'' (\cite{wahba1990spline,scholkopf2002learning}) states that this
variational optimisation problem collapses to a finite-dimensional optimisation
problem: the solution $\hat{f}$ belongs to a finite dimensional space of
functions, the splines of order $p$ with knots at $x_1 \ldots x_n$. The argument
is quite simple \cite{scholkopf2002learning}. Without using the RKHS formalism, it can be sketched as follows: the
error $\sum (y_i - f(x_i))^2$ is indifferent to the values of $f$ outside of the
measurements $\X$, so we need to look for the function that minimises the
penalty given certain values $\fX$. The solution turns out to be a
spline of order $p$ with knots at $\X$, and so the solution of the overall
optimisation problem is just to find the optimal such spline. 

Classical results on splines \cite{duchon1977splines} show that a basis for this space is given by
functions of the form:
\begin{equation}
  \label{eq:spline-space}
  g(x) = \sum_{i=1}^n |x-x_i|^{2p-1}\alpha_i + \sum_{j=0}^{p-1} x^j \beta_j.
\end{equation}
Here $g$ is a sum of a piecewise polynomial term and a polynomial term. Note that the latter spans the null space of the regulariser. This
form has $n+p$ degrees of freedom, but the regularisation term imposes
$\bV_{<p}^\top \vect{\alpha} = 0$, which removes $p$ degrees of freedom. 

We can inject eq. \eqref{eq:spline-space} into the smoothing splines
optimisation problem (eq. \eqref{eq:smoothing-splines}) to turn into a finite
dimensional problem over $\vect{\alpha}$ and $\vect{\beta}$. Some calculus shows
that the problem is equivalent to inverting the following ``saddle-point''
system:
\begin{equation}
  \label{eq:saddlepoint-regularised}
  \begin{pmatrix}
    (-1)^p \bD^{(2p-1)} + \eta\bI & \bV_{< p} \\
    \bV_{< p}^\top & \matr{0} 
  \end{pmatrix}
  \begin{pmatrix}
    \vect{\alpha} \\
    \vect{\beta}
  \end{pmatrix}
  =
  \begin{pmatrix}
    \vect{y} \\
    \vect{0}
  \end{pmatrix}
\end{equation}
where $\bD^{(2p-1)}$ is a symmetric matrix with entries $\bD^{(2p-1)}_{ij}=|x_i-x_j|^{2p-1}$.

Recall that $\eta$ is a regularisation parameter: the smaller $\eta$, the closer
$f$ must fit the data. In the $\eta \rightarrow 0$ limit, regularisation turns
into interpolation, and the system above turns into:
\begin{equation}
  \label{eq:saddlepoint-interp}
  \begin{pmatrix}
    (-1)^p \bD^{(2p-1)} & \bV_{< p} \\
    \bV_{< p}^\top & \matr{0} 
  \end{pmatrix}
  \begin{pmatrix}
    \vect{\alpha} \\
    \vect{\beta}
  \end{pmatrix}
  =
  \begin{pmatrix}
    \vect{y} \\
    \vect{0}
  \end{pmatrix}
\end{equation}
which is the classical system for polyharmonic spline interpolation \cite{BerlTA2004,wendland2004scattered}. On
the other hand, letting $\eta$ go to $\infty$ in the optimisation problem (eq.
\eqref{eq:smoothing-splines}) effectively imposes that the solution belongs to
the null space of the regulariser, {\it i.e.} the space of polynomials of degree
up to $p-1$. In this limit we therefore recover polynomial regression. The same
sort of relationships are present in the flat limit of GP regression. The limit
is sometimes a spline, sometimes a polynomial, sometimes a regression and
sometimes an interpolant. Exactly what happens depends on $n$, on the regularity
of the kernel, and on how much regularisation is applied. 

\subsection{Degrees of freedom of a linear smoother}
\label{sec:deg-freedom}

The notion of \emph{effective degrees of freedom} is important in the analysis of linear
smoothers (see \cite{buja1989linear}). If $\fXh= \bM \vect{y}$, where $\bM$ is
the smoother matrix, the number of effective
degrees of freedom is simply defined as $\Tr \bM$. For instance, if $\bM$ is a
projection (as in the case of polynomial regression), then $\Tr \bM$ is just the
dimension of the space $\vect{y}$ is projected to (the image space). For regularised
regressions, the matrix $\bM$ is not a projection but the eigenvalues are in
$[0,1]$, and summing these eigenvalues, which is what the trace does, is a
natural way of defining a ``dimension'' for the image space. 

In the case of polynomial regression of degree $p$, the number of degrees of
freedom of the smoother simply equals $p+1$, the dimension of the space of
polynomials of degree $p$. In the case of GP smoothers, the number of degrees of freedom equals  (from eq.
\eqref{eq:posterior-mean}):
\begin{equation}
  \label{eq:df-gp}
  \Tr\left(\bK(\bK+\sigma^2\bI)^{-1}\right) = \sum_{i=1}^n \frac{\lambda_i}{\lambda_i + \sigma^2}
\end{equation}
where the $\lambda_i$'s are the eigenvalues of $\bK$. On the
right-hand side, each term in the sum is between 0 and 1. If $\lambda_i$ is much
larger than $\sigma^2$, then the term is close to 1. If $\lambda_i$ is much
smaller than $\sigma^2$, then the term is close to 0. If there is an index $j$
such that $\lambda_j \gg \lambda_{j+1}$, and $\lambda_{j} > \sigma^2 >
\lambda_{j+1}$, then the smoother matrix is close to a projection matrix. Such a
scenario arises in the flat limit. 

The number of  degrees of freedom of (polyharmonic) spline interpolants are a
bit more intricate to work out, because the smoothing matrix does not take a
very convenient form. The result can be obtained either by a brute-force
calculation or by noticing that the problem is the same as computing the
expected size of an extended L-ensemble \cite[Eq. (17)]{tremblay2021extended}, yielding
the following figure:
\begin{equation}
  \label{eq:df-spline}
  \Tr(\bM) = p+\sum_{i=1}^{n-p} \frac{\lambda_i}{\lambda_i+\eta}
\end{equation}
where here the $\lambda_i$'s are the eigenvalues of the matrix $(-1)^r \Qort^\top
\bD^{(2p-1)}\Qort$,  $\Qort$ being an orthonormal basis of the orthogonal of $\mspan  \bV_{< p}$. 
Recall that when $\eta$ goes to 0 regression turns into
interpolation, so that $\bM\vect{y}=\vect{y}$. One can verify from eq.
\eqref{eq:df-spline} that the number of degrees of freedom indeed goes to $n$. In the other
limit, when $\eta \rightarrow \infty$ we perform a polynomial regression of degree
$p-1$, and accordingly the number of degrees of freedom goes to $p$.

\subsection{Hyperparameter selection}
\label{sec:hyperpar-selec}

There are several methods available for hyperparameter selection in GP
regression. The most satisfactory is certainly to avoid hyperparameter selection entirely by
computing the marginal posterior expectation (integrating over the
hyperparameters). This is not tractable analytically and somewhat expensive in
practice, so alternative methods are often preferred. 
Let us set up more appropriate notation. The vector of hyperparameters is  $\vect{\theta}
= (\gamma,\epsilon,\sigma^2)$ if $\sigma^2$ is unknown, and $\vect{\theta}
= (\gamma,\epsilon)$ if $\sigma^2$ is considered known.

A method for hyperparameter selection popularised by
\cite{mackay1999comparison}, but equivalent to a form of Empirical Bayes
\cite{robbins1956empirical}, is to set $\vect{\theta}$ to its maximum-likelihood
value. The probability of the observations $\by$ given the hyperparameters
(marginalising over $\vect{f}$) is:
\begin{equation}
  \label{eq:marginal-likelihood}
  p(\by | \vect{\theta}) = \det \left( 2\pi (\bK_{\vect{\theta}}+\sigma^2\bI) \right)^{-1/2}
  \exp(-\frac{1}{2} \by^\top(\bK_{\vect{\theta}}+\sigma^2\bI)^{-1/2} \by )
\end{equation}
The maximum likelihood estimate of $\vect{\theta}$ is obtained by maximising eq. \eqref{eq:marginal-likelihood}. For our purposes here it is not very
fruitful to compute the asymptotics of eq. \eqref{eq:marginal-likelihood} in the
flat limit because it is divergent, as the prior becomes improper. We therefore
focus our efforts on other selection criteria which are not divergent, as we
show in section \ref{sec:deg-freedom-flatlim}. 

The non-divergent criteria we focus on in this paper are also very popular, and consist of
\begin{enumerate}
	\item \texttt{LOO-MSE}:  leave-one-out cross-validation with a squared-loss
	\item \texttt{LOO-NLL}: leave-one-out cross-validation with a negative log-likelihood
	\item \texttt{SURE}: Stein's Unbiased Risk Estimator. 
\end{enumerate} 
We believe AIC \cite{akaike1974new} and Generalised Cross-Validation \cite{golub1979generalized} should show qualitatively the same behaviour, but we do not study them here. 

For all these three criteria, the smoother matrix plays a central role. Recall that
the smoother matrix $\bMth$ is defined {\it via} the posterior expectation at the
sampled locations which equals :
\begin{equation}
  \label{eq:smoother-matrix-gp}
  \E(\fX | \vect{y}, \vect{\theta}) = \gamma\bK_\varepsilon \left( \gamma \bK_\varepsilon + \sigma^2 \bI \right)^{-1}\vect{y} = \bK_\varepsilon \left(  \bK_\varepsilon + \frac{\sigma^2}{\gamma}\bI \right)^{-1} \vect{y}
  =\bMth \vect{y}
\end{equation}

Cross-validation is a natural way of picking hyperparameters, but one needs to
pick a cost function and a way of splitting the datasets. Leave-one-out (LOO) is
popular with GPs because there are closed-form formulas for two loss functions. 
One is the squared-loss. LOO cross-validation with the
squared-loss reads:
\begin{equation}
  \label{eq:loo-l2}
  C_{\mathrm{loo-mse}}(\vect{\theta})=\frac{1}{n} \sum_{i=1}^n \left(y_i - \E(f(x_i)| \vect{y}_{-i}, \vect{\theta}) \right)^2
\end{equation}
Here $\E(f(x_i)| \vect{y}_{-i}, \vect{\theta})$ is the posterior expectation of
$f(x_i)$ conditional on all the data except $y_i$.
Standard calculations using the Woodbury lemma show that an alternative
formula for the LOO loss is  \cite{friedman2001elements} : 
\begin{equation}
  \label{eq:loo-l2-reexpressed}
  C_{\mathrm{loo-mse}}(\vect{\theta})=
  \frac{1}{n} \sum_{i=1}^n \left(\frac{y_i - (\bMth \vect{y})_i}{1-\bMth(i,i)} \right)^2
\end{equation}
Evident from this formula is that LOO squared-error loss only depends on the
smoother matrix.
A different choice, one that takes uncertainty into
account, is to use the negative log-likelihood as a cost: 

\begin{equation}
  \label{eq:loo-ll}
  C_{\mathrm{loo-nll}}(\vect{\theta})=-\frac{1}{n} \sum_{i=1}^n \log p\left(y_i \vert
    \vect{y}_{-i},\vect{\theta} \right) =
  \frac{1}{n} \left\{ \sum_{i=1}^n  \frac{1}{2} \log \left(2\pi
      \Var(y_i|\vect{y}_{-i},\vect{\theta}) \right)  + \frac{1}{2} \frac{(y_i- \E(y_i|\vect{y}_{-i},\vect{\theta}))^2}{\Var(y_i|\vect{y}_{-i},\vect{\theta})}  \right\}
\end{equation}
Eq. \eqref{eq:loo-ll} also has an equivalent form that is faster to compute and
involves the smoother matrix, see \cite{williams2006gaussian} (p. 117). 

Finally, another way of selecting a hyperparameter, popular in the signal
processing community, is Stein's Unbiased Risk Estimate (SURE, \cite{li1985stein}), which assumes that $\sigma^2$ is known.
\begin{equation}
  \label{eq:SURE}
  C_{\mathrm{SURE}}(\vect{\theta})= -\sigma^2 + \frac{1}{n} \sum_{i=1}^n \left(y_i - (\bMth \vect{y})_i \right)^2 + \frac{2\sigma^2}{n}\Tr \bMth
\end{equation}
SURE is quite similar to AIC in that it features a loss term corrected by a
measure of model complexity, quantified here by the degrees of freedom of the
smoother matrix $\Tr \bMth$.

\section{Semi-Parametric models }
\label{sec:semi-parametric}

The goal of this section is to introduce some notation to unify linear smoothers
like GP regression, polynomial regression and polyharmonic spline regression; we
shall describe them all as ``semi-parametric'' GP models. Semi-parametric
regression \cite{ruppert2003semiparametric}, known in geostatistics as
``universal kriging'' \cite{matheron1969krigeage}, is not a new concept,
and the results in this section are not novel. However, we introduce some
notation that allows us to describe what happens in the flat limit in a compact
and unified way.

Semi-parametric Gaussian process models assume the unknown function $f(x)$ to be
of the form:
\begin{equation}
  \label{eq:semi-parametric-split}
  f(x) = g(x) + \sum_{i} \alpha_i v_i(x)
\end{equation}
where $g(x)$ is non-parametric and $\V = \{ v_1(x) \ldots v_m(x) \}$ is a set of
basis functions forming the parametric part. $g(x)$ is given a zero-mean Gaussian process
prior, and the prior on the weights $\alpha_i$ is the improper uniform prior
$p(\alpha_i) \propto 1$.
We will see later that these improper priors can be viewed as a GP with infinite
variance along certain directions. Despite the improper prior, the posterior is
well-defined under mild conditions (see below), and the resulting fit has useful
properties. Of course, if no basis functions are included, then the model is a
(non-parametric) GP regression, and if no non-parametric term is included, then
we have a parametric model. 

We use the following notation for describing semi-parametric models: 
\begin{definition}[Semi-parametric model]
  A semi-parametric model (SPM) over $\Omega \subseteq \R^d$ is a tuple
  $\mathcal{M}=\SPM{l}{\V}$, where $l(\bx,\by)$ is a (conditionally)-positive definite
  kernel (see def. \ref{def:unisolvent-cpd}) on
  $\Omega$, and $\V = {v_1(\bx),\ldots,v_m(\bx)}$ is a set of linearly-independent basis functions. 
\end{definition}
\begin{example}
  The following describes a SPM over $\R$: $\M =
  \SPM{\exp(-(x-y)^2)}{\{1,x,x^2\}}$. The non-parametric part is a Gaussian
  kernel, and the parametric part are the basis functions $1,x,x^2$. It
  corresponds to a standard GP regression with a Gaussian kernel, except that
  polynomial trends of degree $\leq 2$ are unpenalised.

  A parametric model over $\R$ is the special case where the kernel is uniformly
  0 or (equivalently) missing, {\it e.g.} $\M = \SPM{0}{\{\sin(x),\cos(x)\}}$ is a parametric
  model with two sinusoidal basis functions. A purely non-parametric model has
  $\V = \emptyset$, {\it e.g.} $\M = \SPM{\exp(-(x-y)^2)}{\emptyset}$ is a standard GP
  model with Gaussian covariance. 
\end{example}

A ``conditionally positive-definite'' kernel is a kernel that is only positive
definite on a subspace, as we explain below. The possibility for the
non-parametric kernel to be conditionally positive definite rather than positive
definite is probably non-obvious to the reader. An example where conditionally
positive-definite kernels are used is smoothing spline regression with linear
splines. This prior may be cast as $\M = \SPM{-|x-y|}{\{ 1 \}}$, {\it i.e.} with
a single basis function, namely the constant function. The function $l(x,y) =
-|x-y|$ is not positive definite, as can be easily verified. For instance, if we
evaluate the kernel matrix for $l$ at the locations $x_1=0,x_2=1$, we find $\bL
=
\begin{pmatrix}
  0 & -1 \\ -1 & 0
\end{pmatrix}
$, which has eigenvalues equal to $-1$ and $1$, whereas a positive definite
kernel would give non-negative eigenvalues. Nonetheless, the smoothing
splines SPM is well-defined, because, as we shall explain soon, kernels only
need to be positive definite along the directions orthogonal to the span of the
basis functions.

\begin{definition}[unisolvent sets, conditional positive-definiteness]
  \label{def:unisolvent-cpd}
  A set of locations $\X = \{\bx_1, \ldots, \bx_n\} \subset \Omega$ is said to be unisolvent for the SPM $\M =
  \SPM{l}{ \{ v_1, \ldots, v_m \}}$ if the matrix
  \begin{equation}
    \label{eq:matrix-basis-functions}
    \bV = [v_j(\bx_i)]_{i=1,j=1}^{n, m}
  \end{equation}
  has rank $m$. $\bV$ corresponds to the evaluation of the $m$ basis functions
  (along columns) at the points in $\X$ (along rows). We call it the basis
  matrix. It has a QR decomposition $\bV = \bQ\bR$, where $\bQ$ is an
  orthonormal basis for $\mspan \bV$. 

  The kernel matrix equals
  \begin{equation}
    \label{eq:L-kernel}
    \bL = [l(\bx_i,\bx_j)]_{i=1,j=1}^{n, n}
  \end{equation}
  and we use ``kernel matrix'' rather than covariance matrix because $\bL$ may
  not be positive definite.
  
  The condition that the kernel $l$ be \emph{conditionally positive-definite}
  (see \cite{wendland2004scattered}, ch. 8) with respect to $\V$ corresponds to the requirement that for all unisolvent
  $\X$, the matrix
  \[ \tbL = (\bI - \bQ\bQ^\top)\bL(\bI - \bQ\bQ^\top)\]
  be positive definite. Note that $\bI-\bQ\bQ^\top$ is a projector on the space
  orthogonal to $\mspan \bV$; the requirement is therefore that $\bL$ be
  positive semi-definite on the space orthogonal to $\mspan \bV$. 
\end{definition}
\begin{example}
  We return to $\M = \SPM{-|x-y|}{\{ 1 \}}$, linear smoothing splines, on the set $\X
  = \{ 0, 1 \}$. Here $\bV =
  \begin{pmatrix}
    1 \\ 1
  \end{pmatrix}
  $, which has (trivially) full column rank, so that $\X$ is unisolvent. The
  orthonormal form of the basis matrix is
  $\bQ =   \begin{pmatrix}
    1 \\ 1
  \end{pmatrix} / \sqrt{2}
  $. As before $\bL =
  \begin{pmatrix}
    0 & -1 \\ -1 & 0
  \end{pmatrix}
  $, and
  $\tbL =
  \begin{pmatrix}
    1/2 & -1/2 \\ -1/2 & 1/2
  \end{pmatrix}
   =   \begin{pmatrix}
    \ostwo \\ -\ostwo
  \end{pmatrix}\begin{pmatrix}
    \ostwo & -\ostwo
  \end{pmatrix}
  $, so that $\tbL$ is indeed positive semi-definite. 
  More generally, since in this case there is only one basis function,
  all sets $\X \in \R$ of size at least 1 are unisolvent. The fact that $l(x,y)
  = -|x-y|$ is conditionally positive-definite w.r.t. the constant 
  function is shown in \cite{micchelli1986interpolation}. 
\end{example}

The requirement that the measurement locations $\X$ be unisolvent is necessary
when using a SPM, because it essentially states that the basis functions need to
be identifiable from $\X$, or equivalently that the posterior distribution be
proper, despite the improper prior on the parametric part. We do not wish to linger too much on unisolvent
sets, except to note that there are non-trivial sets in $\R^d$ that are not
unisolvent w.r.t. polynomial basis functions. For instance, if $\V =
\{1,x_1,x_2,x_1x_2,x_1^2,x_2^2 \}$ in
$\R^2$ (where $x_1$ and $x_2$ represent the first two coordinates), then
choosing a point set where $x_1= x_2$ will lead to trouble, since the matrix
$\bV$ will have linearly-dependent columns. There are more surprising examples
to be found, but these examples are all algebraic sets, and thus occur with
probability 0 when sampling $\X$ independently from $\Omega$
\cite{sauer2006polynomial}. 

A useful way of thinking about SPMs is to view them as limits of standard
non-parametric GP models when the prior covariance along the span of the basis
functions goes to infinity (so that they become unpenalised). The details can be
found in appendix \ref{sec:semi-param-models-as-limits}, but the gist is that we may define a
family of kernels indexed by $\varepsilon$,
\[ k_\varepsilon(\bx,\by) = l(\bx,\by) + \frac{1}{\varepsilon} \sum_{i=1}^m v_i(\bx)v_i(\by)\]
which represents the prior covariance of a model
\[ f_\varepsilon(x) = g(x) + \sum_{i} \alpha_i v_i(x) \]
where $g(x)$ is a GP with covariance $l$ and $\alpha_1 \ldots \alpha_n$ are
sampled i.i.d. from a $\N(0,\varepsilon^{-1}) $ Gaussian. We let $\flatlim$ to make
the parametric part unpenalised, and appendix \ref{sec:semi-param-models-as-limits} shows that
\[ \lim_{\flatlim} \SPM{k_\varepsilon}{\emptyset} = \SPM{l}{\V}\]
{\it i.e.} the non-parametric model becomes semi-parametric in the limit. 
This explains why we are allowed to build ``improper'' models based on
conditionally-positive definite $l$: if the directions of\footnote{These directions correspond  to the space spanned by the eigenvectors associated with negative eigenvalues of $\bL$.} ``negative variance" are all along
the span of the basis functions, then they will be swamped for $\varepsilon$
large enough by the
$\frac{1}{\varepsilon}$ positive-definite term. A tidier construction would use
``intrinsic'' priors, as in \cite{rue2005gaussian}. 

We end this section with some concrete formulas for inference in SPMs, derived
from the limit viewpoint in appendix \ref{sec:semi-param-models-as-limits}.
These formulas can be used for implementation.

\begin{proposition}
  The conditional expectation in a SPM $\SPM{l}{\V}$ has the following form:
  \begin{equation}
    \label{eq:conditional-exp-semi-parametric}
    \E(f(x) | \vect{y} ) =
    \begin{pmatrix} \vect{l}_{x,\X} & \vect{v}_x \end{pmatrix}
    \saddle{\bL+\sigma^2\bI}{\bV}^{-1}
    \begin{pmatrix} \vect{y} \\ \vect{0} \end{pmatrix}
  \end{equation}
  where $\vect{l}_{x,\X} =
  \begin{bmatrix}
    l(x,x_1) & \ldots & l(x,x_n)
  \end{bmatrix}
  $ and $\vect{v}_x =
  \begin{bmatrix}
    v_1(x) & \ldots & v_m(x)
  \end{bmatrix}
  $.
\end{proposition}
\begin{proof}
  In appendix \ref{sec:semi-param-models-as-limits}. 
\end{proof}
One may check that this generalises the non-parametric case by removing the
basis functions, and the parametric case by setting $l = 0$ (which implies $\bL
= \matr{0}$ and $\vect{l}_{x,\X} = \vect{0}$). In addition, we can see in eq.
\eqref{eq:conditional-exp-semi-parametric} that a SPM results in a linear smoother (in the sense that the fit is a linear
function of $\by$), and that the fit takes the form
\[ \hat{f}(\bx) = \sum_{i=1}^n \alpha_i l(\bx,\bx_i)  + \sum_{j=1}^m \beta_j v_j(\bx)\]
where the coefficients $\vect{\alpha}$ and $\vect{\beta}$ depend on $\by$ and
$\sigma^2$. Although the fit seems to take its value in a $n+m$-dimensional
space of functions, eq. \eqref{eq:conditional-exp-semi-parametric} actually
implies the condition $\bV^\top\vect{\alpha} = 0$ which removes $m$ degrees of
freedom. This generalises the case of smoothing splines introduced earlier
((eq.see \eqref{eq:spline-space}).

An expression for the smoother matrix can be derived either by taking $\bx =
\bx_i$ in eq. \eqref{eq:conditional-exp-semi-parametric}, or by using the
correspondence between SPMs and extended-L-ensembles
\cite{tremblay2021extended,fanuel2020determinantal} to get
\begin{equation}
  \label{eq:smoother-spm}
  \bM = \bQ\bQ^\top+ \tbL (\tbL + \sigma^2 \bI)^{-1}
\end{equation}
where $\bQ$ is an orthonormal basis for $\bV$ and   $\tbL = (\bI -
\bQ\bQ^\top)\bL(\bI - \bQ\bQ^\top)$. Eq. \eqref{eq:smoother-spm} is the sum of a
projection matrix (as arises in a least-squares fit of a parametric model), and
a regularised fit (as arises in a non-parametric GP model), limited to the
subspace orthogonal to $\bV$. 

The conditional variance takes a similar form to eq.
\eqref{eq:conditional-exp-semi-parametric}, namely:
\begin{proposition}
  The conditional variance in a semi-parametric model $\SPM{l}{\V}$ equals:
  \begin{equation}
    \label{eq:conditional-var-semi-parametric}
    \Var(f(x) | \vect{y} ) = \vect{l}_{x,\X} - 
    \begin{pmatrix} \vect{l}_{x,\X} & \vect{v}_x \end{pmatrix}
    \saddle{\bL+\sigma^2\bI}{\bV}^{-1}
    \begin{pmatrix} \vect{l}_{\X,x} \\ \vect{v}_x^\top \end{pmatrix}
  \end{equation}
\end{proposition}
\begin{proof}
  In appendix \ref{sec:semi-param-models-as-limits}.
\end{proof}
Here again the parametric and non-parametric special cases can be recovered by
setting $l=0$ or removing the basis functions.


\section{Prediction-equivalence of semi-parametric models}
\label{sec:prediction-equivalence}

In the flat limit, standard GP models become equivalent to certain
semi-parametric models (SPMs), in the sense that they give the same predictions
(conditional expectation and conditional variance) regardless of what the value
of $\sigma^2$ is, where the measurements  $\X$  occur and where the prediction is
sought. The aim of this section is to formalise the notion of
predictive-equivalence of SPMs, and to exhibit a simple criterion for proving
equivalence based on the smoother matrix. 

\begin{definition}[Prediction-equivalence for semiparametric models]
  \label{def:pred-equiv-semip}
  Two semi-parametric models $\M = \SPM{l}{\V}$ and $\M' = \SPM{l'}{\V'}$ are said to be prediction-equivalent over a domain
  $\Omega \in \R^d$, noted $\M \sim \M'$, if $|\V|=|\V'|$, and for any finite $\X
  \subset \Omega$ unisolvent for $\V$ and $\V'$, for all $x \in \Omega$, $\by \in \R^{|\X|}$, $\sigma^2 \in \R^+$:
  \begin{enumerate}
  \item The predictive expectations (eq. \eqref{eq:conditional-exp-semi-parametric}) are equal:  $\E_{\M}(f(x) | \by) = \E_{\M'}(f(x) | \by) $
  \item The predictive variances (eq. \eqref{eq:conditional-var-semi-parametric}) are equal :  $\Var_{\M}(f(x) | \by) = \Var_{ \M'}(f(x) | \by) $
  \end{enumerate}
  {\it i.e.}, the predictive distributions are equal.
\end{definition}

In some cases the predictive-equivalence of two models is easy enough to establish. For instance, if $\V$ and
$\V'$ are two sets of basis functions for the same function space, then
whatever $l$, $\SPM{l}{\V} \sim \SPM{l}{\V'}$. To take a concrete example, $\V =
\{ x_1, x_2\}$ spans the same space as $\V' = \{ x_1-x_2,x_1+x_2 \}$
and so using one rather than the other changes nothing to the model
(theoretically, if not numerically). If the intuitive argument does not convince, one can also
check equivalence directly via eq. \eqref{eq:conditional-exp-semi-parametric}
and \eqref{eq:conditional-var-semi-parametric}.

A more subtle source of prediction-equivalence is the following: if $v_i \in
\V$, then, for any $\alpha$:
\[ \SPM{l(x,y) + \alpha v_i(x)v_i(y)}{\V} \sim \SPM{l(x,y)}{\V}\]
This form of prediction-equivalence follows directly from the argument outlined
in section \ref{sec:semi-param-models-as-limits}, or again can be checked via eqs.
\eqref{eq:conditional-exp-semi-parametric} and
\eqref{eq:conditional-var-semi-parametric}. By extension, for any set of
coefficients $(\alpha_i)$,
\[ \SPM{l(x,y) + \sum_{i=1}^m \alpha_i v_i(x)v_i(y)}{\V} \sim \SPM{l(x,y)}{\V}\]

In our flat limit computations however, we cannot show equivalence so directly. What we
have access to are smoother matrices, but it turns out that this is enough. The
next lemma is essential for our proofs, and concerns prediction-equivalence of
two \emph{non-parametric} models. 
\begin{lemma}
  \label{lem:equivalence-smoother}
  These two statements are equivalent:
  \begin{enumerate}[a)]
  \item The nonparametric models with covariance $k$ and $k'$ are
    prediction-equivalent on $\Omega$, {\it i.e.} $ \SPM{k}{\emptyset} \sim \SPM{k'}{\emptyset}$.
  \item   For all finite $\X \subset \Omega$, $\sigma^2 \in \R^+$, the smoother matrices $\bM_{\X} = 
    \bK_\X(  \bK_\X
    + \sigma^2 \bI)^{-1}$ and $\bM'_{\X} =  \bK'_\X(  \bK'_\X
    + \sigma^2 \bI)^{-1}$ are equal
  \end{enumerate}
\end{lemma}
\begin{proof}
We prove each implication separately. (a) $\implies$ (b) is straightforward.
Since $ \SPM{k}{\emptyset} \sim \SPM{k'}{\emptyset}$, then $\E_{ k}(x|\by) = \E_{ k'}(x|\by)$ for all $x$,
including $x\in \X$, implying that:
\[ \vect{\delta}_i \bM_{\X} \by = \vect{\delta}_i \bM'_{\X} \by  \]
for all $i \in \{1,\ldots,|\X| \}$ and $\by \in \R^{|\X|}$. This implies equality of $\bM_\X$ and
$\bM'_\X$ ({\it i.e.}, take $\by$ to be any $\vect{\delta}_j$ ). 

(b) $\implies$ (a) is less direct, but essentially the same as the derivation
for the fast formula 
for leave-one-out cross-validation.

The main trick is that the prediction mean
  and  variance at $x$  given observations at $\X$ can be computed from the
  smoother matrix for $\Y = \X \cup x$. Let $|\Y| = n$.
  We begin with the variance. The predictive variance at $x$ for kernel $k$
  equals
  \[ v_x = k_{x,x} - \bk_{x,\X}(\bK_\X+\sigma^2\bI)^{-1} \bk_{\X,x}\]
  which is a Schur complement in the block matrix
  \begin{equation}
    \begin{pmatrix}
      \label{blockmatrixdemo:eq}
      \bK_{\X} + \sigma^2\bI &   \bk_{\X,x} \\
      \bk_{x,\X} & k_{x,x}
    \end{pmatrix} = \bK_{\Y} + \sigma^2 \bI - \sigma^2 \dn \dn^\top
  \end{equation}
  {\it i.e.}, we have
  \[ \frac{1}{v_x} = \dn^\top(\bK_{\Y} + \sigma^2 \bI - \sigma^2 \dn \dn^\top)^{-1}
    \dn\]
  Applying the Woodbury lemma, we have
  \[
    \frac{1}{v_x} = \dn^\top(\bP + \frac{\sigma^2 \bP \dn \dn^\top\bP}{1-\sigma^2\dn^\top\bP \dn})\dn
  \]
  where $\bP = (\bK_\Y + \sigma^2\bI)^{-1} = \sigma^{-2} (\bI - \bM_{\Y})$. Noting $c
  = \dn^\top\bM_{\Y} \dn$ and simplifying, we obtain:
  \[ v_x = \sigma^2\frac{c}{1-c}\]
  By equality of the smoother matrices $\bM_{\Y} = \bM'_{\Y}$ for any ${\Y}$, we have $c'=c$ and
  $v_x'=v_x$, thus the predictive variances are equal. 
  For the predictive means, one can repeat a similar computation with the
  following formula:
  \[ \E(f(x) | \by) = \sigma^{-2} \dn^\top(\bK_{\Y}^{-1} + \sigma^{-2}\bI - \sigma^{-2} \dn
    \dn^\top)^{-1}
    \begin{pmatrix}
      \by \\
      0
    \end{pmatrix}
  \]
  which is obtained from the conditional posterior over $\vect{f}_\X,f(x)$ given
  $\vect{y}$. Applying the Woodbury lemma again, we see that the expectation
  depends only on the smoother matrix for $\Y$. The calculation is equivalent to proving formula
  (5.26) in \cite{friedman2001elements}.

  An alternative way of proving the same result uses the block inverse formula for the matrix of the l.h.s in equation  (\ref{blockmatrixdemo:eq}), since we have
  \[
    -\frac{\E(f(x) | \by) }{v_x} = \dn^\top (\bK_{\Y} + \sigma^2 \bI - \sigma^2 \dn \dn^\top)^{-1} \begin{pmatrix}
      \by \\
      0
    \end{pmatrix}
  \]
  which leads using the previous calculations to 
  \[
    \E(f(x) | \by) = \frac{\dn^\top \bM_\Y 
      \begin{pmatrix}
        \by \\
        0
      \end{pmatrix}}{1-c}
  \]
  which shows again that the {\it a posteriori} mean at $x$ only depends on the full smoother matrix $\bM_\Y$.
\end{proof}

With the above lemma in hand, extension to semi-parametric models is
straightforward:
As with standard kernels, predictive equivalence can be assessed from equality
of smoother matrices:
\begin{proposition}
  \label{prop:pred-eq-smoother-semipar}
  These two statements are equivalent: letting $\M = \SPM{l}{\V},\M' = \SPM{l'}{\V'}$
  \begin{enumerate}[a)]
  \item $\M \sim \M' $.
  \item   For all finite $\X \subset \Omega$ unisolvent for $\V$ and $\V'$, $\sigma^2 \in \R^+$, the smoother matrices for $\M$ and $\M'$ at $\X$
    are equal.
     \end{enumerate}
\end{proposition}
\begin{proof}
  
  The proof is a variant of lemma \ref{lem:equivalence-smoother}. (a)
  $\implies$ (b) follows from the same argument.

  For (b) $\implies$ (a), we use the characterisation of semi-parametric models
  as limits. $k_\varepsilon(x,y) = l(x,y) + \varepsilon^{-1} \sum_{v \in \V} v(x)v(y)$
  converges in $\flatlim$ to the semi-parametric model $S$, and similarly
  $k'_\varepsilon(x,y) =  l'(x,y) + \varepsilon^{-1} \sum_{v \in \V'} v'(x)v'(y)$
  goes to $S'$, in the sense that the predictive means and variances converge to
  that of $S$ and $S'$. By corollary \ref{cor:limit-smoother-matrix-semipar}, we
  know that the smoother matrix if $k_\varepsilon$ equals $\bM_\varepsilon = \bM_0 +
  \O(\varepsilon)$, its counterpart $\bM'_\varepsilon = \bM'_0 +
  \O(\varepsilon)$ and, by assumption, since the smoother matrices for $S$ and
  $S'$ are equal then $\bM_0 = \bM'_0$. From the proof of lemma
  \ref{lem:equivalence-smoother}, we know that the predictive means and
  variances for  $k_\varepsilon$ and $k'_\varepsilon$ are continuous functions
  of $\bM(\varepsilon)$ and $\bM'(\varepsilon)$, and therefore have the same limit as $\flatlim$.
\end{proof}

In practice, as mentioned in section \ref{sec:hyperpar-selec}, a ``vertical
scale'' hyperparameter is present in GP models. For semiparametric models, this
means that we consider the family $\M(\gamma) = \SPM{\gamma l}{\V}$ indexed by $\gamma
 \in \R^+$, and $\gamma$ is set by minimising a hyperparameter selection
 criterion like those in section \ref{sec:hyperpar-selec}. The marginal
 likelihood cannot be used here (unmodified, at least), because the prior is
 improper. This leaves us with cross-validation and SURE \footnote{AIC and
   Generalised Cross Validation would work as well}. An important property of predictive-equivalent models is
 that two equivalent models remain equivalent post-selection: the value of these
 selection criteria are equal for all values of $\sigma^2$ and $\gamma$. 
 \begin{proposition}[Post-selection equivalence]
   \label{prop:post-selec-equiv}
   Let $\M = \SPM{l}{\V} \sim \M' = \SPM{l'}{\V}$ on $\Omega$, and
   consider the families of models $\M(\gamma) = \SPM{\gamma l}{\V}$
   and $\M'(\gamma) = \SPM{\gamma l'}{\V}$.
   Then for any data $\by$ and noise variance $\sigma^2$, the value of the
   selection criteria given by eq. \eqref{eq:loo-l2}, eq. \eqref{eq:loo-ll} and
   eq. \eqref{eq:SURE} are the same for $\M(\gamma)$ and $\M'(\gamma)$.
   Consequently, if $\gamma^\star$ is the optimal value of the criterion for
   $\M$, it equals the optimal value for $\M'$, and the selected models
   $\M(\gamma^\star)$ and $\M'(\gamma^\star)$ are prediction-equivalent. 
 \end{proposition}
 \begin{proof}
   First, it is is easy to check that if $\SPM{l}{\V} \sim \SPM{l'}{\V}$ then
   $\SPM{\gamma l}{\V} \sim \SPM{\gamma l'}{\V}$. 
   The leave-one-out criteria rely on predictive means and variances so the
   result follows directly from the definition of predictive equivalence. For
   the SURE criterion, the result follows because eq. \eqref{eq:SURE} only
   depends on the smoother matrix.
 \end{proof}
 \begin{remark}
   The result can be extended to selection of $\sigma^2$ as well, if $\sigma^2$
   is unknown and selected via eq. \eqref{eq:loo-ll}. 
 \end{remark}

Because the vertical scale hyperparameter $\gamma$ is always present in practice
in a SPM, it is useful to introduce a notion of equivalence ``up to a
constant'':
\begin{definition}
  We say that $\M = \SPM{l}{\V} $ and $\M' = \SPM{l'}{\V'}$ are
  equivalent \emph{up to a constant}, noted $ \M \propto \M'$, if there exists
  $\alpha \in R'$ such that $\SPM{l}{\V} \sim \SPM{\alpha l'}{\V'}$.
\end{definition}
This relaxed form of equivalence also holds post-selection. By prop.
\ref{prop:post-selec-equiv}, if $\gamma^\star$ is the optimal value of $\gamma$
for the family $\M(\varepsilon)$, then $\frac{\gamma^\star}{\alpha}$ is the
optimal value for the family $\M'(\gamma)$, and $\M(\gamma^\star) \sim
\M'(\frac{\gamma^\star}{\alpha})$. 
Therefore, because two models $\M \propto \M'$ define effectively the same family of models up to a
change of scale, we shall use the $\propto$ notation in our results to hide
irrelevant multiplicative factors.

\section{Main result in the univariate case}
\label{sec:results}


We use the notation $a_\varepsilon \fleq \M$ to denote models that become
prediction-equivalent in the flat limit, in the sense that the predictive
distributions of $a_\varepsilon$ converge to that of $\M$. The precise definition
we use is an asymptotic variant of definition \ref{def:pred-equiv-semip}. Here
 $k_\varepsilon$ denotes a family of kernels indexed by a parameter
 $\varepsilon$.

\begin{definition}[Asymptotic prediction-equivalence]
  \label{def:asymptotic-pred-eq}
   $k_\varepsilon$ is said to be asymptotically prediction-equivalent to
  a fixed (semi)-parametric model $\M = \SPM{l}{\V}$ over a domain
  $\Omega \in \R^d$, noted $k_\varepsilon \overset{\flatlim}{\sim} \M$, if $|\V|=|\V'|$, and for any finite $\X
  \subset \Omega$ such that $\bV_\X$ and $\bV'_\X$ have full column rank, for all $x \in \Omega$, $\by \in \R^{|\X|}$, $\sigma^2 \in \R^+$:
  \begin{enumerate}
  \item The predictive expectations are such that:  $\E_{
      k_\varepsilon}(f(x) | \by) = \E_{ \M}(f(x) | \by) + \O(\varepsilon) $
  \item The predictive variances are such that:  $\Var_{ k_\varepsilon}(f(x) |
    \by) = \Var_{ \M}(f(x) | \by) + \O(\varepsilon) $
  \end{enumerate}

  We use the short-hand $k_\varepsilon \fleq \M$ if there exists $\M' \propto
  \M$ such that $k_\varepsilon \overset{\flatlim}{\sim} \M'$.
\end{definition}

We are now ready to state our main result in the one-dimensional case. We look
at models of the form $k_\varepsilon(x,y)\varepsilon^{-p}$, where the vertical
scale hyperparameter ($\gamma$) grows as $\flatlim$. This lets us control the
degrees of freedom of the fit in the flat limit; the higher the value of $p$,
the more degrees of freedom we allow. A more thorough discussion of degrees of
freedom can be found below (section \ref{sec:deg-freedom-flatlim}). What the
equivalent model turns out to be in the flat limit depends on $p$ and $r$, the
regularity of the kernel, and nothing else. 

In the statement of the theorem, $k_\varepsilon$ designates
a family of kernels indexed by a inverse-scale parameter $\varepsilon$, of the
form
\[ k_\varepsilon(x,y)= \kappa(\varepsilon x, \varepsilon y) \]
The required assumptions are:
\begin{enumerate}
\item $\kappa$ is stationary; i.e. there exists $\psi$ such that $\kappa(x,y) =
  \psi\left(|x-y| \right)$
\item $\psi$ is analytic in a neighbourhood of 0. 
\end{enumerate}
The second assumption can be removed, but doing so in general requires handling
non-integer $p$, as we explain in the appendix (section \ref{sec:non-analytic-kernels}).
\begin{theorem}
  \label{thm:equivalent-kernels-1d}
  Let $k_\varepsilon(x,y))$ a family of kernels with
  inverse-scale parameter $\epsilon$, verifying the assumptions above. Let $\kappa$ have regularity parameter $r$, and let $p$ be an integer. 
  Then the following asymptotic equivalence holds:
  \[ k_\varepsilon\varepsilon^{-p} \fleq \SPM{l_p}{\V_p}\]
  where $l_p(x,y)$ and $\V_p$ depend on the interplay between $p$ and $r$. There are four different cases:
  \begin{itemize}
  \item $p<2r-1$ and $p$ is even, \emph{i.e.}, $\exists\, m<r\text{ s.t. } p = 2m$. Then $l(x,y) = x^{m+1}y^{m+1}$ and $\V =
    \{x^0,x^1,\ldots,x^m\}$. This case amounts to penalised polynomial regression.
  \item $p<2r-1$ and $p$ is odd, \emph{i.e.}, $\exists\, m<r-1 \text{ s.t. } p =
    2m+1$. Then $l(x,y) = 0$ and $\V = \{x^0,x^1,\ldots,x^{m+1} \}$. This case
    amounts to unpenalised polynomial regression (eq. \eqref{eq:polyreg}).
\item $p=2r-1$. In this case, $l(x,y) = (-1)^r |x-y|^{2r-1}$ and $\V =
  \{x^0,x^1,\ldots,x^{r-1} \}$, which amounts to smoothing spline regression
  (eq. \eqref{eq:smoothing-splines}). 
  \item $p>2r-1$. This case leads to an interpolant independently of the value
    of  $\sigma^2$. The interpolant is either a spline of degree $r$ 
    or a polynomial (infinite $r$). 
\end{itemize}
\end{theorem}
\begin{proof}
  A full proof is given in appendix \ref{sec:proof-of-univar-result}. The bulk of the proof consists in
  obtaining the limit of the smoother matrix corresponding to model
  $k_\varepsilon\varepsilon^{-p}$ in $\flatlim$. We can work this out from the
  results in \cite{BarthelmeUsevich:KernelsFlatLimit}, which provide the
  asymptotic eigenvalues and eigenvectors of $\bK_\varepsilon$. 
  From this we obtain an expression for the smoother matrix as
  $\bM_\varepsilon=\bM_0+\O(\varepsilon)$ where $\bM_0$ depends on $p$ and $r$.
  The proof is completed by appealing to lemma \ref{lem:equivalence-smoother}
  and proposition \ref{prop:pred-eq-smoother-semipar}, which allows us to
  deduce equivalence of models from equality of smoother matrices. 

  Examining the
  proof of lemma \ref{lem:equivalence-smoother} and proposition
  \ref{prop:pred-eq-smoother-semipar}, we see that the predictive mean and
  variance depend smoothly on $\bM_\varepsilon$ and so we have asymptotic
  predictive equivalence in the sense of definition
  \ref{def:asymptotic-pred-eq}. The different cases are the different cases for
  smoother matrices. 
\end{proof}
\begin{example}
  If $k_\varepsilon(x,y)$ is the exponential kernel, which has $r=1$, then $\varepsilon^{-1}
  k_\varepsilon(x,y)$ is asymptotically equivalent to the model $\SPM{-|x-y|}{1}$, {\it i.e.} the
  parametric part is the constant function and the non-parametric part is the
  kernel $l(x,y)=-|x-y|$. By eq. \eqref{eq:conditional-exp-semi-parametric}, it implies that 
  $\hat{f} = E(f|\by)$ tends to:
  \[ \hat{f}(x) = -\sum_{i=1}^n \beta_i |x-x_i| + \alpha + \O(\varepsilon)\]
  The kernel functions $|x-x_i|$ are piecewise linear and so the fit goes in the
  limit to a linear spline. Recall that $\vect{\beta}$ is constrained: $\bV^\top
  \vect{\beta} = 0$, which here simplifies to  $\sum \beta_i = 0$. This sets the
  boundary  conditions, as one may easily check by looking at derivatives
  outside the range of the data: $\frac{d}{dx}\hat{f}(x) = 0$ if $x$ is to
  the left or right of the observations $\X$, so that the fit has a built-in,
  implicit Neumann boundary condition. 
  Generalising further, if $k_\varepsilon(x,y)$ is a kernel with finite $r$, and
  setting $p=2r-1$, $\varepsilon^p k_\varepsilon(x,y)$ is asymptotically
  equivalent to the model $S =\SPM{(-1)^r |x-y|^{2r-1}}{\{1,x,\ldots,x^{r-1}\}}$,
  which leads to the asymptotic fit:
  \[ \hat{f}(x) = (-1)^r \sum_{i=1}^n \beta_i |x-x_i|^{2r-1} + \sum_{j=0}^{r-1}
    \alpha_j v^j + \O(\varepsilon)\]
  with the constraint $\bV^\top\vect{\beta} = 0$.
\end{example}

The multivariate counterpart of theorem \ref{thm:equivalent-kernels-1d} can be
found in section \ref{sec:multivariate}. The multivariate theorem resembles the
univariate one, but requires quite a bit of notation. Instead of going directly
to multivariate equivalent models, we take a look instead at degrees of freedom and hyperparameter selection in
the flat limit.

\section{Degrees of freedom, hyperparameter selection, and practical consequences}
\label{sec:deg-freedom-flatlim}

In this section we study the behaviour of the degrees of freedom, and the various hyperparameter selection
methods as $\flatlim$. 
All results are applicable to the multivariate case even
though the numerical examples concern the univariate case. 

Since degrees of freedom play such an important role in hyperparameter
selection, we also look at their asymptotics and show that scaling $\gamma$ as
$\gamma_0\varepsilon^{-p}$ for some well-chosen $\gamma_0$ and $p$ keeps the
degrees of freedom constant as $\flatlim$. We stress two practical implications
of these results.

First, for certain datasets, very low values of $\varepsilon$ may be appropriate
or even optimal (in terms of prediction performance). In section
\ref{sec:flat-limit-solutions} we show an empirical example of this phenomenon,
where solutions at very low values of $\varepsilon$ are selected in the
hyperparameter selection procedure, because the data contain a linear trend.

Second, that there \emph{is} a good solution in small $\varepsilon$ may not be
visible to practitioners, because common practice is to use a ``nugget term'' (a
diagonal perturbation to the kernel matrix). As we show in section
\ref{sec:nuggets} the use of a nugget term increases numerical stability, at the
cost of distorting the results of hyperparameter selection.

We would argue that a good alternative for many datasets is to directly use the
flat limit models instead. In particular, splines are well-established, work
well in low dimensions, and benefit from solid implementations (in the R mgcv
package, for instance, \cite{wood2006generalized}).

It turns out that flat-limit models can approximate the resuls of a GP
fit even when $\varepsilon$ is relatively large. In section
\ref{sec:towards-practical} we use the theoretical results on degrees of freedom
to formulate a ``matched approximation'' for a given GP model. The matched
approximation to a kernel $k_\varepsilon$ is the flat limit model with the same
regularity and degrees of freedom for the measurement locations $\X$. An
interpretation of theorem \ref{thm:equivalent-kernels-1d} is that the matched
approximation becomes exact as $\flatlim$. The empirical results we obtain show
that the matched approximation is sometimes very good, especially for Matérn
models. 

Finally, the flat limit is also valid in more complicated models with
non-Gaussian likelihoods, and observations that depend on general linear
functionals of $f$ (rather than just pointwise evaluation). We show such an
example in section \ref{sec:non-gaussian}. 

\subsection{Degrees of freedom and isofreedom curves}
\label{sec:isofreedom}

\begin{figure}
  \centering
  \includegraphics[width=14cm]{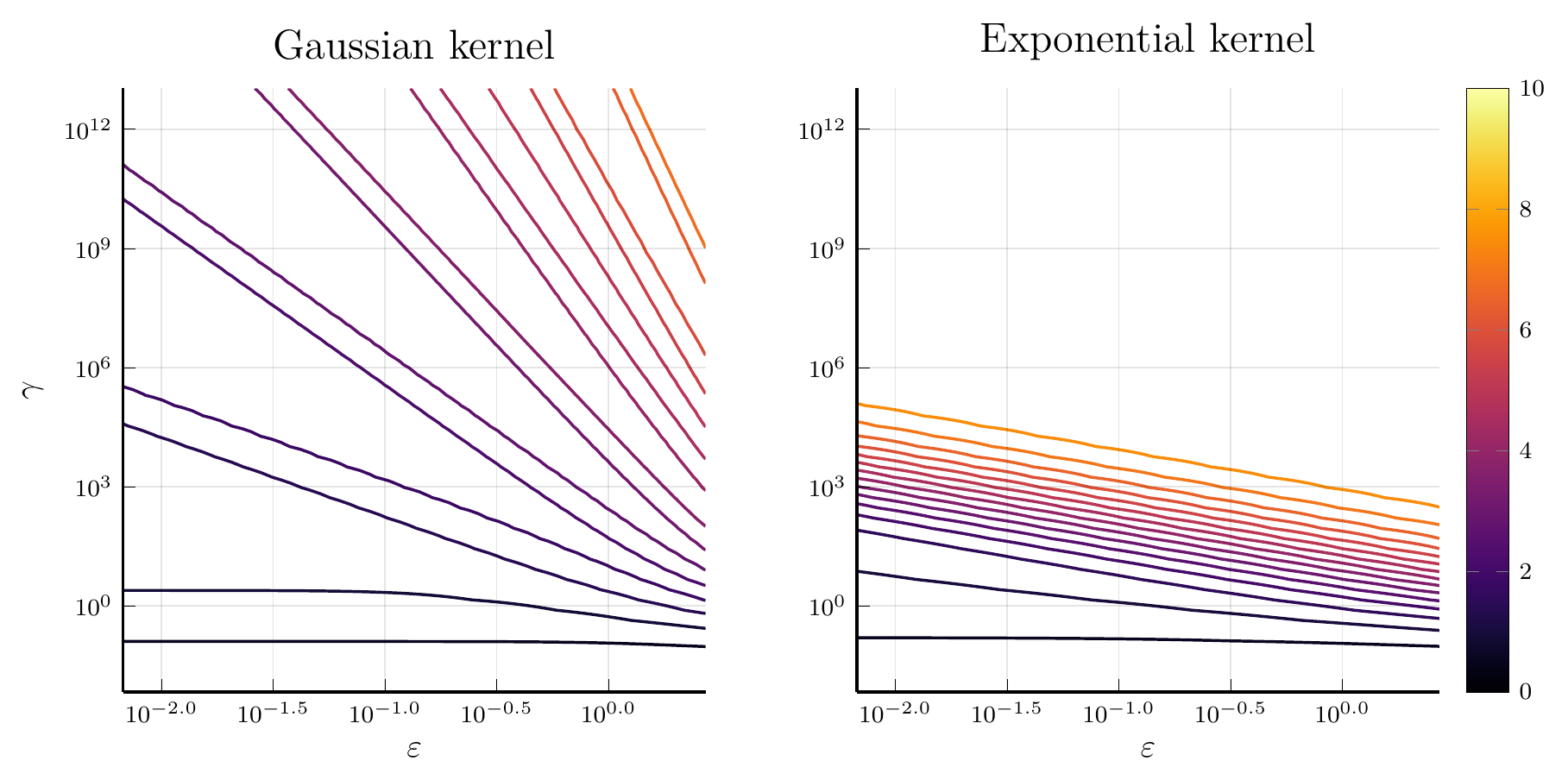}
  \caption{Isofreedom curves (see section \ref{sec:isofreedom}) of a smoother matrix as a function of
    $\varepsilon$ and $\gamma$. The set $\X$ was generated by sampling $n=8$
    points from the unit interval. We show the contours of $\Tr(\gamma\bK_\varepsilon(\gamma\bK_\varepsilon
    +\sigma^2 \bI)^{-1})$ as a function of $\gamma$ and $\varepsilon$ for the
    Gaussian and exponential kernels. Notice that the contours become linear in
    small $\varepsilon$. 
    }
  \label{fig:deg-freedom}
\end{figure}
As explained in section \ref{sec:deg-freedom}, the degrees of freedom of a
linear smoother measure in some sense the ``dimension'' of the range of the
smoother matrix. Given a GP model defined by a kernel $k$ and measurements at
$\X$, the degrees of freedom typically increase with larger $\varepsilon$ and
larger $\gamma$. We show an example in figure \ref{fig:deg-freedom}, where the degrees of freedom
are displayed as a function of $\gamma$ and $\varepsilon$ for a randomly drawn
point set $\X$ in $\R$. 
Because degrees of freedom decrease as $\flatlim$, in section
\ref{sec:results} we let $\gamma$ increase as $\flatlim$ so
that a nontrivial smoother matrix could arise in the limit. The particular form
chosen is $\gamma = \gamma_0\varepsilon^{-p}$, which looks like a choice of
convenience but actually has a deeper motivation. 
What one might notice on figure \ref{fig:deg-freedom}, which has log-log axes, is that the contours
become lines in small $\varepsilon$. We dub these contours ``iso-freedom
curves'', because they correspond to sets of the form
\[ \mathcal{F}_m= \left\{(\varepsilon,\gamma)| \Tr(\gamma\bK_\varepsilon(\gamma\bK_\varepsilon
      +\sigma^2 \bI)^{-1}) = m \right\}\]
for fixed values of $m$.
Given $m$, $\sigma^2$ and $\varepsilon$, we can solve for the value of $\gamma$
such that the degrees of freedom equal $m$. $\gamma$ should verify:
\begin{align*}
  \Tr(\gamma\bK_\varepsilon(\gamma\bK_\varepsilon
  +\sigma^2 \bI)^{-1}) = m \\
  \Leftrightarrow \sum \frac{\gamma\lambda_i(\varepsilon)}{\gamma\lambda_i(\varepsilon)+\sigma^2} - m = 0 \numberthis \label{eq:scaling-function}
\end{align*}
Eq. (\ref{eq:scaling-function}) is a rational equation in $\gamma$, and the
eigenvalues of $\lambda_i(\varepsilon)$ are analytic in $\varepsilon$. Call
$\gamma_m(\varepsilon)$ the solution of eq. (\ref{eq:scaling-function}) as a
function of $\varepsilon$, and note that it is a parametrisation of the
iso-freedom curve, giving $\gamma$ as a function of $\varepsilon$.
The Newton-Puiseux theorem implies that
$\gamma_m(\varepsilon)$ can be expanded as a Puiseux series in small
$\varepsilon$ (see \cite{barthelme2023determinantal}), {\it i.e.} that there exist
$\gamma_0,l \in \ZZ, s \in \ZZ^+$ such that:
\begin{equation}
  \label{eq:puiseux-series}
  \gamma_m(\varepsilon) = \varepsilon^{\frac{l}{s}}(\gamma_0 + \O(\varepsilon^{\frac{1}{s}}) )
\end{equation}
A Puiseux series is just a power series in $\varepsilon^{\frac{1}{s}}$, and if
$s=1$ it is actually a power series. Notice that $\log \gamma_m(\varepsilon)
\approx \frac{l}{s} \log \varepsilon + \log \gamma_0$, which explains why the
iso-freedom curves look linear in log-log coordinates. 
In equation \eqref{eq:puiseux-series},
$l,s$ and $\gamma_0$ depend on $m$ (the desired number of degrees of freedom),
and the kernel function. They can actually be determined in closed-form using
the Newton polygon \cite{moro2002first}, but that would carry us outside the
scope of this manuscript. Among other things, it is not too hard to show that
$s=1$ here, so that the iso-freedom curves have integer slopes in small
$\varepsilon$. 

\subsection{Hyperparameter selection in the flat limit}
\label{sec:hyperpar-selection-flatlim}

In section \ref{sec:hyperpar-selec}, we introduced three hyperparameter
selection methods: the SURE criterion, and two criteria based on leave-one-out
cross-validation. Given the results above, one can verify that all three
criteria are \emph{constant in $\flatlim$ along isofreedom lines}. Figure \ref{fig:hyperpar-sel}
gives a visual illustration of this fact.
\begin{figure}
  \label{fig:illus-hypersel}
  \centering
  \includegraphics[width=14cm]{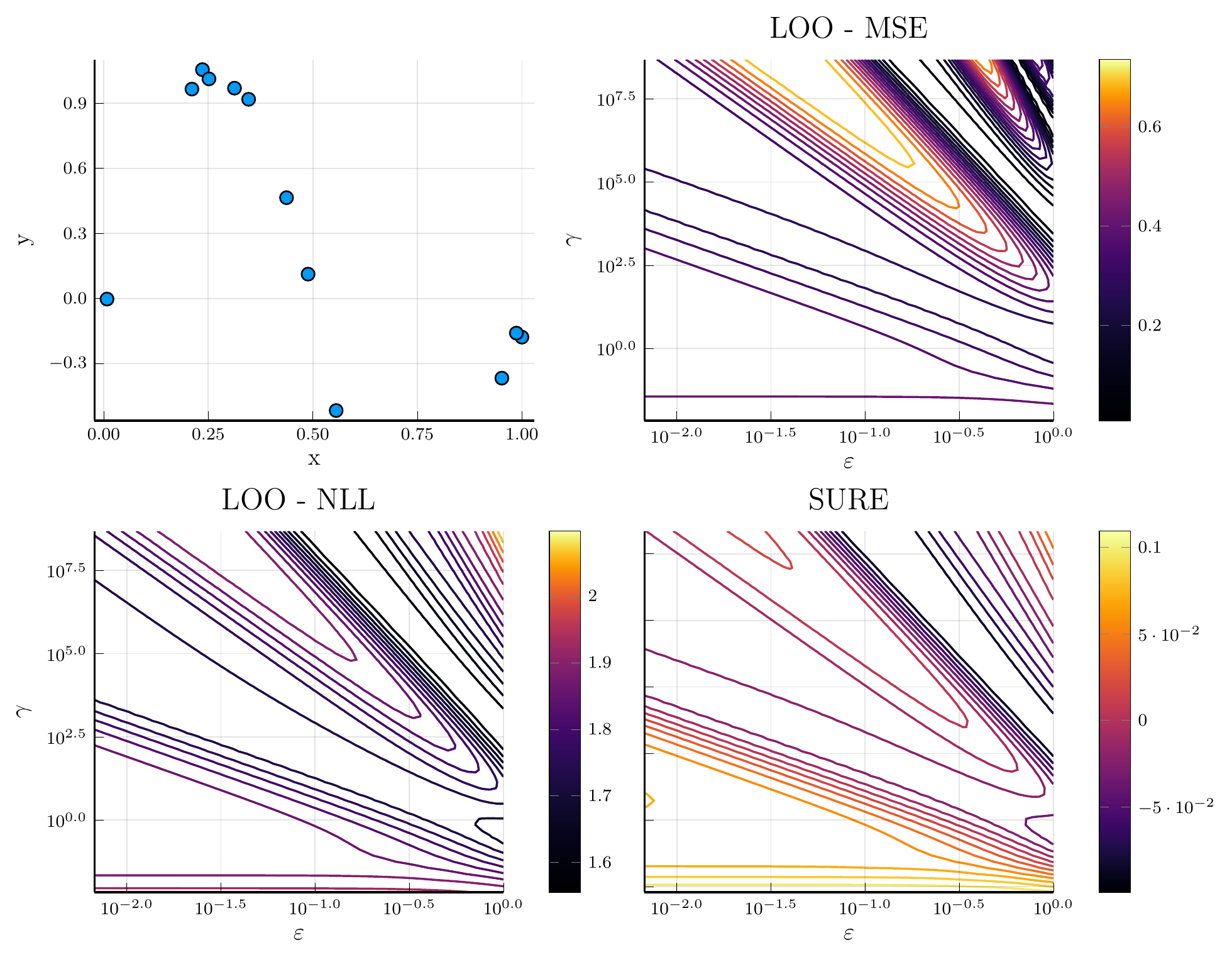}
  \caption{Asymptotics of hyperparameter selection. A synthetic dataset is shown in the upper-left panel. The three other
    panels are contour plots for three different hyperparameter selection criteria.
    We set $ \sigma^2 = (0.1)^2$ and used a Gaussian kernel. Note that these
    three criteria are mostly constant along lines, which have asymptotically
    constant degrees of freedom (see fig. \ref{fig:deg-freedom}). }
  \label{fig:hyperpar-sel}
\end{figure}

To see why the three criteria are asymptotically constant along iso-freedom
lines, consider theorem \ref{thm:equivalent-kernels-1d} and equation
\eqref{eq:puiseux-series} jointly. We shall state the result informally.
Following a contour with constant degrees of
freedom to the limit $\flatlim$, we need to set $\gamma(\varepsilon) =
\varepsilon^{-p}(\gamma_0 + \O(\varepsilon))$ (by eq.
\eqref{eq:puiseux-series}). This is identical in $\flatlim$ to setting
$\gamma(\varepsilon) = \varepsilon^{-p}\gamma_0$ (up to negligible terms), and
we may apply theorem \ref{def:asymptotic-pred-eq}, which tells us that the
predictive mean and variance converge to finite quantities. It is clear from the
formulas of the two leave-one-out criteria (eq. \eqref{eq:loo-l2} and
\eqref{eq:loo-ll}) that they must then converge to finite quantities as well.
The SURE criterion (eq. \eqref{eq:SURE}) must converge as well since the degrees
of freedom are asymptotically constant and the smoother matrix converges. 

\subsection{Flat-limit solutions are sometimes optimal}
\label{sec:flat-limit-solutions}

One implication of the fact that selection criteria do not diverge is that, for
some datasets, the \emph{optimal solutions may be in small $\varepsilon$}. This
is the case if the data contain strong polynomial trends that become unpenalised
in the flat limit. For instance, when using a kernel with regularity order
$r=3$, trends up to quadratic order are unpenalised in the flat limit. If such a
trend is present in the data, and the signal-to-noise ratio is sufficiently low,
then the flat limit solution may be optimal.

Let us offer a concrete example of this phenomenon. In this example the true
latent function is the sum of a sinusoid and a linear trend, specifically:
\[ f(x) = 0.1 \sin(2 \pi x) + x\]
We use a $r=3$ Matérn kernel, and perform hyperparameter selection for
$\varepsilon$ and $\gamma$ the classical
way, using numerical optimisation. We use box constraints to constrain the
search to regions where the matrices can be inverted, but to reflect normal
practice we use standard floating point arithmetic and not arbitrary-precision.
Since hyperparameter-selection criteria are known to have multiple minima, we
use 10 different random initialisations for the optimisation. We show the
results for the $C_{\mathrm{ll}}$ criterion (eq. \eqref{eq:loo-ll}), but similar
results hold for the other two criteria.

The results appear on fig. \ref{fig:linear-trend-high} and \ref{fig:linear-trend-low}. If the noise level is larger than a certain
threshold (depending on $n$), then a single optimum shows up, with a very low
value of $\varepsilon$. It corresponds to fitting just the linear trend. At
intermediate signal-to-noise ratios, two optima are present, one that fits just
the linear trend, and one that tries to follow the sinusoid as well. Finally, at
low noise, only the latter is present. It is noteworthy that in these
simulations \emph{both} minima are at low values of $\varepsilon$, showing that
flat-limit solutions can indeed emerge in practice. 

\begin{figure}
  \centering
  \includegraphics[width=14cm]{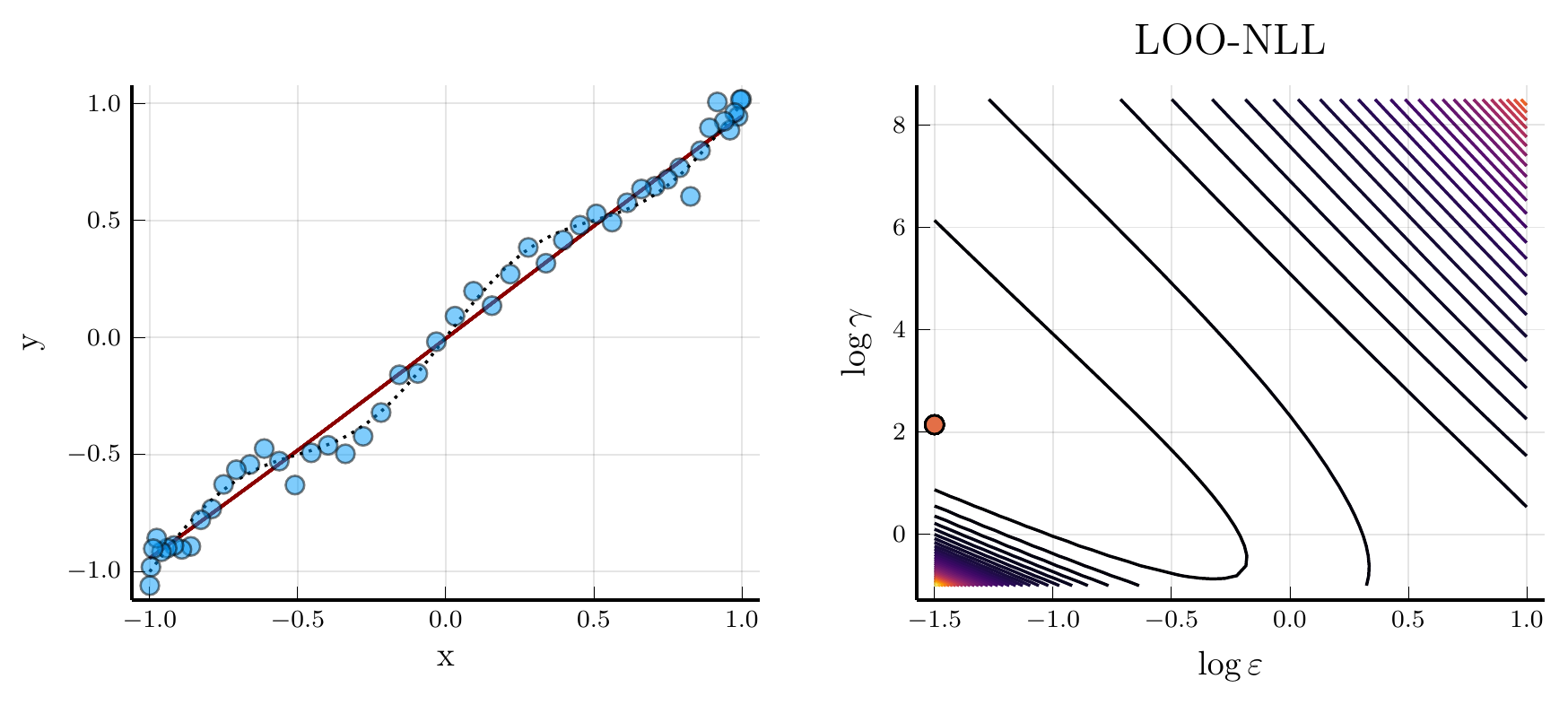}
  \caption{Flat-limit solutions can be optimal in certain scenarios. The data
    are shown on the left in light blue, the true function is the dotted black
    line, and the GP fit corresponding to the selected hyperparameters is in
    red. The contour plot represents the optimisation landscape for criterion
    $C_{\mathrm{ll}}$. The optimum was obtained numerically and is shown as the
    dot. }
  \label{fig:linear-trend-high}
\end{figure}

\begin{figure}
  \centering
  \includegraphics[width=14cm]{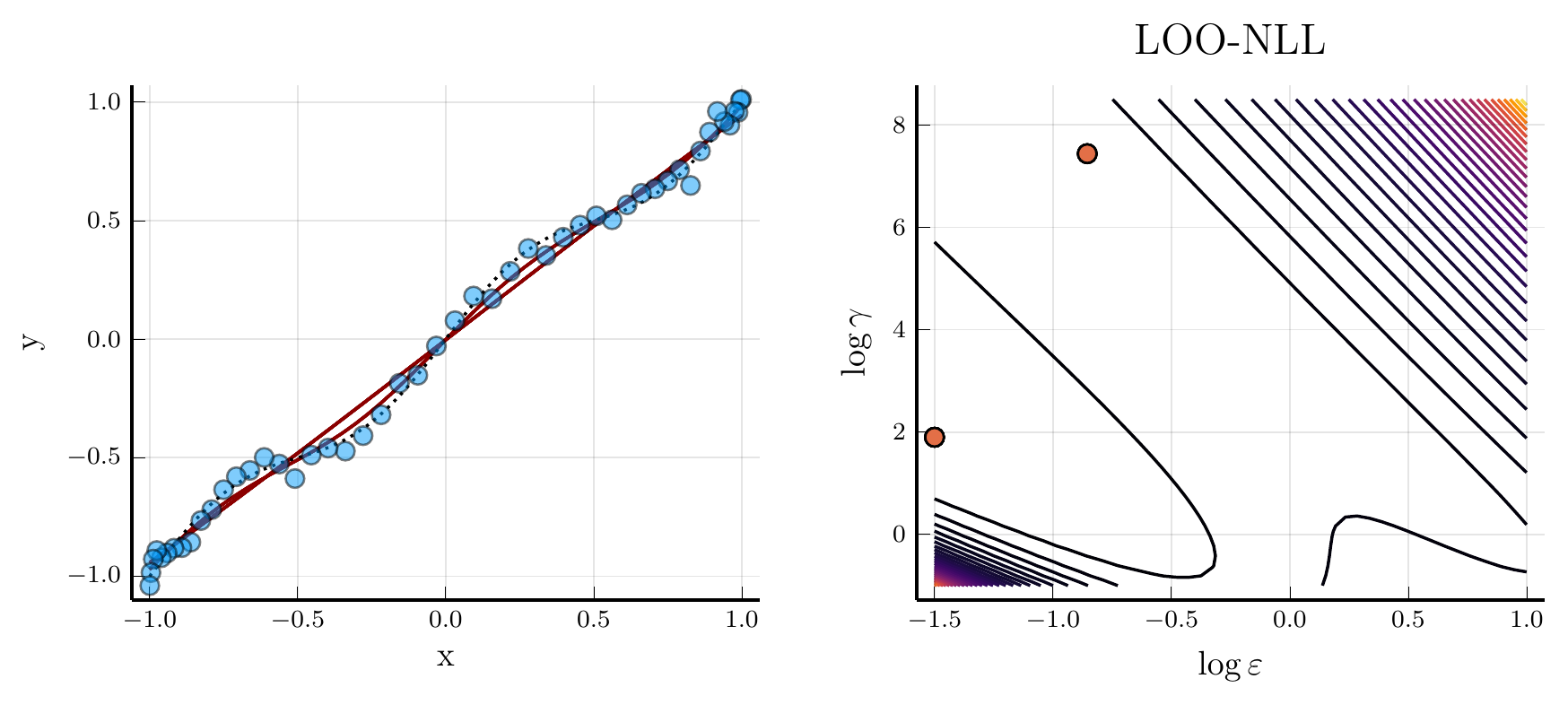}
  \caption{Same as in fig. \ref{fig:linear-trend-high}, but with lower noise
    variance. Here two different optima emerge, corresponding to a linear fit
    for the one, and to a fit that follows the sinusoid for the other.  }
  \label{fig:linear-trend-low}
\end{figure}

\subsection{Why low $\varepsilon$ solutions are frequently invisible in practice}
\label{sec:nuggets}

However, these solutions may be invisible or unattainable when using naïve
numerical methods, especially when the Gaussian kernel is used. The main source of numerical difficulty arises when computing
the smoother matrix:
\[ \bM_\varepsilon =
  \bK_\varepsilon(\bK_\varepsilon+\frac{\sigma^2}{\gamma} \bI)^{-1} \]
Since $\gamma$ becomes very large as $\flatlim$, $\frac{\sigma^2}{\gamma}$ is
small, and one must invert a poorly conditioned matrix. A Cholesky decomposition
in standard floating point precision may fail, so that the small-$\varepsilon$
part of the space is inaccessible. 
In practice sometimes a ``nugget term'' is used to alleviate numerical
difficulties: one replaces  $\bK_\varepsilon$ with $\bK_\varepsilon+\nu\bI$,
where $\nu$ is small. However, once the nugget term is added, increasing
$\gamma$ beyond $\nu$ has no effect. Some useful eigenvectors are made invisible
by the nugget term and this has the effect of ``clipping'' the surface of
hyperparameter selection criteria, as shown in fig. \ref{fig:nugget}. Since this
is clearly undesirable, a better option in the future may be to adapt existing
methods for stable RBF interpolation ({\it e.g.}, \cite{fornberg2011stable,fornberg2008stable}) to GP regression
problems.

\begin{figure}
  \centering
  \includegraphics[width=14cm]{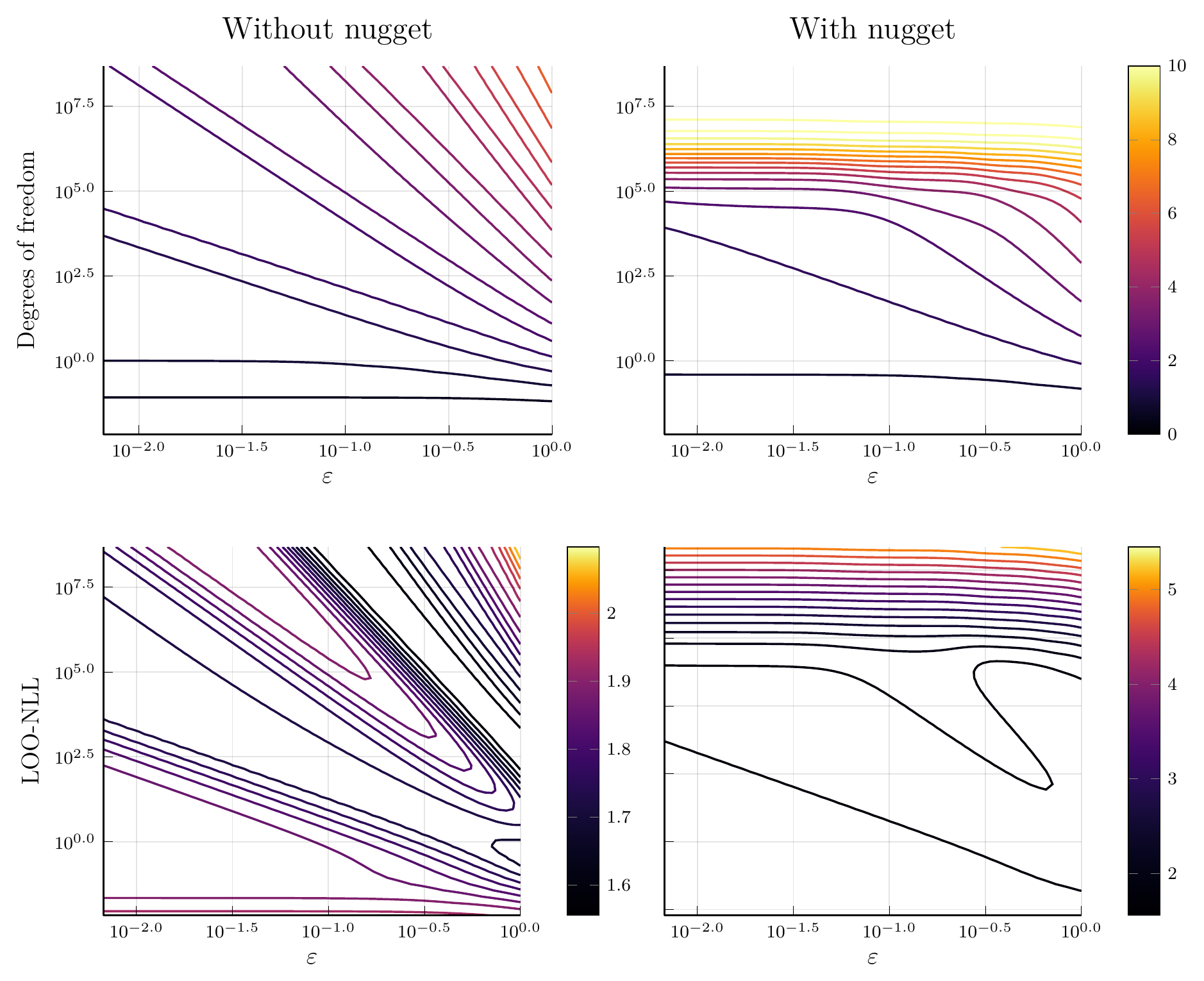}
  \caption{Effect of a nugget term on hyperparameter selection as $\flatlim$. A
    ``nugget term" is a small multiple of the identity that is added to the
    kernel matrix to get around numerical issues. We show the degrees of freedom
    (top row)
    and the LOO-NLL criteria (bottom row) with and without the nugget term (left
    and right columns). The nugget term
    equals $10^{-6} \bI$, we use the Gaussian kernel and the same data as in
    figure \ref{fig:hyperpar-sel}. }
  \label{fig:nugget}
\end{figure}

\subsection{Towards practical approximations}
\label{sec:towards-practical}

Our limit results are not directly applicable when faced with the question:
``what is a useful approximation of a particular GP model at a particular value
of $\varepsilon$''? We do not claim to have a universal recipe, but we shall
present in this section a particular approximation that gives surprinsingly good results in certain cases. 

This approximation is best understood graphically. We take as input a certain
kernel function, and a certain value for $\sigma^2$, $\gamma$ and $\epsilon$. We
can think of it as occupying a certain position in the space of hyperparameters
as shown on figure \ref{fig:deg-freedom} or \ref{fig:hyperpar-sel} for example.
The approximation we suggest, which we call the \emph{matched approximation},
consists in following the iso-freedom line from that point to $\flatlim$.
Following theorem \ref{thm:equivalent-kernels-1d}, the matched approximation
will be either a polynomial or a spline regression, with the same number of
degrees of freedom as the original GP regression. The process is illustrated
graphically in figure \ref{fig:illus-matched-approx-gaussian} and
\ref{fig:illus-matched-approx-matern} for two different kernels.

Let us sketch a concrete algorithm for kernels with $r=\infty$ and degrees of
freedom set to $m \in \R^+$.  By theorem
\ref{thm:equivalent-kernels-1d}, the equivalent semiparametric models are of the
form $S= \SPM{ \gamma x^{p}y^{p} }{\{x^0,\ldots,x^{p-1}\}}$ for some degree $p$.
The corresponding fit will have between $p$ and $p+1$ degrees of freedom, where
the former is attained with $\gamma=0$ and the latter with $\gamma \rightarrow
\infty$. We therefore need to set $p=\floor{m}$ and adjust $\gamma$ such that
the $m$-th eigenvalue of the smoother matrix equals $m-p$. A similar algorithm
applies for finite $r$. 

The matched approximation is illustrated on figs.
\ref{fig:illus-matched-approx-gaussian} to  \ref{fig:matched-results-gp-matern},
for kernels with different regularities, and for different values of
$\varepsilon$. For the Matérn kernel the quality of the approximation is
excellent even though $\varepsilon=2$; it is hard to account for this fact in
our current perturbative framework. 

\begin{figure}
  \centering
  \includegraphics[width=14cm]{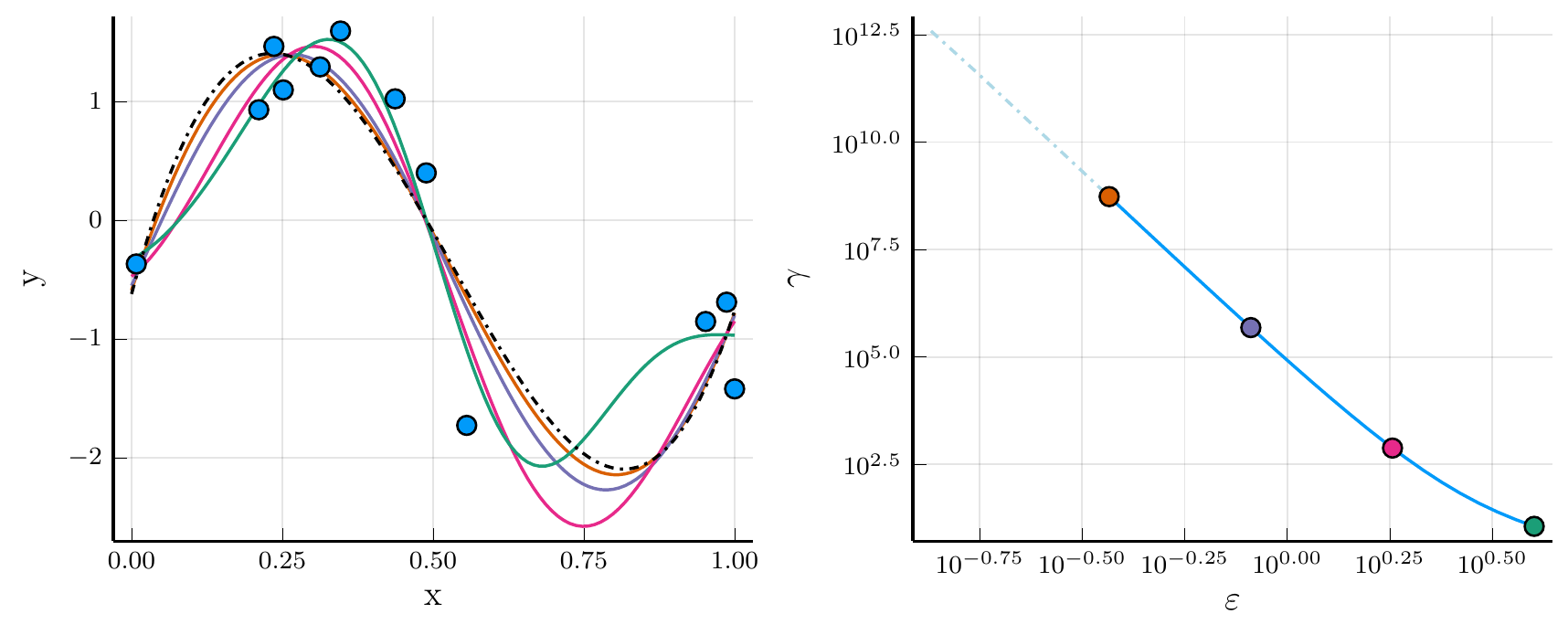}
  \caption{The matched approximation. The fit in green on the left-hand panel is
    a GP regression with $\varepsilon=4$ and 5 degrees of freedom. The kernel is
    Gaussian. The other coloured curves are other models along the isofreedom
    curve, with lower values of $\varepsilon$. The isofreedom curve is shown on
    the right-hand panel. The matched approximation is the limit obtained by
    following the isofreedom curve all the way to $\flatlim$. The corresponding
    fit is the dashed black curve on the left. See fig.
    \ref{fig:illus-matched-approx-matern} for the same thing with a Matern
    kernel.   }
  \label{fig:illus-matched-approx-gaussian}
\end{figure}

\begin{figure}
  \centering
  \includegraphics[width=14cm]{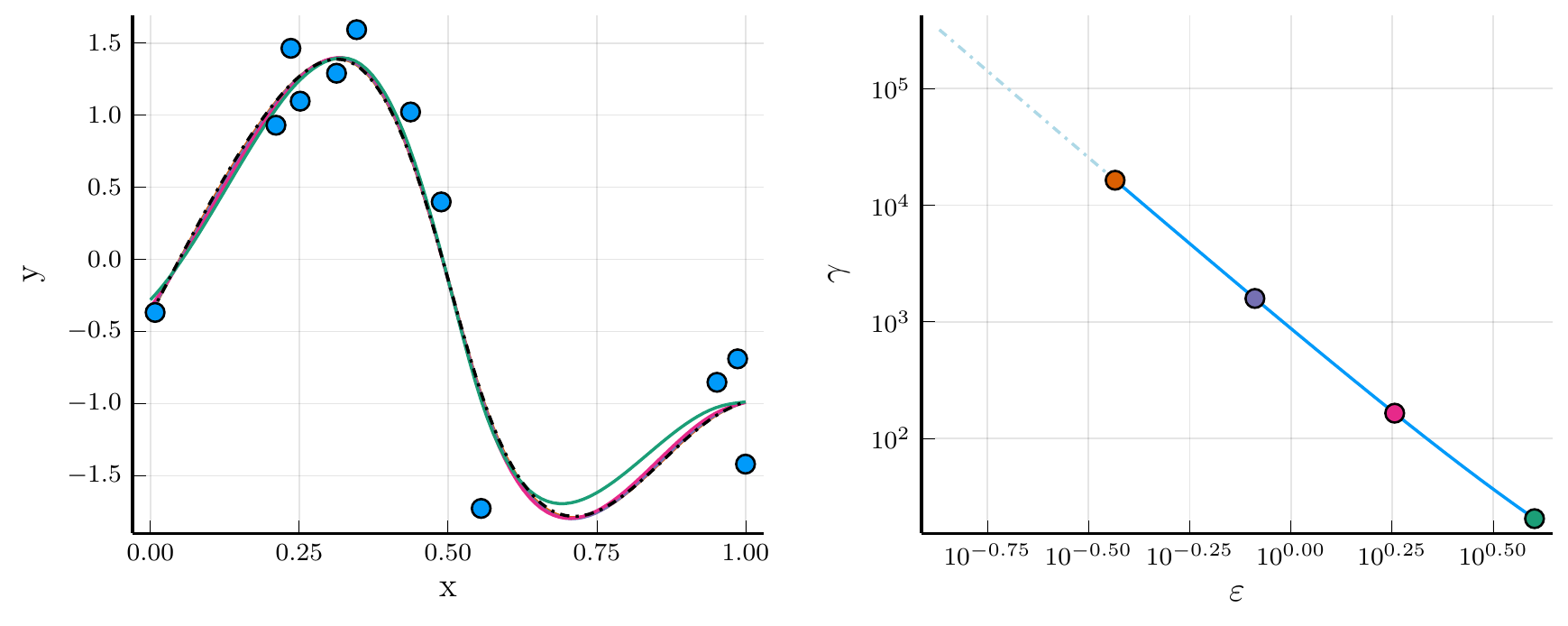}
  \caption{ Same as fig. \ref{fig:illus-matched-approx-gaussian}, but the kernel
  used is a Matérn kernel with regularity order $r=2$. The other parameters are
  identical. Note that the matched approximation fit is now very close to the
  original fit (even though the original fit has $\varepsilon=4$).}
  \label{fig:illus-matched-approx-matern}
\end{figure}

\begin{figure}
  \centering
  \includegraphics[width=14cm]{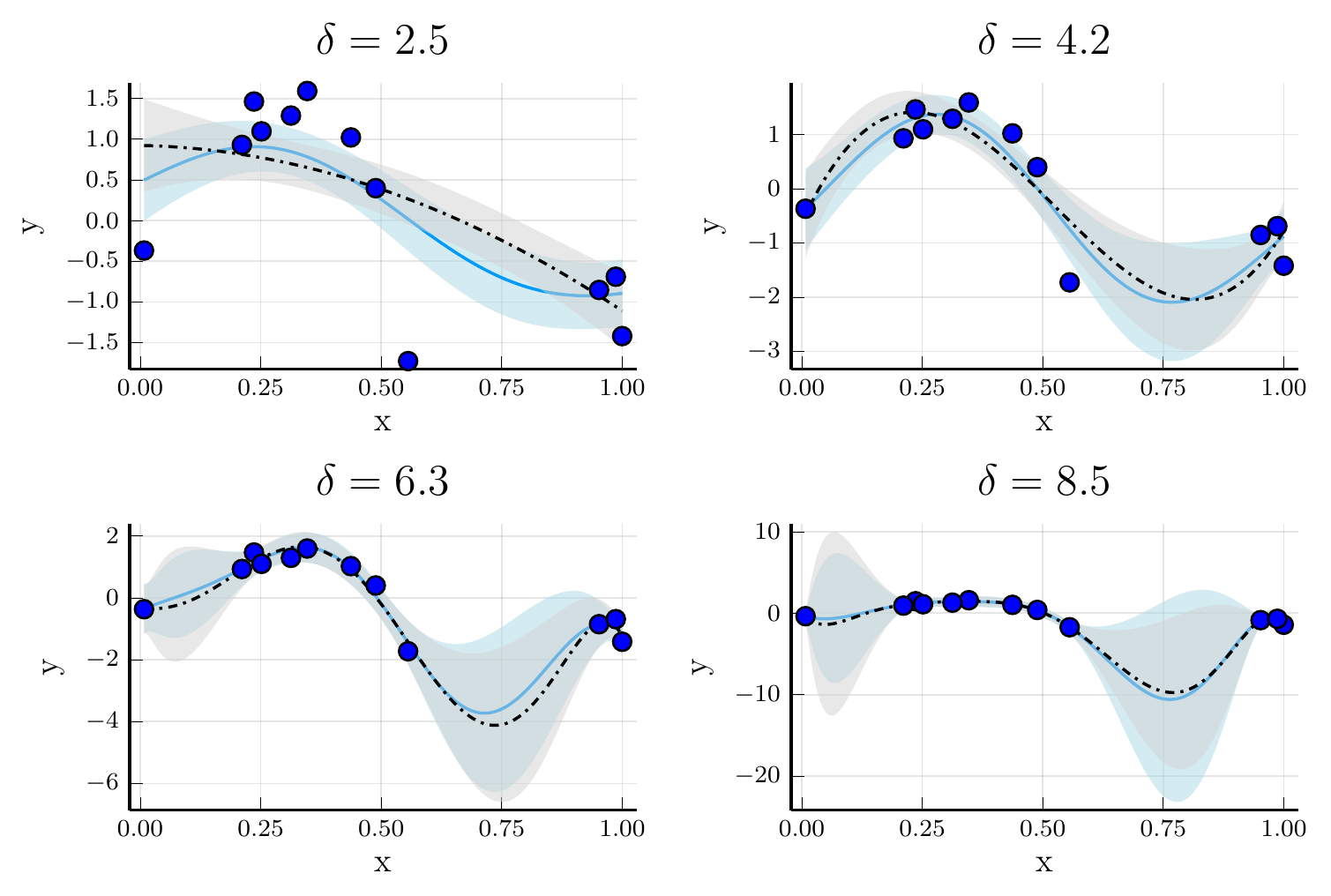}
  \caption{GP regression compared to its matched approximation (dashed curve). Here we use the
    Gaussian kernel and $\varepsilon=2$. The bands around the fit show $\pm$
    the standard deviation of the predictive distribution. We denote by $\delta$ the degrees of freedom for the
  different fits.}
  \label{fig:matched-results-gp-gauss}
\end{figure}

\begin{figure}
  \centering
  \includegraphics[width=14cm]{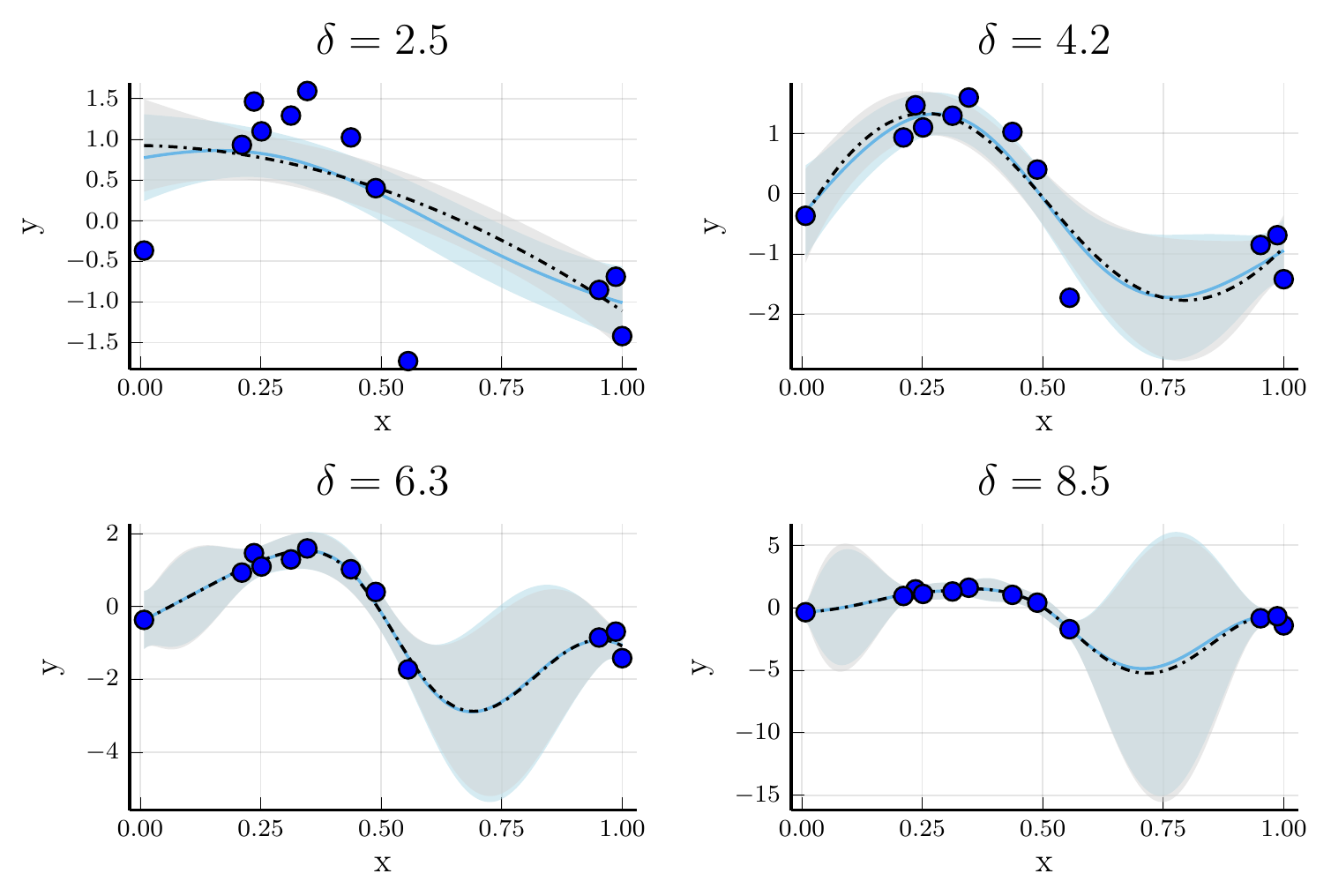}
  \caption{Same as fig. \ref{fig:matched-results-gp-gauss}, but with a Matérn
    kernel with regularity $r=3$.}
  \label{fig:matched-results-gp-matern}
\end{figure}

\subsection{An example with non-Gaussian likelihood}
\label{sec:non-gaussian}

Gaussian processes are used in myriad applications, and most of them involve
non-Gaussian likelihood functions. In appendix \ref{sub:non-gaussian}, we sketch how our results
extend to non-Gaussian likelihoods. Here, we illustrate these results
numerically, via an application that involves logistic likelihoods.

A classical example of combining a Gaussian
process prior with a non-Gaussian likelihood is GP classification, where the
data $\by \in \{0,1\}^n$ are independent binary outcomes and the model is:
\begin{equation}
  \label{eq:classification-lik}
  p(y_i = 1 | f(\bx_i)) = \Phi( f(\bx_i))
\end{equation}
where $\bx_1 \ldots \bx_n$ are ``feature vectors'' or covariates in $\R^d$, $f$
is a non-parametric function modelled as a Gaussian process, and $\Phi$ is a
sigmoidal link function. Often, $\Phi$ is chosen to be the logistic 
function $\Phi(x) = \frac{1}{1+\exp(-x)}$. Compared to the case of
Gaussian likelihoods, an additional difficulty is that the posterior
distribution over $f$ is not a Gaussian process (lack of conjugacy). This
makes it necessary to approximate the posterior, using some form of MCMC method
or approximate inference. In the numerical illustration shown below, we use
Expectation Propagation (EP, \cite{minka2001family}), which provides a Gaussian approximation to the
posterior. EP converges to the correct posterior in the large-$n$ limit \cite{dehaene2018expectation}, and is
known empirically to be extremely accurate in finite samples \cite{williams2006gaussian}.

An additional direction for extending our results consists in letting the
likelihood depend on arbitrary linear functionals of the Gaussian process $f$.
This lets model cases where observations depend on derivatives of $f$, or mean
values of $f$ over some area \cite{sarkka2011linear}. We sketch that extension
in appendix \ref{sec:general-linear}.

\begin{figure}
  \centering
  \includegraphics[width=14cm]{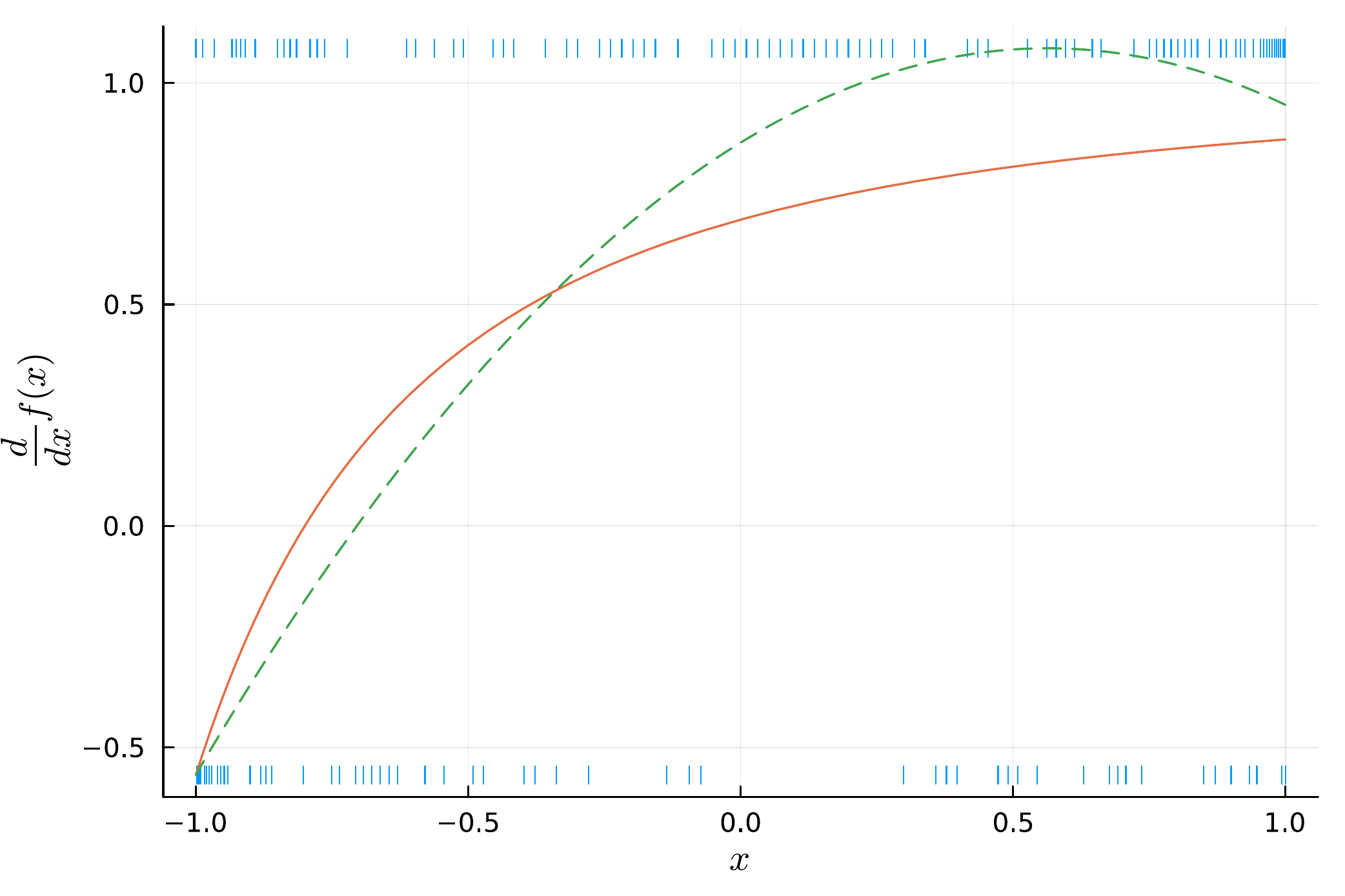}
  \caption{Reconstruction of a function from noisy observations of the sign of
    its derivative. In this example the measurements are binary, $y_i \in
    {0,1}$, with $\E(y_i) = \Phi(f'(x_i))$ and $\Phi$ the logistic function. We
    take $f(x) = x+\frac{1}{1.8+x}$. $f'(x)$ is shown as a solid red line, and
    the simulated measurements as blue dashes at the bottom  and at the top. Fitting a degree-2 polynomial to $f'$ results in the dashed green
    line. Integrating the estimated $f'(x)$ results in a estimator of $f(x)$. 
} 
  \label{fig:illus-derivative}
\end{figure}

Our numerical example brings together these two features (see fig. \ref{fig:illus-derivative}): we assume that the
observations correspond to the sign of the derivative of $f(x)$ (up to noise).
Concretely, the model is as follows: $f(x)$ is a univariate function, and we
observe $y_1 \ldots y_n \in {0,1}^n$ with
\begin{equation}
  \label{eq:classification-lik-fprime}
  p(y_i = 1 | f'(\bx_i)) = \Phi( f'(\bx_i))
\end{equation}
where $x_1 \ldots x_n$ are a set of locations in $[-1,1]$. The goal is to
reconstruct $f$ from these observations. Following \cite{sarkka2011linear}, if
$f$ is a Gaussian process with covariance $k(x,y)$, then $f'$ is a Gaussian
process with covariance $\frac{\partial }{\partial x} \frac{\partial}{\partial
  y} k(x,y)$. Accordingly, eq.~\eqref{eq:classification-lik-fprime} is just an
instance of GP classification with a specific kernel. We can run a standard
version of EP for Gaussian process classification to obtain an approximation of
$p(\bm{f'} | \by )$, where $\bm{f'} = [f'(x_1) \dots f'(x_n)]^t$.

The quantity of interest is however $f$, not $f'$. We can use $p(f(x)| \by)
\propto p(f | \bm{f'})p(\bm{f'} | \by)$ and standard Gaussian conditioning
formulas to obtain:
\begin{equation}
  \label{eq:conditional-exp-derivatives}
  \E(f(x) | \by ) = b(x)^t \bM^{-1} \E(\bm{f'} | \by)
\end{equation}
with $b(x) = [\frac{\partial }{\partial y}
k(x,y)\vert_{y=x_i} ]_{i=1}^n $ a vector of first derivatives of the kernel
function, and
\[ \bM=\left[\frac{\partial }{\partial a}
  \frac{\partial}{\partial b} k(a,b) \vert_{a=x_i,b=x_j}\right]_{i=1,j=1}^n\]
a matrix of
second derivatives. In equation \ref{eq:conditional-exp-derivatives}, $b(x)^t \bM^{-1}$ is best viewed as an
integration operator, that converts observations of a derivative
into an approximation of the function at $x$. A similar equation can be derived
for the posterior variance of $f(x)$ as a function of the variance of
$\bm{f'}|y$. Since no information is available on the mean value of $f$ over the
interval, we plot below the results for $f - \frac{1}{2}\int_{-1}^1 f(x)dx$.
Equivalently, we condition on $\int_{-1}^1 f(x)dx = 0$.

If we use the Gaussian kernel for $f$, then the flat limit behaviour should
match that of a polynomial model. Let us briefly work out what that polynomial
model looks like. Assuming $f(x) = \sum \alpha_{i=0}^m x^{i}$, then $f'(x) =
\sum_{i=1}^{m} (i+1)\alpha_{i+1} x^{i}$. If we use a flat prior on the
coefficients $\alpha$, eq. \ref{eq:conditional-exp-derivatives} turns into a
classical Bayesian logistic regression with $m$ covariates, and we can use EP to
approximate the posterior over $\alpha_1 \ldots \alpha_m$ given $\by$.
Numerically, a better alternative is to use the Legendre polynomials (instead of
the monomial basis) to improve conditioning, and that is what we do in our
implementation.

If instead of the Gaussian kernel we use a kernel with finite smoothness, the
flat limit behaviour corresponds to fitting a smoothing spline. We need
$f$ to be at least once differentiable, which implies that the order of
smoothness should be at least one (which precludes the exponential kernel). In
our illustrations we use a Matérn kernel with $r=3$. The matching
semi-parametric model is of the form $S =\SPM{ -|x-y|^{5}}{\{1,x,x^2\}}$. This
corresponds to a spline of degree 3 (for $f(x)$), and the corresponding model
for $f'(x)$ is then a spline of degree 2. Stating this in terms of kernels
exaggerates the complexity of what we are doing: the procedure consists in
fitting a smoothing spline to the data, which produces an estimate for $f'(x)$,
and integrating that estimate to get an estimate of $f(x)$.  

It remains to compare the results of fitting a GP to directly fitting an
equivalent flat-limit model. In the regression case, we made use of the
effective degrees of freedom to match GP fits to flat-limit results. When
observations are non-Gaussian, an additional source of difficulty arises,
because the estimate is not linear in the observations (which in our case or
binary anyways). Different generalisations can be found in the literature, and
here we follow \cite{o1986automatic} and define the effective d.o.f. from a
Gaussian approximation to the posterior distribution. Specifically, the
approximation produced by Expectation Propagation for the posterior $p(\ff|
\by)$ takes the following form:
\begin{equation}
  \label{eq:ep-approx}
  q(\ff) = \exp(-\frac{1}{2} \ff^t (\bK^{-1} + \bH) \ff + \mathbf{r}^t \ff)
\end{equation}
where $\bH$ is a diagonal matrix, $\mathbf{r}$ a vector, and both depend (non-linearly)
on the data $\by$. In \cite{o1986automatic} the Gaussian approximation at the
mode is
used implicitly, and here we can define analogously the d.o.f. as
\begin{equation}
  \label{eq:dof-non-Gaussian}
  dof(\bK,\by) = \Tr((\bK^{-1} +  \bH)^{-1}\bH)
\end{equation}

One can check that this definition generalises the case of
(heteroskedastic) Gaussian observations, and \cite{o1986automatic} outline
asymptotic arguments in terms of model selection. We will
not repeat them here but note that they carry over to the Gaussian approximation
formed by Expectation Propagation, by the results in
\cite{dehaene2018expectation}. 

Figures \ref{fig:invprob-gaussian} and \ref{fig:invprob-matern} show the results for Gaussian and Matérn kernels,
respectively, with $\varepsilon = \frac{1}{2}$. The approximate d.o.f. given by
eq. \ref{eq:dof-non-Gaussian} succeeds in matching the GP fits to very close
flat-limit equivalents. The match is markedly better with higher d.o.f., even at
higher values of $\varepsilon$ (not shown). We suspect that this has to do with
faster convergence of the eigenvectors associated with smaller eigenvalues to
their flat limit, but our theory is currently insufficient to properly
explain this phenomenon.

\begin{figure}
  \centering
  \includegraphics[width=14cm]{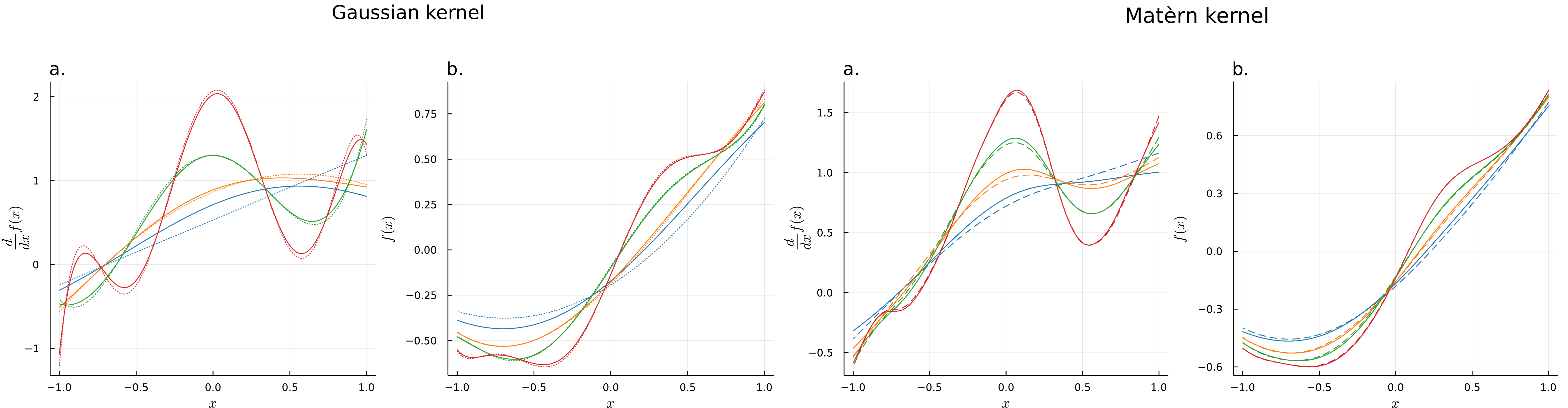}
  \caption{Results of fitting a GP model with Gaussian covariance to the problem
    outlined in fig. \ref{fig:illus-derivative} (same data, $\varepsilon = \frac{1}{2}$), compared to
    matching flat-limit results. The GP fits are shown as solid lines, and
    polynomial fits as dashed lines. a. Estimated value of $f'(x)$, with different
    d.o.f., matched to polynomial fits b. Corresponding estimates for $f(x)$. 
  } 
  \label{fig:invprob-gaussian}
\end{figure}
\begin{figure}
  \centering
  \includegraphics[width=14cm]{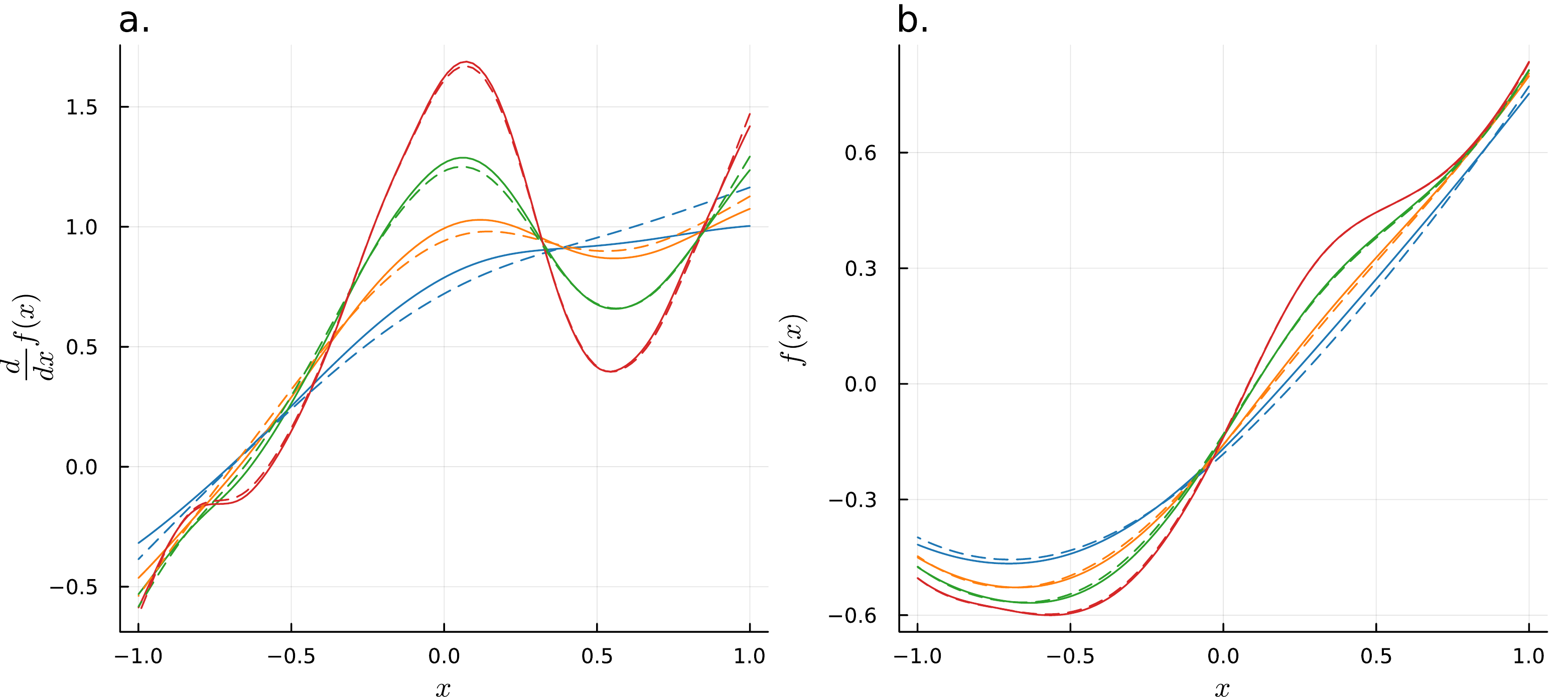}
  \caption{Same as fig. \ref{fig:invprob-gaussian}, with a $r=3$ Matérn kernel.
    The matching flat limit models are smoothing splines.     The GP fits are shown as solid lines, and
    the smoothing spline fits as dashed lines.
  } 
  \label{fig:invprob-matern}
\end{figure}


\section{Results in the multivariate case}
\label{sec:multivariate}

To deal with the multivariate case, we require a bit of background on
multivariate polynomials and polyharmonic splines. 

\subsection{Preliminaries and notation}
\label{sec:notation-multivar}

Much of the material here is drawn from
\cite{barthelme2023determinantal,BarthelmeUsevich:KernelsFlatLimit}, please
refer to these papers for a more extensive background. Much information can also be found in {\it e.g. }\cite{wendland2004scattered} .

Let $\vect{x} =
\begin{pmatrix}
  x_1 & x_2 & \ldots & x_d
\end{pmatrix}^\top
\in \R^d $. A monomial in $\vect{x}$ is a function of
the form:
\[ \vect{x}^{\vect{\alpha}} = \prod_{i=1}^d x_i^{\alpha_i}  \]
for $\vect{\alpha} \in \mathbb{N}^d$ (a multi-index). The degree of a monomial
is defined $|\vect{\alpha}|=\sum_{i=1}^d \alpha_i$. For instance:
$ \vect{x}^{(1,3,1)} = x_1^1x_2^3x_3^1$  has degree 5.

A multivariate polynomial in $\vect{x}$ is a weighted sum
of monomials in $\vect{x}$, and its degree is equal to the maximum of the
degrees of its component monomials. As an example, $ -\vect{x}^{(1,2,1)}+ \vect{x}^{(0,1,1)}+2.2\vect{x}^{(1,0,0)}-1$ is a
multivariate polynomial of degree 4 in $\R^3$:

An important difference between the univariate and the multivariate case is that
when $d>1$, there are several monomials of any given degree, instead of just
one. For instance, with $d=2$, the first few monomials are $
  \vect{x}^{(0,0)}$ of degree 0; $
  \vect{x}^{(1,0)},\vect{x}^{(0,1)} $ of degree 1 ;
 $ \vect{x}^{(2,0)},\vect{x}^{(1,1)},\vect{x}^{(0,2)}$ of degree 3. The number of monomials of degree $k$ in dimension $d$ is :
\begin{equation}
  \label{eq:number-monomials}
  \HH_{k,d} = {k+d-1 \choose d-1}.
\end{equation}
The notation $\HH_{k,d}$ comes from the notion of \emph{homogeneous
  polynomials}.

A homogeneous polynomial is a polynomial made up of monomials with equal degree.
Therefore, the set of homogeneous polynomials of degree $k$ has dimension
$\HH_{k,d}$. The set of polynomials of degree $k$ is spanned by the sets of
homogenous polynomials up to $k$, and has dimension:
\begin{equation}
  \label{eq:dim-polynomials}
  \PP_{k,d} = \HH_{0,d} + \HH_{1,d} + \ldots + \HH_{k,d} = {k+d \choose d}.
\end{equation}
Note that $\PP_{0,d}=1$ and $\PP_{1,d}=d+1$. By convention, we will also set $\PP_{-1,d}$ to be equal to $0$.

The fact that there are several monomials for each degree in dimension $d$ is
reflected in the structure of the eigenvalues in the flat limit. Previously, in
the $r=\infty$ case, each eigenvalue had a different order in $\varepsilon$. In
the multivariate case, there are blocks of eigenvalues with the same order in
$\varepsilon$, corresponding to a block of homogeneous polynomials of a given
degree $m$. For instance, in $d=2$, there is one monomial of degree $0$, two
monomials of degree $1$ ($x_1$ and $x_2$), three monomials of degree $2$
($x_1^2,x_1x_2$ and $x_2^2$), and in general $m+1$ monomials of order $m$. As
first shown in \cite{schaback2005multivariate}, these give rise to a single
eigenvalue of order $\varepsilon^0$, two eigenvalues of order $\varepsilon^2$,
three eigenvalues of order $\varepsilon^4$, etc.

\subsubsection{Polynomial bases and orderings}
\label{sec:bases-orderings}

The multivariate flat limit is more complicated than the univariate case, even
though the results are substantially the same. The reason why the results are
more complicated is fairly deep and boils down to the lack of a natural order on
the set of multivariate monomials. 

We use  in section \ref{sec:results} the fact that eigenvectors of smooth kernel matrices tend to
discrete polynomials. In dimension $d=1$, there is an obvious way to construct a
basis of orthogonal polynomials, which is just to apply the Gram-Schmidt process
to the monomials $(1,x,x^2,x^3,\ldots)$. The monomials in dimension 1 are
naturally ordered by increasing degree. In dimension two, the degree only gives
a partial order. For instance, at degree one, even though the constant
polynomial is a consensus starting point, we have to decide at degree 1 which of
$x_1$ or $x_2$ should come first in the Gram-Schmidt process. Depending on which
we pick, we get a different orthogonal basis spanning polynomials of degree
$\leq 1$. We could also decide to orthogonalise $x_1+x_2$ followed by $x_1-x_2$,
and get yet another basis. Multivariate orthogonal polynomials are non-unique,
and therefore both richer and more complicated than univariate orthogonal
polynomials.

To state our results, we need to pick an ordering on the monomials, even though
the ordering is immaterial to the actual limits (kernel matrices do not care how
we order monomials). The need for
an ordering is an annoyance that can probably be lifted by finding a
representation that is intrinsically invariant, but we have not found one as
yet. 

In any event, given an ordering, for an ordered set of points $\Omega = \{\vect{x}_1, \ldots, \vect{x}_n\}$, all in $\RR^d$, we define the multivariate Vandermonde matrix as:
\begin{equation}
  \matr{V}_{\le k} = 
  \begin{bmatrix}
    \matr{V}_{0} & \matr{V}_{1}  & \cdots & \matr{V}_{k}  
  \end{bmatrix} \in \mathbb{R}^{n\times \PP_{k,d}},
\end{equation}
where each block $\matr{V}_i \in\mathbb{R}^{n\times \HH_{i,d}}$ contains the monomials of degree $i$ evaluated on
the points in $\Omega$. As an example, consider $n=3$, $d=2$ and the ground set
\[
  \Omega = \{\left[\begin{smallmatrix}y_1 \\z_1 \end{smallmatrix}\right], \left[\begin{smallmatrix}y_2 \\z_2 \end{smallmatrix}\right], \left[\begin{smallmatrix}y_3 \\z_3 \end{smallmatrix}\right]\}.
\]
One has, for instance for $k=2$:
\[
  \matr{V}_{\le 2} =
  \left[\begin{array}{c|cc|ccc}
          1 & y_1 & z_1 & y_1^2 & y_1z_1 & z_1^2 \\
          1 & y_2 & z_2 & y_2^2 & y_2z_2 & z_2^2 \\
          1 & y_3 & z_3 & y_3^2 & y_3z_3 & z_3^2 \\
        \end{array}\right],
    \]
    where the ordering within each block is arbitrary. 

We will use $\bV_{\le k}(\X)$ to denote the matrix $\bV_{\le k}$ reduced to its
lines indexed by the elements in $\mathcal{X}$. As such, $\bV_{\le k}(\X)$ has
$|\X|$ rows and $\PP_{k,d} $ columns. The QR decomposition of $\bV_{\le k}$
inherits a natural block structure from $\bV_{\le k}$ corresponding to the
degrees of the monomials, {\it i.e.} we may split $\bQ_{\le k}$ into blocks $\bQ_{0}$,
$\bQ_{1}$, etc. where $\bQ_{i}$ comes from the Gram-Schmidt process applied to
monomials of degree $i$ onto monomials of lower degree.

What this means for kernel matrices is that the particular limiting eigenbasis
that appears as $\flatlim$ depends more strongly on the kernel than in the
univariate case. In a sense, the kernel implicity selects a particular family of
orthogonal polynomials. The specific basis is determined by the so-called
\emph{Wronskian matrix} of the kernel, defined as:
\begin{equation}\label{eq:wronskian_nd}
  \matr{W}_{\leq k}  = 
  \left[
    \frac{k^{(\va,\vect{\beta})} (\vect{0},\vect{0})}{\va!\vect{\beta}!}
  \right]_{|\va| \leq k, |\vb| \leq k} \in\mathbb{R}^{\PP_{k,d}\times \PP_{k,d}}.
\end{equation}
$k^{(\va,\vect{\beta})}$ being the partial derivatives of $k(\vect{x},\vect{y})$ with respect to $\vect{x}^{\vect{\alpha}} $ and $\vect{x}^{\vect{\beta}} $.
Here we index the matrix using multi-indices (equivalently, monomials), so that an element of
$\matr{W}_{\leq k}$ is {\it e.g.}, $\matr{W}_{(0,2),(2,1)}$ which is a scaled
derivative of $k(\vect{x},\vect{y})$ of order $(0,2)$ in $\vect{x}$ and $(2,1)$
in $\vect{y}$.
For example, for $d=2$ and $k=2$ we may write 
\[
\matr{W}_{\leq 2} =\begin{bmatrix}
k^{((0,0),(0,0))} & k^{((0,0),(1,0))} & k^{((0,0),(0,1))} & \frac{k^{((0,0),(2,0))}}{2} & {k^{((0,0),(1,1))}} & \frac{k^{((0,0),(0,2))}}{2} \\
k^{((1,0),(0,0))} & k^{((1,0),(1,0))} & k^{((1,0),(0,1))} & \frac{k^{((1,0),(2,0))}}{2} & {k^{((1,0),(1,1))}} & \frac{k^{((1,0),(0,2))}}{2} \\
k^{((0,1),(0,0))} & k^{((0,1),(1,0))} & k^{((0,1),(0,1))} & \frac{k^{((0,1),(2,0))}}{2} & {k^{((0,1),(1,1))}} & \frac{k^{((0,1),(0,2))}}{2} \\
\frac{k^{((2,0),(0,0))}}{2} & \frac{k^{((2,0),(1,0))}}{2} & \frac{k^{((2,0),(0,1))}}{2} & \frac{k^{((2,0),(2,0))}}{4} & {\frac{k^{((2,0),(1,1))}}{2}} & \frac{k^{((2,0),(0,2))}}{4} \\
k^{((1,1),(0,0))} & k^{((1,1),(1,0))} & k^{((1,1),(0,1))} & \frac{k^{((1,1),(2,0))}}{2} & {k^{((1,1),(1,1))}} & \frac{k^{((1,1),(0,2))}}{2} \\
\frac{k^{((0,2),(0,0))}}{2} & \frac{k^{((0,2),(1,0))}}{2} & \frac{k^{((0,2),(0,1))}}{2} & \frac{k^{((0,2),(2,0))}}{4} & {\frac{k^{((0,2),(1,1))}}{2}} & \frac{k^{((0,2),(0,2))}}{4} \\
\end{bmatrix}\in\mathbb{R}^{\PP_{2,2}\times \PP_{2,2}}
\]
for a given ordering of the monomials, and where all the derivatives are taken
at $\vect{x}=0,\vect{y}=0$. Eq. \eqref{eq:wronskian_nd} makes Wronskian matrices
look more daunting to compute than they really are. We explain in the appendix
how the Wronskian may easily be computed in the stationary case from the Fourier transform of the kernel.

\subsubsection{Polyharmonic splines}
\label{sec:polyharmonic-splines}

Polyharmonic splines \cite{duchon1977splines} generalise smoothing splines in
$d>1$, and play the same role in the flat limit. For our
purposes here, the space of polyharmonic splines of order $r$ in dimension $d$
for a point set $\X$ is given by functions of the form:
\begin{equation}
  \label{eq:phs-def}
  f(\bx) = \sum_{i=1}^n \beta_i \norm{\bx - \bx_i}^{2r-1} + \sum_{|\vect{\gamma}| < k} \alpha_{\vect{\gamma}}\bx^{\vect{\gamma}}
\end{equation}
where $\bV_{< k}^\top\vect{\beta}= 0$.

We recognise the general form of
semi-parametric models (eq. \eqref{eq:cond-exp-semi-par-RKHS}), where here the
parametric part is played by monomials of degree less than $r$, and the
non-parametric part by the radial basis function $\norm{\bx - \bx_i}^{2r-1}$.
In our notation, polyharmonic spline models are therefore semiparametric models given by
\[ \M_r = \SPM{ (-1)^r \norm{\bx - \by}^{2r-1}  }{ \{ \bx^{\vect{\alpha}} \vert
    |\vect{\alpha}| < r \}} \]
Note that polyharmonic splines generalise splines to $d>1$, but they are not
piecewise polynomials. The fact that $l(\bx) = (-1)^r \norm{\bx - \by}^{2r-1}$
is conditionally positive-definite is proved in
\cite{micchelli1986interpolation} \footnote{More precisely, it is a minor
  variant of the functions actually studied.}.



\subsection{Smoother matrices in $d > 1$}
\label{sec:smoother-multivariate}

The smoother matrices in $d>1$ have the same kind of limit as in the univariate
case. Depending on the growth rate of $\gamma$, $n$ and the regularity of the
kernel, sometimes one has polynomials, sometimes splines. The next lemma gives a
complete picture, and reexpresses theorem 5.2 from
\cite{barthelme2023determinantal} in a form adapted to the GP context. To
lighten the notation in the lemma, we define the following matrices, which
appear in the flat limit of the eigenvectors: 
\begin{equation}
  \bP_l = \bQ_{l}\bQ_{l}^\top\bV_{l} \bar{\bW}_l\bV_{l}^\top\bQ_{l}\bQ_{l}^\top\label{eq:eigenvec-lim}
\end{equation}
where  $\bar{\bW}_l \in\mathbb{R}^{\HH_{l,d}\times \HH_{l,d}}$ is the  Schur
complement:
\begin{equation}
  \label{eq:wronskian-schur}
  \bar{\bW}_l= \matr{W}_{\lrcorner} - \matr{W}_{\llcorner}
  (\matr{W}_{\leq l - 1})^{-1}\matr{W}_{\urcorner}
\end{equation}
in the block description of $\matr{W}_{\leq l}$ or
\[
\left(\begin{array}{c|c}\matr{W}_{\leq l - 1} & \matr{W}_{\urcorner} \\\hline \matr{W}_{\llcorner}& \matr{W}_{\lrcorner} \end{array}\right)
\]
We recall that $\widetilde{\bD^{(2r-1)}}$ is the matrix $\bD^{(2r-1)} = \left[ \norm{\bx_i -
    \bx_j}^{2r-1} \right]_{i,j}$ with monomials of degree $< r$ projected out, {\it i.e.}
\begin{equation}
  \label{eq:Dtilde}
 \widetilde{\bD^{(2r-1)}} = (\bI - \bQ_{< r} \bQ_{< r}^\top)\bD^{(2r-1)}(\bI - \bQ_{< r} \bQ_{< r}^\top)
\end{equation}

The following lemma is not particularly easy to read and the reader may skip ahead to the
theorem at no great loss. It generalizes to the multivariate setting the first steps in the proof of theorem 
 \ref{thm:equivalent-kernels-1d} in the univariate case (see Th. \ref{eigenflat1d:th} and subsections \ref{proof1dsmooth:sssec} and  \ref{proof1dnonsmooth:sssec} ).
\begin{lemma}
  \label{lem:smoother-multivariate}
  Let $\X \subset \Omega \subset \R^d$ with a set of $|\X| = n$ measurement
  locations, $k_\varepsilon$ a kernel with regularity $r$, $p$ an integer and $\gamma(\varepsilon)
  = \gamma_0\varepsilon^{-p}$. Then the smoother matrix
  \[ \bM_\varepsilon = \bK_\varepsilon\left(\bK_\varepsilon +
    \frac{\sigma^2}{\gamma(\varepsilon)} \bI\right)^{-1} \] has the following expansion in
  $\flatlim$:
  \begin{equation}
    \label{eq:limit-smoother-multivar}
    \bM_\varepsilon = \bA + \bB \matr{\Gamma} \bB^\top+ \O(\varepsilon)
  \end{equation}
  where $\bA$ is a projection matrix, $\bB^\top\bA=0$, and $\matr{\Gamma}$ is
  diagonal (and in some cases null).  
  $\bA$, $\bB$ and $\matr{\Gamma}$ depend on $r$, $n$ and $p$.
  First, $p$ is either even or odd, meaning that only one out of the two
  following values $\frac{p}{2},\frac{p+1}{2}$ is an integer. We call that
  integer l. The possible limits are:
  \begin{itemize}
  \item If $\PP_{l-1,d} \geq n$ or $r < \frac{p+1}{2}$ then $\bM_\varepsilon = \bI + \O(\varepsilon)$
  \item If $r > \frac{p+1}{2}$ and $p$ is odd, then $\matr{\Gamma}=0$ and $\bA =
    \bQ_{<l}\bQ_{<l}^\top$ 
  \item If  $r > \frac{p+1}{2}$ and $p$ is even, then $\bA =
    \bQ_{< l}\bQ_{< l}^\top$, $\bB$ are the (non-null) eigenvectors of $\bP_l$
    (defined above) and $\gamma_{ii} = \frac{\gamma_0\lt_i}{1+\gamma_0\lt_i}$,
    where $\lt_i$ is the i-th eigenvalue of $\bP_l$. 
  \item If $r = \frac{p+1}{2}$, then $\bA = \bQ_{\leq r-1}\bQ_{\leq r-1}^\top$  and
    $\bB$ are the non-null eigenvectors of $f_{2r-1}\tilde{\bD}^{(2r-1)}$, $\lt_i$
    its eigenvalues, and $\gamma_{ii} = \frac{\gamma_0\lt_i}{1+\gamma_0\lt_i}$.
  \end{itemize}
\end{lemma}

In a nutshell, the smoother matrices are in the limit either projection
matrices, or the sum of a projection matrix and a smoother matrix. This
indicates that the limiting models are generally semi-parametric and
occasionally parametric.

\subsection{Main result in $d>1$}
\label{sec:main-result-multidim}

The generalisation of theorem \ref{thm:equivalent-kernels-1d} to the
multivariate case requires the following assumptions. 

In the statement of the theorem, $k_\varepsilon$ designates
a family of kernels indexed by an inverse-scale parameter $\varepsilon$, of the
form
\[ k_\varepsilon(\bx,\by)= \kappa(\varepsilon \bx, \varepsilon \by) \]
The required assumptions are: 
\begin{enumerate}
\item $\kappa$ is stationary and radial (isotropic); i.e. there exists $\psi$ such that $\kappa(\bx,\by) =
  \psi\left(\norm{\bx-\by}_2 \right)$
\item $\psi$ is analytic in a neighbourhood of 0. 
\end{enumerate}
The first assumption is for simplicity, and because the most common types of
kernels are radial. Non-radial kernels can be dealt with using the tools in
\cite{BarthelmeUsevich:KernelsFlatLimit}, but at the cost of greater complexity.
As before, the second assumption can be removed in some cases, with some
subtleties involved, see appendix \ref{sec:non-analytic-kernels}. 

\begin{theorem}
	\label{thm:equivalent-kernels-nd}
	Let $k_\varepsilon(\bx,\by) = \kappa( \varepsilon(\bx-\by) )$ a family of kernels with
  inverse-scale parameter $\epsilon$. Let $\kappa$ be a stationary positive-definite kernel for $\bx,\by$ in
	$\R^d$, with regularity parameter $r$, and $p$ an integer. 
	Then the following asymptotic equivalence holds:
	\[ k_\varepsilon\varepsilon^{-p} \fleq \SPM{l_p}{\V_p}\]
	where $l_p(x,y)$ and $\V_p$ depend on the interplay between $p$ and $r$. There are four different cases:
	\begin{itemize}
  \item $p<2r-1$ and $p$ is even, \emph{i.e.}, $\exists\, m<r\text{ s.t. } p =
    2m$. Then $l(\bx,\by) = \sum_{|\ba|=m,|\bb|=m} \bar{W}_m(\ba,\bb)
    \bx^{\ba}\by^{\bb}$, $\V = \{ \bx^{\vect{\gamma}} \vert\ |\vect{\gamma}| < m
    \}$. This case amounts to penalised polynomial regression.
  \item $p<2r-1$ and $p$ is odd, \emph{i.e.}, $\exists\, m<r-1 \text{ s.t. } p =
    2m+1$. Then $l(\bx,\by) = 0$, $\V = \{ \bx^{\vect{\gamma}} \vert\
    |\vect{\gamma}| < m \}$. This case amounts to unpenalised polynomial
    regression.
  \item $p=2r-1$. In this case, $l(\bx,\by) = (-1)^r \norm{\bx-\by}^{2r-1}$, $\V
    = \{ \bx^{\vect{\gamma}} \vert\ |\vect{\gamma}| < r \} $, which amounts to
    polyharmonic spline regression.
  \item $p>2r-1$. This case leads to an interpolant, if it exists, independently
    of the value of $\sigma^2$. The interpolant is either a polyharmonic spline
    (finite $r$) or a polynomial (infinite $r$). The interpolant may not exist;
    this depends on the number of points in $\X$ and its geometry.
	\end{itemize}
\end{theorem}
\begin{proof}
  Nearly identical to the proof of \ref{thm:equivalent-kernels-1d}. The only
  difference is that lemma \ref{lem:smoother-multivariate} is used in the part
  labelled ``Final step of the proof''.
\end{proof}

We show in appendix \ref{sec:schur-compl-wronskians} that for separable kernels the Schur complements of the
Wronskian (eq. \eqref{eq:wronskian-schur}) are actually diagonal. For the
Gaussian kernel a further simplification is possible, and gives a very compact
limit result. The  ``polynomial kernel'' of order $m$ is
\[ \rho_m(\bx,\by) = (\bx^\top\by)^m \]
and its associated reproducing kernel Hilbert space is the set of monomials
in $\R^d$ of degree $m$.

\begin{corollary}[Flat limit of Gaussian kernels]
  \label{cor:flat-limit-gaussian}
  For the Gaussian kernel in $\R^d$, the following equivalence holds as
  $\flatlim$:
  \begin{enumerate}
  \item For even $p = 2m$, \[ k_\varepsilon\varepsilon^{-p} \fleq \SPM{
        \rho_m }{ \{ \bx^{\vect{\gamma}} \vert\  |\vect{\gamma}| < m \}}\]
  \item For odd $p = 2m+1$
    \[ k_\varepsilon\varepsilon^{-p} \fleq \SPM{
        0 }{ \{ \bx^{\vect{\gamma}} \vert\  |\vect{\gamma}| < m \}}\]
  \end{enumerate}
\end{corollary}
The proof is given in section \ref{sec:schur-compl-wronskians}.
The corollary states that in the flat limit, depending on the level of
regularisation, the GP model is either plain (multivariate) polynomial
regression, or a SPM with a parametric part of polynomials of degree $< m$, and
a ``non-parametric'' part which is given by the polynomial kernel of degree $m$. 

\subsection{Numerical results}
\label{sec:numerical-results-2d}

We illustrate our results with a few simulations in dimension 2. We generated a
set of 30 random locations in $[0,1]^2$ (sampled uniformly and independently),
and noiseless observations $y_i$ from the function
$f(x_1,x_2)=\exp\left(-3((x_1-0.5)^2+(x_2-0.5)^2
\right)\sin\left(3(x_1+x_2)\right)$.
Fig. \ref{fig:numresults-2d} shows (on the left) the contour lines of two GP regressions with $\varepsilon=1$, one with a
Gaussian kernel, the other with a Matérn kernel with $r=3$. $\gamma$ has been
adjusted so that the degrees of freedom equal approximately 12 in both cases. On the
right, the corresponding matched approximations (as in section
\ref{sec:towards-practical}), respectively multivariate polynomials and
polyharmonic splines. Over this range and for this value of $\varepsilon$ the
agreement is excellent (but see later for caveats). 
\begin{figure}
  \centering
  \includegraphics[width=12cm]{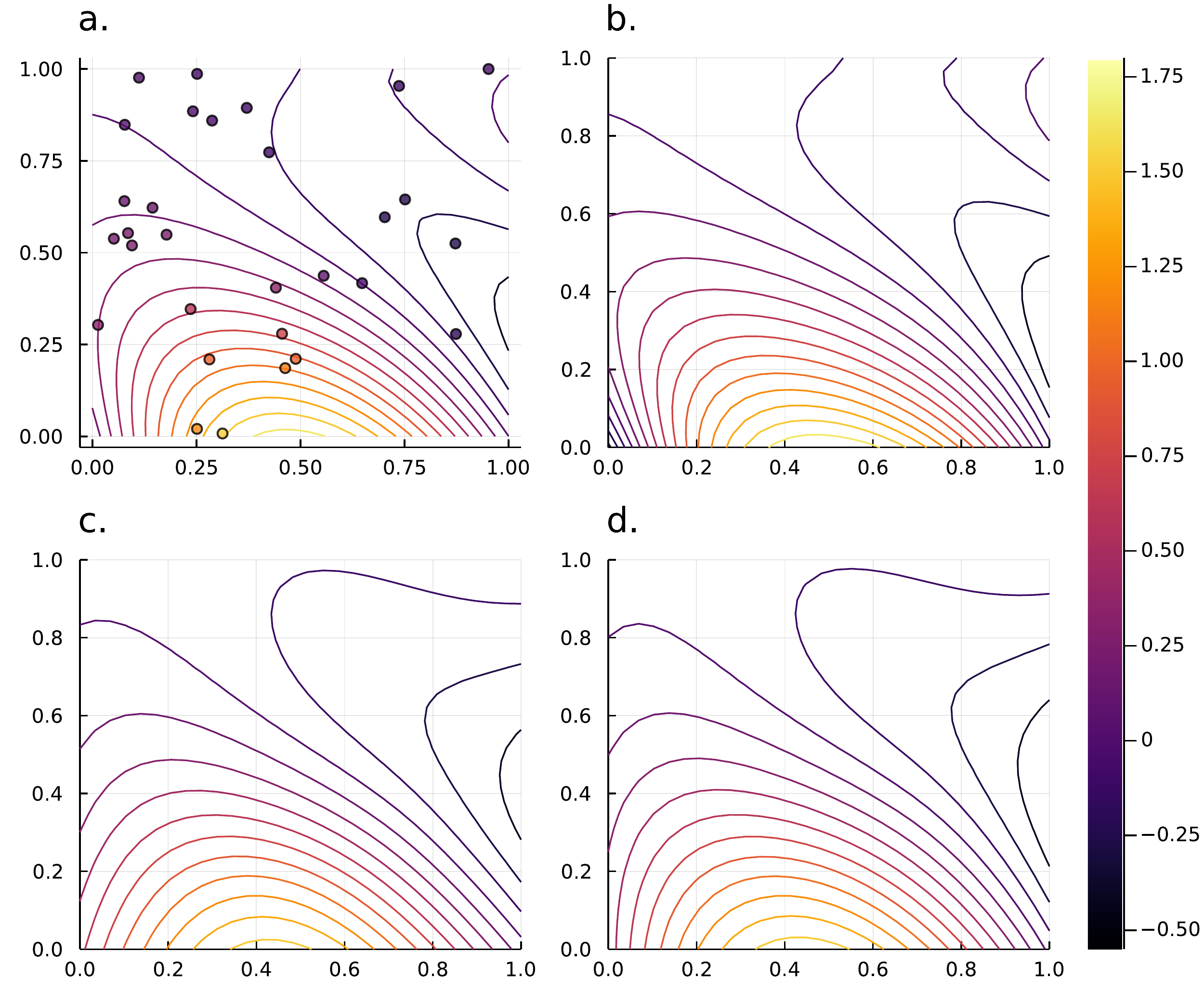}
  \caption{Gaussian process fits in $d=2$ and their matched approximations.
  a. Gaussian process fit with a Gaussian kernel (the measurement locations are
  plotted with value $y_i$ as a colour scale) b. Matched approximation for the
  Gaussian kernel. c. Gaussian process fit with a Matérn kernel ($r=3$). d.
  Matched approximation (polyharmonic splines). }
  \label{fig:numresults-2d}
\end{figure}

While the matched approximation may be surprinsingly accurate close to
the measurement locations, polynomials and polyharmonic splines generally
diverge as $\norm{\bx} \rightarrow \infty$, unlike GP models, which return to a
baseline of 0. Consequently, the matched approximations are very inaccurate far
from the data, as shown in fig. \ref{fig:numresults-2d-zoomed-out} , which is just a zoomed-out version of fig.
\ref{fig:numresults-2d}. There are ways of tapering the matched approximation
to prevent divergence, but we leave the details for future work. 

\begin{figure}
  \centering
  \includegraphics[width=12cm]{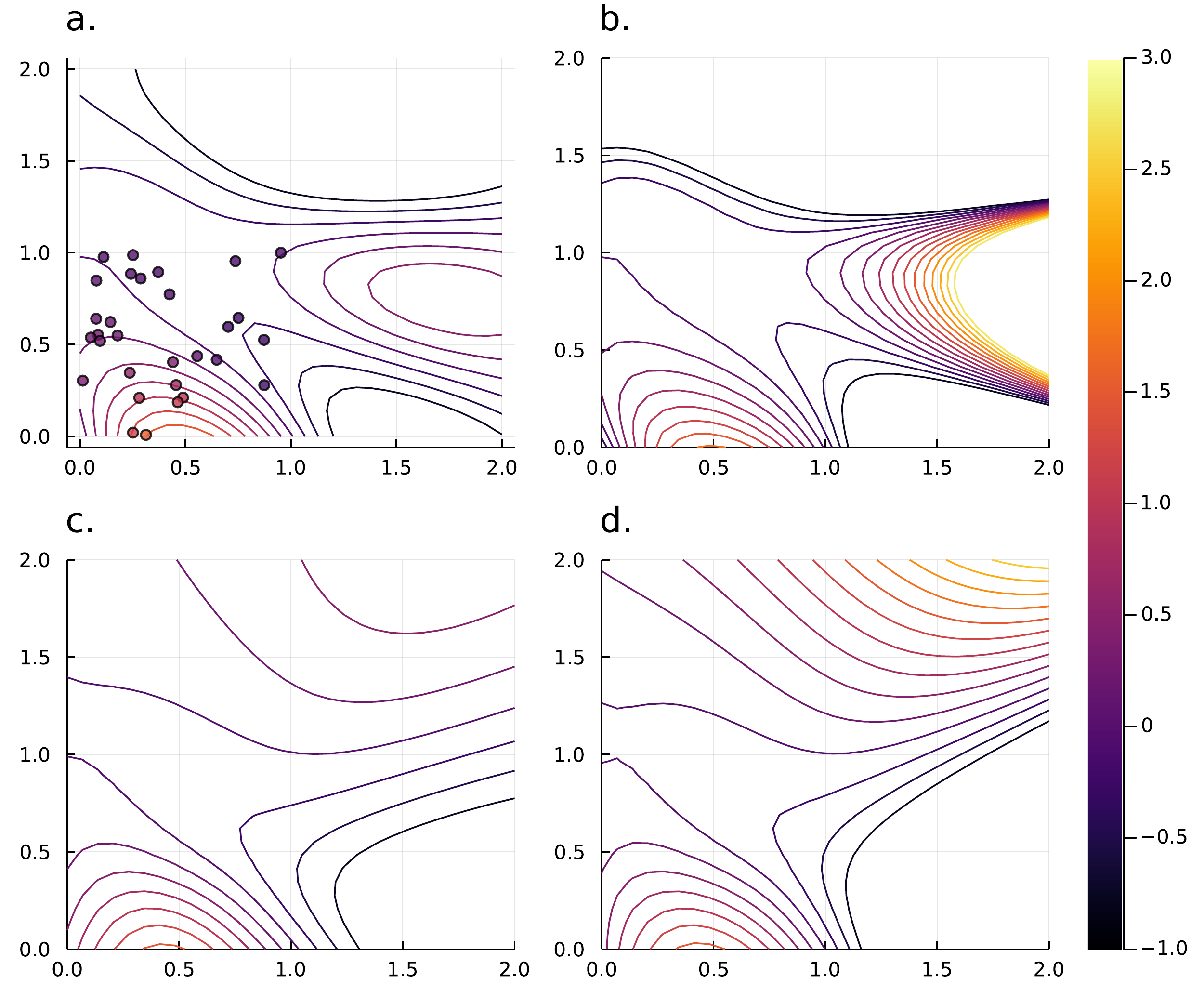}
  
  \caption{Same contents as in fig. \ref{fig:numresults-2d}, but over a broader
    domain, to show the     divergence in the matched approximations.}
  \label{fig:numresults-2d-zoomed-out}
\end{figure}

Finally, in the introduction we described the flat limit in terms of the family of fits,
seen as a parametric curve (parameterised by $\gamma$). The predictions for the
GP with a Gaussian kernel ``go through'' the polynomial predictions in the
limit. The same holds true in the multidimensional case, as per theorem
\ref{thm:equivalent-kernels-nd}. For some appropriate value of $\gamma$, the
prediction of the GP will come to match that of the model
$\SPM{0}{\sum_{|\vect{\alpha}| \leq k} \vect{x}^{\vect{\alpha}}}$, a multivariate polynomial model of
degree $k$. We show this on fig. \ref{fig:asymptotic-curve-gaussian}, which is similar to fig.
\ref{fig:pred-curve-gaussian}: the prediction of the model at locations
$\vect{x}_a = (0.2,0.1)$ and $\vect{x}_b=(0.8,0.8)$ are plotted for different
values of $\gamma$ and fixed $\varepsilon$.

For kernels with finite regularity index, theorem
\ref{thm:equivalent-kernels-nd} shows that the behaviour in the flat limit
depends on $\gamma$. For low values of $\gamma$, they behave like polynomial
models. For high values, like polyharmonic splines. This is the behaviour that
appears on fig. \ref{fig:asymptotic-curve-matern}. 

\begin{figure}
  \centering
  \includegraphics[width=10cm]{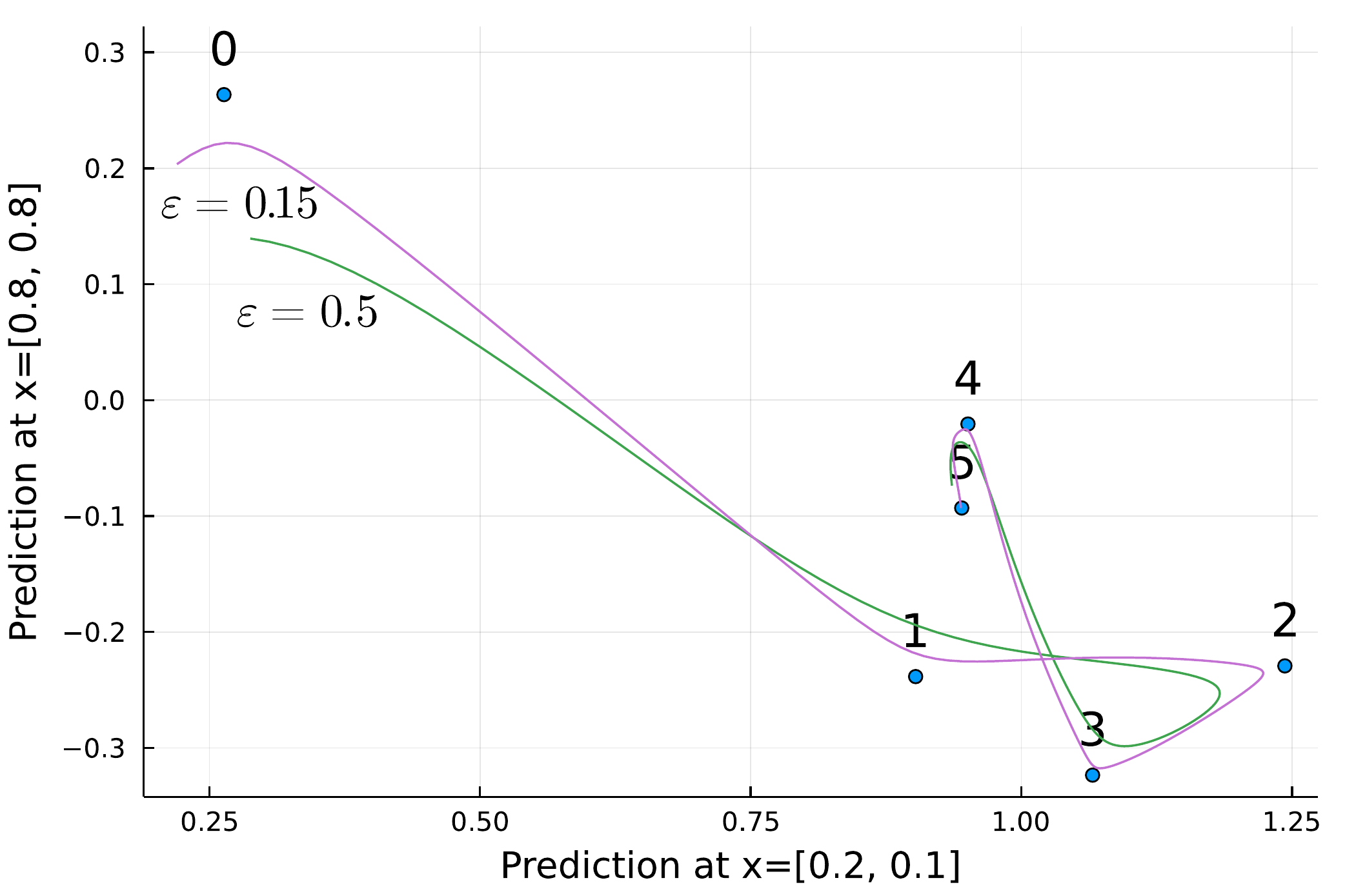}
  \caption{Prediction of GP model (with Gaussian kernel) vs. polynomial models.
    We show the predictions of the model for the data shown on fig.
    \ref{fig:numresults-2d} for a range of values of  $\gamma$ (continous
    curves), and two different values of $\varepsilon$. The labelled points are
    the prediction of the polynomial models of degree $0$ to $5$. As per theorem
  \ref{thm:equivalent-kernels-nd}, as $\flatlim$, the continuous curve must
  go through the labelled points.}
  \label{fig:asymptotic-curve-gaussian}
\end{figure}

\begin{figure}
  \centering
  \includegraphics[width=10cm]{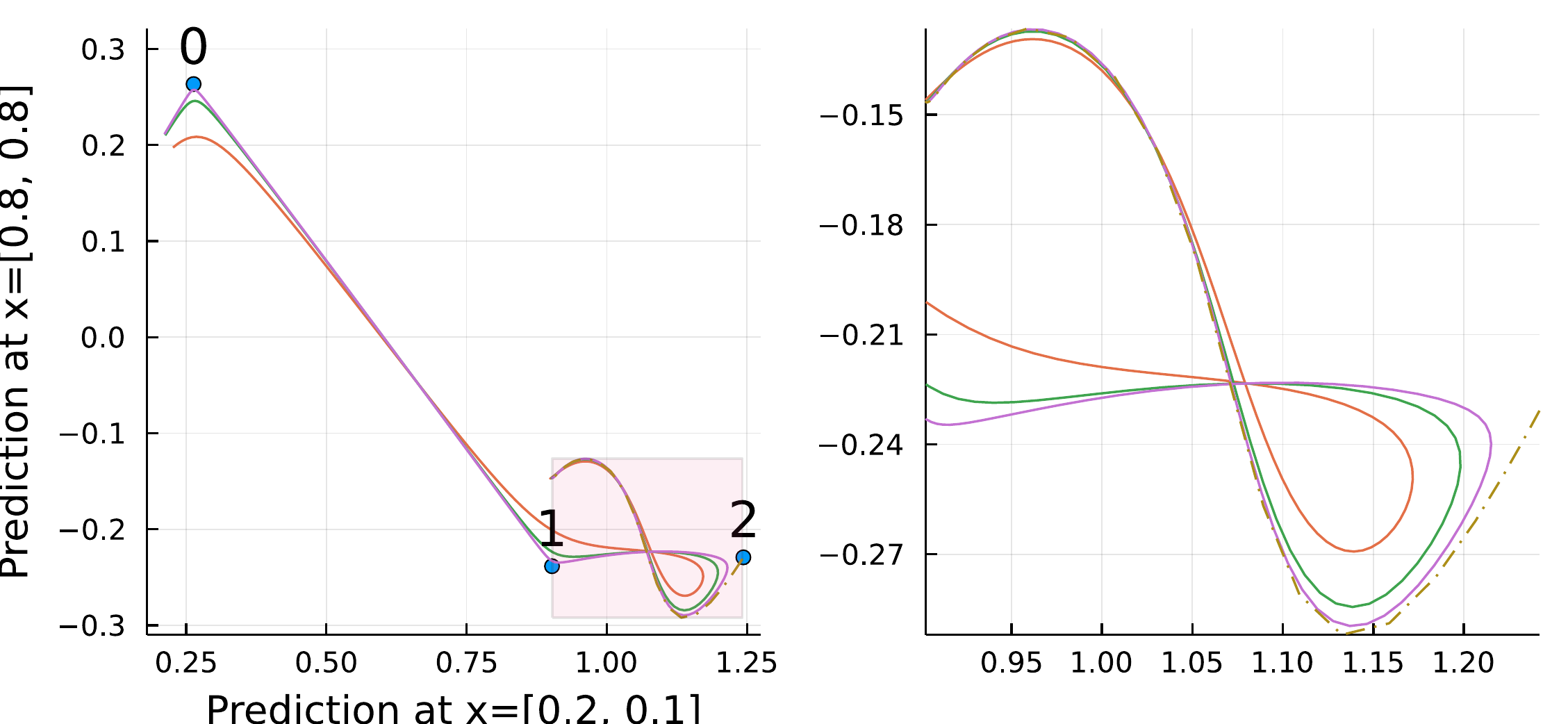}
  \caption{Prediction of GP model (with Matérn covariance, $r=3$) vs. polynomial
    models and polyharmonic spline kernels.     The figure uses the same data and principle as in
    fig. \ref{fig:asymptotic-curve-gaussian}, but according to theorem
    \ref{thm:equivalent-kernels-nd}, this GP model matches polynomials of degree
    up to 2 at low $\gamma$, and polyharmonic splines at higher values of
    $\gamma$. Accordingly, we show the polynomial predictions for degree up to
    two (labelled points). The three curves are for decreasing values of
    $\varepsilon$, $\varepsilon=0.5$ is in red, $\varepsilon=0.15$ in green and
    $\varepsilon=0.05$ in purple. The panel to the right shows a zoomed-in
    version of the rectangle highlighted in the first panel. The dotted curve
    corresponds to the polyharmonic splines. }
  \label{fig:asymptotic-curve-matern}
\end{figure}

\section{Conclusion}
\label{sec:conclusion}

The flat limit of Gaussian process regression highlights the very strong
connections GP methods share with classical methods like polynomial regression
and smoothing splines. The fact that, at least in certain cases, the flat limit
gives a very good approximation for large values of $\varepsilon$ shows that it may
be useful in practice once the limits of the approximation are better
understood. 

We conclude with some open questions and directions. First, while smoothing
splines in $d=1$ can be implemented at cost $\O(n)$ \cite{reinsch1967smoothing},
polyharmonic splines in $d>1$ have cost $\O(n^3)$. On the other hand, for Matérn
models with low regularity coefficient, there exist efficient (approximate)
methods based on a stochastic PDE formulation \cite{lindgren2011explicit}. Since
such GP models have polyharmonic splines as their flat limit, this suggests that
stochastic PDE methods should be applicable.

Finally, if a tractable ``sharp limit'' ($\varepsilon \rightarrow \infty$)
expansion were available, there might be a way of finding good 
approximations that work over a broad range of values of $\varepsilon$, for
instance via matched asymptotic approximations. Such an approximation would be
both interesting theoretically and practically useful. 

\section{Appendix}
\label{sec:proofs}

\subsection{Semi-parametric models as limits}
\label{sec:semi-param-models-as-limits}

In this section we introduce SPMs  as
a limit (we do not claim that this is particularly original).
This section parallels section 4 in \cite{tremblay2021extended}. Readers
familiar with DPPs may be interested to note that extended L-ensembles are to
semi-parametric GPs what L-ensembles are to GPs, see \cite{fanuel2020determinantal}
\footnote{Sampling measurement
  locations from the appropriate extended L-ensemble guarantees for instance
  that the posterior distribution is proper (integrable). }. 

Our definition of regression with semi-parametric Gaussian fields is as follows: let
$l(x,y)$ be a kernel (not necessarily positive definite, as we will see), and
$v_1(x),\ldots,v_r(y)$ a set of basis functions. Then semi-parametric GP regression is
just GP regression with the kernel 
\[ \kappa_\varepsilon(x,y) = l(x,y) + \varepsilon^{-1} \sum_i^p v_i(x)v_i(y) \]
in the limit $\flatlim$. Even though the prior variance goes to infinity along
some directions, the posterior distribution is generally well-defined, and
quantities like the smoother matrix tend to finite limits. 
Although the construction naturally works for $l(x,y)$ positive definite, recall
that this
is not a requirement, and $l(x,y)$ may have negative eigenvalues, so long as
they align with the subspace spanned by the $v_i$'s.

We introduce some notation, borrowed from \cite{tremblay2021extended}, that will
be used throughout this section. A non-negative pair is the discrete counterpart
to a SPM $\SPM{l}{\V}$. 
\begin{definition}
	\label{def:nnp}
  A Nonnegative Pair, noted $\NNP{\bL}{\bV}$ is a pair $\bL \in \R^{n \times n}$, $\bV \in \R^{n
    \times p}$, $0\leq p\leq n$, such that $\bL$ is symmetric and conditionally positive semi-definite with respect to
  $\bV$, and $\bV$ has full column rank. Wherever a NNP $\NNP{\bL}{\bV}$ appears below, we consistently use the following notation: 
\begin{itemize}
\item $\bQ \in \R^{n \times p}$ is an
  orthonormal basis of  $\mspan \bV$, such that $\bI - \bQ\bQ^\top$
  is a projector on $\orth \bV$
\item $\widetilde{\bL} = (\bI - \bQ\bQ^\top)\bL(\bI -
  \bQ\bQ^\top)\in \R^{n \times n}$ is also symmetric and thus diagonalisable. From  \cite[Prop. 2.3]{tremblay2021extended} we know that all its eigenvalues are non-negative. We will denote by $q$ the rank of $\widetilde{\bL}$. Note that $q \le n-p$ as the $p$ columns of $\bQ$ are trivially eigenvectors of $\widetilde{\bL}$ associated to $0$. We write
  \[
    \widetilde{\bL} = \tbU\tbLam\tbU^\top
  \] 
  its truncated spectral decomposition; where $\tbLam=\text{diag}(\widetilde{\lambda}_1,\ldots,\widetilde{\lambda}_q)\in \mathbb{R}^{q\times q}$ and $\tbU \in \mathbb{R}^{n\times q}$ are the diagonal matrix of nonzero eigenvalues and the matrix of the corresponding eigenvectors of $\widetilde{\bL}$, respectively.
\end{itemize}
\end{definition}
Saddle-point systems feature prominently in our formulas:
\begin{definition}
  The saddle-point system associated with a NNP $\NNP{\bL}{\bV}$ is the
  $(2n)\times (n+p)$ matrix:
  \[ S(\bL,\bV) = 
    \begin{pmatrix}
      \bL & \bV \\
      \bV^\top & \matr{0}
    \end{pmatrix}
  \]
\end{definition}
It has the same form as the system that appears in polyharmonic  spline interpolation, and this
is no accident. 
Our first step will be to find the limit of the smoother matrix, which here
reads:
\begin{equation}
  \label{eq:smoother-semi-parametric}
  \bM_\varepsilon = (\bL + \varepsilon^{-1}\bV\bV^\top)(\bL + \sigma^2\bI + \varepsilon^{-1}\bV\bV^\top)^{-1} = (\varepsilon\bL + \bV\bV^\top)\left( \varepsilon (\bL + \sigma^2\bI) +\bV\bV^\top \right)^{-1}
\end{equation}
To do so we use matrix perturbation theory (treating $\varepsilon \bL$ as a perturbation), and specifically the 
approach of \cite{avrachenkov2013analytic}. The difficulty lies in dealing with
$\left( \varepsilon (\bL + \sigma^2\bI) +\bV\bV^\top \right)^{-1}$, which is divergent as
$\flatlim$ since $\bV\bV^\top$ is not invertible. Because of that, it does not
admit a power series. However, it does admit a Laurent series, which is an
expansion involving negative orders of $\varepsilon$. We do not need the theory
developed in \cite{avrachenkov2013analytic} in its full generality for our
purposes here. We introduce a simplified version tailored to our needs.
\begin{theorem}
  \label{thm:laurent-expansion}
  Let $\bA(\varepsilon) = \bV\bV^\top + \varepsilon \bC$, invertible for
  $\varepsilon > 0$, with $\bC \in \R^{n\times
  n }$ symmetric and $\bV \in \R^{n \times p}$,  $p<n$, so that $\bV\bV^\top$ is non-invertible. Then:
  \begin{equation}
    \label{eq:laurent-expansion}
    \bA(\varepsilon)^{-1} = \varepsilon^{-1} (\bB_0 + \varepsilon \bB_1 + \varepsilon^2 \bB_2 + \ldots)
  \end{equation}
  This is a Laurent expansion around $\varepsilon=0$, and its terms $\bB_0,
  \bB_1, \ldots $ are the solutions of the  
  following equation, called the ``master equation'':
  \begin{align*}
    \bV\bV^\top\bB_0 &= \matr{0} \\
    \bV\bV^\top \bB_1 + \bC \bB_0 &= \bI \\ 
    \bV\bV^\top \bB_2 + \bC \bB_1 &= \matr{0} \\
    \bV\bV^\top \bB_3 + \bC \bB_2 &= \matr{0} \\
    \vdots
  \end{align*}
  or equivalently:
  \begin{equation}
    \label{eq:master-eq-compact}
    \bV\bV^\top \bB_i + \bC \bB_{i-1} = \delta_{i,1} \bI 
  \end{equation}
  for $i$ going from $1$ to $\infty$.
  In addition, all the terms $\bB_0,\bB_1,\ldots$ are symmetric. 
\end{theorem}
The proof can be found in \cite{avrachenkov2013analytic} but straightforward to
sketch. The existence of the Laurent expansion (eq.
\eqref{eq:laurent-expansion}) is a consequence of Cramer's rule. The master
equation is obtained by plugging eq. 
\eqref{eq:laurent-expansion} into $\bA \bA^{-1} = \bI$ and matching terms by
order. 

Using the Laurent expansion we find:
\begin{corollary}
  \label{cor:limit-smoother-matrix-semipar}
  Let $\NNP{\bL}{\bV}$ a NNP. The smoother matrix $\bM_\varepsilon$ has the following expansion in
  small $\varepsilon$:
  \begin{eqnarray}
    \label{eq:smoother-semi-parametric-lim}
    \bM_\varepsilon &= &\bQ \bQ^\top + \tbU \tbLam \left( \tbLam + \sigma^2 \bI \right)^{-1} \tbU^\top + \O(\varepsilon) \\
    &=&\bQ \bQ^\top + \tbL \left( \tbL + \sigma^2 \bI \right)^{-1} + \O(\varepsilon)
  \end{eqnarray}
\end{corollary}
\begin{proof}
  We use theorem \ref{thm:laurent-expansion} with $\bC = \bL + \sigma^2 \bI$ in
  eq. \eqref{eq:smoother-semi-parametric}, which gives
  \begin{equation}
    \label{eq:smoother-semi-parametric-exp1}
    \bM_\varepsilon = (\varepsilon\bL + \bV\bV^\top)\left( \varepsilon (\bL + \sigma^2 \bI) +\bV\bV^\top \right)^{-1} = \varepsilon^{-1} (\varepsilon\bL + \bV\bV^\top)(\bB_0 + \varepsilon \bB_1 + \ldots)
  \end{equation}
  where $\bB_0$ and $\bB_1$ verify the master equation:
  \begin{align*}
    \bV\bV^\top\bB_0 &= \matr{0} \\
    \bV\bV^\top\bB_1 + (\bL + \sigma^2 \bI)\bB_0 &= \matr{I} 
  \end{align*}
  Expanding in eq. \eqref{eq:smoother-semi-parametric-exp1}, we find:
  \begin{equation}
    \label{eq:smoother-semi-parametric-exp1}
    \bM_\varepsilon = \varepsilon^{-1}(\bV\bV^\top \bB_0) + (\bL \bB_0 + \bV\bV^\top \bB_1) + \O(\varepsilon)
  \end{equation}
  The diverging term $\bV\bV^\top \bB_0$ is null by the master
  equation. Again by the master equation, the constant-order term equals $\bI -
  \sigma^2 \bB_0$.
  We now solve for $\bB_0$. Note that $\bV\bV^\top\bB_0 = \matr{0}$ implies $\bV^\top
  \bB_0=\matr{0}$ ($\bV$ is $n$ times $p$ and has full rank).
  Therefore, $\bB_0$ is orthogonal to $\mspan \bV$ and we may express $\bB_0$ in
  a basis that spans the complement of $\mspan \bV$. Recall that the notation we
  introduced, $\bQ$ is an orthogonal basis for $\bV$, $\tbL  = (\bI - \bQ\bQ^\top)\bL(\bI -
  \bQ\bQ^\top)\in \R^{n \times n}$, and so $\tbL$ lies in the complement of
  $\mspan \bV$. The (non-null) eigenvectors of $\tbU$ of $\tbL$ may therefore be taken as
  a basis for the complement of $\mspan \bV$, and we have that $\bB_0 = \tbU
  \bZ_0 $ for some matrix $\bZ_0$.
  Inserting this form into the master equation (in the second term), we have:
  \begin{equation*}
    \bV\bV^\top\bB_1 + (\bL + \sigma^2 \bI) \tbU \bZ_0 = \matr{I} 
  \end{equation*}
  Multiplying to the left by $\tbU^\top$, we have:
  \begin{align*}
    & \tbU^\top (\bL + \sigma^2 \bI) \tbU \bZ_0  = \tbU^\top \\
     \iff & \bZ_0 = \left( \tbU^\top (\bL + \sigma^2 \bI) \tbU  \right)^{-1} \tbU^\top \\
     \iff & \bB_0 = \tbU \left( \tbU^\top (\bL + \sigma^2 \bI) \tbU  \right)^{-1} \tbU^\top = \left( \tbU  \tbU^\top (\bL + \sigma^2 \bI) \tbU \tbU^\top  \right)^{\dag} 
  \end{align*}
  by analogy with $\tbL = \left( \tbU  \tbU^\top \bL \tbU \tbU^\top  \right)^{\dag} $,
  we have  $\bB_0 = (\widetilde{\bL + \sigma^2 \bI})^\dag = (\widetilde{\bL} + \sigma^2 \tbU \tbU^\top)^\dag  =\tbU (\tbLam +\sigma^2\bI)^{-1} \tbU^\top$. Inserting
  this result in eq. \eqref{eq:smoother-semi-parametric-exp1}, we obtain
$    \bM_\varepsilon = \bI - \sigma^2\tbU (\tbLam +\sigma^2\bI)^{-1} \tbU^\top  + \O(\varepsilon) $.
Then observe that $\bM_\varepsilon$ diagonalises in $(\bQ ; \tbU )$ to finally obtain:
   \begin{equation*}
    \bM_\varepsilon = \bQ \bQ^\top + \sigma^2\tbU \tbLam(\tbLam +\sigma^2\bI)^{-1} \tbU^\top  + \O(\varepsilon)
  \end{equation*}
  hence completing the proof.
\end{proof}

The expression for the smoother matrix has a simple interpretation: anything in
the span of $\bV$ goes through unpenalised (for instance, constant and linear
trends), and the rest is penalised in the usual way. This fits in with the
``semi-parametric regression'' interpretation.

Next, we examine the conditional expectation at an unobserved location:
\begin{corollary}
  \label{cor:variance-sp}
  The conditional expectation $ \E(f(x) | \vect{y})$ has the following expansion
  in the semi-parametric limit:
  \begin{equation}
    \label{eq:conditional-exp-semi-parametric-app}
    \E(f(x) | \vect{y} ) =
    \begin{pmatrix} \vect{l}_{x,\X} & \vect{v}_x \end{pmatrix}
    \saddle{\bL+\sigma^2\bI}{\bV}^{-1}
    \begin{pmatrix} \vect{y} \\ \vect{0} \end{pmatrix} + \O(\varepsilon)
  \end{equation}
\end{corollary}
\begin{proof}
  We start with:
  \[ \E(f(x) | \vect{y} ) = (\vect{l}_{x,X} + \varepsilon^{-1} \vect{v}_x \bV^\top
    ) (\bL + \sigma^2 \bI + \varepsilon^{-1} \bV\bV^\top)^{-1} \vect{y} \]
  and insert the Laurent expansion as previously, to obtain:
  \begin{equation}
    \label{eq:cond-exp-semi-parametric-1}
    \E(f(x) | \vect{y} ) = \varepsilon^{-1} ( \vect{v}_x \bV^\top \bB_0 \vect{y})
    + \left( \vect{l}_{x,\X}\bB_0  +  \vect{v}_x \bV^\top \bB_1 \right) \vect{y} +
    \O(\varepsilon)
  \end{equation}
  As previously, the diverging term disappears since $\bV^\top\bB_0=0$. 
  The master equation for $\bB_0$ can be rewritten as:
  \begin{equation}
    \label{eq:master-eq-1}
    \saddle{\bV + \sigma^2 \bI}{\bV}
    \begin{pmatrix}
      \bB_0 \\
      \bV^\top \bB_1
    \end{pmatrix}
    =
    \begin{pmatrix}
      \bI \\
      \matr{0}
    \end{pmatrix}
  \end{equation}
  We may rewrite eq. \eqref{eq:cond-exp-semi-parametric-1} as:
  \[ \E(f(x) | \vect{y} ) =
    \begin{pmatrix}
      \vect{l}_{x,\X} & \vect{v}_x
    \end{pmatrix}
    \begin{pmatrix}
      \bB_0 \\
      \bV^\top \bB_1 
    \end{pmatrix}
\vect{y} +
    \O(\varepsilon)\]
  and eq. \ref{eq:conditional-exp-semi-parametric} is obtained by solving eq.
  \ref{eq:master-eq-1} and injecting the result.
\end{proof}
\begin{remark}
  By setting $\sigma^2=0$ in the equation, we recover the interpolation case.
  One may also verify that setting $x \in \X$ recovers a column of the smoother
  matrix.
  In addition, eq. \ref{eq:conditional-exp-semi-parametric} implies that the
  function $\hat{f} = E(f|\by)$ belongs in $\flatlim$ to a specific function
  space:
  \begin{equation}
    \label{eq:cond-exp-semi-par-RKHS}
    \hat{f}(x) = \sum_{i=1}^n \beta_i l(x,x_i) + \sum_{j=1}^p \alpha_j v_j(x) + \O(\varepsilon)
  \end{equation}
  This looks at first sight like a function space of dimension $n+p$ but by eq.
  (\ref{eq:conditional-exp-semi-parametric}) $\bV^\top\vect{\beta} = 0$, which
  removes $p$ degrees of freedom. The first term corresponds to the
  non-parametric part, the second to the parametric part. The spline basis of
  eq. \eqref{eq:spline-space} is a special case of this general form. 
\end{remark}

Finally, we may also obtain the asymptotic predictive variance using the same
technique (although it requires going a step further in the master equation):
\begin{corollary}
  The conditional expectation $ \Var(f(x) | \vect{y})$ has the following
  expansion in the semi-parametric limit:
  \begin{equation}
    \label{eq:conditional-var-semi-parametric-app}
    \Var(f(x) | \vect{y} ) = \vect{l}_{x,\X} - 
    \begin{pmatrix} \vect{l}_{x,\X} & \vect{v}_x \end{pmatrix}
    \saddle{\bL+\sigma^2\bI}{\bV}^{-1}
    \begin{pmatrix} \vect{l}_{x,\X} \\ \vect{v}_x^\top \end{pmatrix} + \O(\varepsilon)
  \end{equation}
\end{corollary}
\begin{proof}
    We follow the same steps as above, starting with:
  \begin{align*}
\Var(f(x) | \vect{y} ) &= \vect{l}_{x,x} + \varepsilon^{-1} \vect{v}_x \vect{v}_x^\top +
    (\vect{l}_{x,\X} + \varepsilon^{-1}  \vect{v}_x\bV^\top)
    (\bL+\sigma^2\bI + \varepsilon^{-1} \bV\bV^\top)^{-1}
                         (\vect{l}_{\X,x} + \varepsilon^{-1}  \bV\vect{v}_x^\top) \\
                       &= \vect{l}_{x,x} + \varepsilon^{-1} \vect{v}_x \vect{v}_x^\top + \varepsilon^{-2}
                         (\varepsilon \vect{l}_{x,\X} +  \vect{v}_x\bV^\top)
                         (\bB_0 + \varepsilon \bB_1 + \varepsilon \bB_2 + \ldots)
                         (\varepsilon \vect{l}_{\X,x} +   \bV\vect{v}_x^\top)
  \end{align*}
  We now extract the terms in the expansion, starting with the lowest valuation:
  \[ [\varepsilon^{-2} ] \Var(f(x) | \vect{y} ) = - \vect{v}_x\bV^\top \bB_0 \bV
    \vect{v}_x^\top \]
  This term is zero by the master equation.
  The next order is:
  \begin{equation}
    \label{eq:intrisic-var-epsilon-min1}
    [\varepsilon^{-1} ] \Var(f(x) | \vect{y} ) = \vect{v}_x\vect{v}_x^\top  - \vect{v}_x\bV^\top \bB_1 \bV \vect{v}_x^\top  - 2 \vect{l}_{x,\X} \bB_0 \bV \vect{v}_x^\top
  \end{equation}
  The master equation implies $\bB_0 \bV = 0$, so the last term drops out.
  We also have $\bB_0(\bL + \sigma^2\bI) + \bB_1 \bV\bV^\top = \bI $, and multiplying
  to the left by $\bV\bV^\top$ yields $\bV\bV^\top \bB_1 \bV\bV^\top = \bV\bV^\top$. Let
  $\vect{\alpha} $ any vector such that $\vect{\alpha} \bV  = \vect{v}_x$.
  Then
  \[ \vect{\alpha} \bV\bV^\top \bB_1 \bV\bV^\top \vect{\alpha}^\top = \vect{v}_x \bV^\top
    \bB_1 \bV \vect{v}_x = \vect{v}_x\vect{v}_x^\top \]
  This shows that $[\varepsilon^{-1} ] \Var(f(x) | \vect{y} ) = 0$, and
  therefore that the conditional variance is not divergent as $\flatlim$. 
  We now compute the constant-order term:
  \begin{equation}
    \label{eq:intrisic-var-epsilon-0}
    [\varepsilon^{0} ] \Var(f(x) | \vect{y} ) = l_{x,x}  - \vect{v}_x\bV^\top \bB_2 \bV \vect{v}_x^\top  - 2 \vect{l}_{x,\X} \bB_1 \bV \vect{v}_x^\top - \vect{l}_{x,\X} \bB_0  \vect{l}_{\X,x}
  \end{equation}
  The next order in the master equation is:
  \[ \bV\bV^\top \bB_2 + (\bL + \sigma^2\bI)\bB_1 = 0\]
  Multiplying to the right by $\bV \bV^\top$, we have:
  \[ \bV\bV^\top \bB_2 \bV\bV^\top =  - (\bL + \sigma^2\bI)\bB_1 \bV\bV^\top\]
  Again using $\vect{\alpha}$ such that $\vect{\alpha} \bV  = \vect{v}_x$, we
  obtain 
  \[ \vect{v}_x \bV^\top \bB_2 \bV \vect{v}_x^\top =  -\vect{\alpha} (\bL + \sigma^2\bI)\bB_1 \bV\vect{v}_x^\top\]
  We now inject this result in eq. \eqref{eq:intrisic-var-epsilon-0}, and
  re-express it in the following form:
  \begin{align*}
    [\varepsilon^{0} ] \Var(f(x) | \vect{y} )  &= l_{x,x}  + \vect{\alpha}
    (\bL+\sigma^2\bI)\bB_1 \bV\vect{v}_x^\top  - 2 \vect{l}_{x,\X} \bB_1 \bV \vect{v}_x^\top -
                                                \vect{l}_{x,\X} \bB_0  \vect{l}_{\X,x} 
  \end{align*}
 We already note that the master equation for $\bB_0$ is equivalent to \ref{eq:master-eq-1}. Solving for $\bB_0$ and $ \bV^\top \bB_1$
 is easy and gives in particular $\bV^\top \bB_1= (\bV^\top \left(\bV^\top (\bL+\sigma^2 \bI)^{-1} \bV\right)^{-1} \bV^\top  (\bL+\sigma^2 \bI)^{-1}$. Thus the term 
 containing $\vect{\alpha}$ equals $ \vect{\alpha}
    (\bL+\sigma^2\bI)\bB_1 \bV\vect{v}_x^\top = \vect{v}_x \left(\bV^\top (\bL+\sigma^2 \bI)^{-1} \bV\right)^{-1} \vect{v}_x^\top$ .
   
Collecting the different terms, we can write  
    \begin{align}
    \label{AlmostFinalEq}
    [\varepsilon^{0} ] \Var(f(x) | \vect{y} ) &= l_{x,x}  +  
    \begin{pmatrix} \vect{l}_{x,\X} & \vect{v}_x \end{pmatrix}
     \begin{pmatrix} 
     \bB_0 & \bB_1 \bV  \\
     \bV^\top\bB_1 & -\left(\bV^\top (\bL+\sigma^2 \bI)^{-1} \bV\right)^{-1} 
     \end{pmatrix}
    \begin{pmatrix} \vect{l}_{x,\X} \\ \vect{v}_x^\top \end{pmatrix} 
    \end{align}
Now, using block inverse formula allows to show that if $\bA \in \R^{n\times n}$ is an invertible matrix  $\bB\in \R^{n\times p}, p\leq n$ is full rank, then
  \begin{eqnarray*}
    \saddle{\bA}{\bB}^{-1} = 
    \begin{pmatrix} 
      \bA^{-1} - \bA^{-1} \bB (\bB^\top \bA^{-1} \bB)^{-1} \bB^\top \bA^{-1} & \bA^{-1} \bB (\bB^\top \bA^{-1} \bB)^{-1}   \\
      (\bB^\top \bA^{-1} \bB)^{-1}\bB^\top \bA^{-1} & - (\bB^\top \bA^{-1} \bB)^{-1}
    \end{pmatrix}
  \end{eqnarray*}
 Using this result, we observe  that the matrix involved in Eq. (\ref{AlmostFinalEq})  is the inverse of $$\saddle{\bL+\sigma^2\bI}{\bV}, $$
   hence finishing the proof. 
\end{proof}
\begin{remark}
It is interesting to develop the inverse of the saddle point matrix,. The {\it a posteriori} mean  then reads
\begin{equation}
    \E(f(x) | \vect{y} ) =  \vect{v}_{x}  \beta +  \vect{l}_{x,\X}(\bL+\sigma^2 \bI)^{-1} \left( \vect{y}  - \bV \beta \right)
  \end{equation}
  where $\beta= \left(\bV^\top (\bL+\sigma^2 \bI)^{-1} \bV\right)^{-1} \bV^\top(\bL+\sigma^2 \bI)^{-1}\vect{y} $. The estimation is performed after the polynomial trend has been removed. Note that the preceding results corresponds to the Bayesian approach in which $g$ admits a Gaussian prior and $\beta$s are chosen Gaussian with an infinite variance,  an uninformative prior which however leads to a proper posterior (see {\it e.g.} \cite{Gu1992}). 
\end{remark}


\subsection{Proof of theorem \ref{thm:equivalent-kernels-1d}}
\label{sec:proof-of-univar-result}

As stated in the proof sketch, most of the proof consists in working out what the
limiting smoother matrices are. These results can be found implicitly in   
\cite{barthelme2023determinantal}. The approach we use here is much more direct,
however, and hopefully easier to follow. 

We  start by recalling results from \cite{BarthelmeUsevich:KernelsFlatLimit}
on eigenvalues and eigenvectors in the flat limit. 
There are two essential facts to keep in mind. One is that most eigenvalues of
kernel matrices go to 0 in the flat limit, but they do so at different speeds.
The other is that the eigenvectors go to orthogonal polynomials or splines,
depending on the regularity of the kernel and the magnitude of the associated
eigenvalues.

\subsubsection{Asymptotics of the eigenvalues}
\label{sec:eigenvals-1d}

In the cases we examine, the smoother matrix reads:
\[ M_\varepsilon = \bK_\varepsilon(\bK_\varepsilon + \frac{\sigma^2}{\gamma} \bI
  )^{-1} = \bU_\varepsilon
  \diag(\frac{\lambda_i(\varepsilon)}{\lambda_i(\varepsilon) +
    \frac{\sigma^2}{\gamma}}) \bU_\varepsilon^\top\]

We will review the behaviour of the eigenvectors $\bU_\varepsilon$ later. An
interpretation that is helpful to keep in mind is that the smoother matrix acts
like a filter: the measurement $\vect{y}$ is transformed to the eigenbasis,
then (by analogy with the Fourier transform) each discrete  ``frequency'' is
scaled by $\frac{\lambda_i(\varepsilon)}{\lambda_i(\varepsilon) +
  \frac{\sigma^2}{\gamma}}$, after which the data is transformed back to its
original space. The function $s(\lambda) = \frac{\lambda}{\lambda+\frac{\sigma^2}{\gamma}}$
is analogous to a filter response function. It is an increasing function of
$\lambda$, and maps $\R^+$ to $[0,1]$. Notably, if $\lambda$ is small
compared to $\frac{\sigma^2}{\gamma}$, then $s(\lambda) \approx 0$, and if it is
large $s(\lambda) \approx 1$. In a nutshell, what happens in the flat limit is
that eigenvalues grow apart (by orders of magnitude), so that the eigenvalues
separate into three groups. A first group is
much larger than $\frac{\sigma^2}{\gamma}$, and these have $s(\lambda) \approx
1$; a second  has approximately the same magnitude as $\frac{\sigma^2}{\gamma}$,
and finally a third group is much smaller, and have $s(\lambda) \approx 0$. That
explains why GPs behave in the flat limit like semiparametric regression (like
intrinsic GPs): some directions in $\vect{y}$ go through the smoother matrix
unchanged, some are penalised, and some are clamped down to 0. 

The asymptotics of the eigenvalues of kernel matrices in the flat limit are as
follows. If the kernel matrix is analytic in $\varepsilon$, then the eigenvalues can also
be written as analytic functions\footnote{The requirement that
  $\bK_\varepsilon$ be analytic is probably artificial, see \cite{akian2016non}} of $\varepsilon$, {\it i.e.} $\lambda_i(\varepsilon) =
\varepsilon^{\nu_i} (\lt_i + \O(\varepsilon))$. The
coefficient $\nu_i$ is the valuation of the $i$-th eigenvalue, and gives the rate
at which it vanishes as $\flatlim$. In a log-log plot of eigenvalues versus
$\varepsilon$, it defines the limiting slope. 
When $d=1$, the valuation of the eigenvalues is given by the following result
(from \cite{BarthelmeUsevich:KernelsFlatLimit})
\[ \nu_i =
  \begin{cases}
    2(i - 1) \mathrm{\ if\ } i \leq r \\
    2r - 1 \mathrm{\ otherwise}
  \end{cases}
\]
Note the dependency on $r$, the smoothness order of the kernel. Two extreme
cases are the Gaussian kernel ($r=+\infty$), for which the eigenvalues are
$\O(1),\O(\varepsilon^2),\O(\varepsilon^4),\O(\varepsilon^6),\ldots$ and the
exponential kernel, which has $r=1$ and eigenvalues that are
$\O(1),\O(\varepsilon),\O(\varepsilon),\O(\varepsilon),\ldots$. Kernels with
$r>1$ behave like the Gaussian kernel for the first $r$ eigenvalues, then the
next $n-r$ eigenvalues are all of the same order. 

Now consider the asymptotics of terms of the form $
\frac{\lambda_i}{\lambda_i+\frac{\sigma^2}{\gamma}}$, as they appear in the
spectral form of the smoother matrix. If $\gamma$ is constant as a function of
$\varepsilon$, then, for all kernels with $r\leq 1$,

\[ \frac{\lambda_1}{\lambda_1+\frac{\sigma^2}{\gamma}} =
  \frac{\lt_1}{\lt_1+\frac{\sigma^2}{\gamma}} + \O(\varepsilon) \]
and for all $i > 1$
\[ \frac{\lambda_i}{\lambda_i+\frac{\sigma^2}{\gamma}} =
  \O(\varepsilon) \]
All but the first term go to 0 in $\flatlim$. In the filtering interpretation,
that means the smoothing matrix will only let through the part of $\vect{y}$
that is proportional to the first
eigenvector and everything else will be clamped down to zero. The associated
eigenvector is the constant vector, so that the output of the smoother is a
constant function. In addition, there will be some regularisation, given by
\[ 0 < \frac{\lt_1}{\lt_1+\frac{\sigma^2}{\gamma}} < 1 \]
The smoother matrix for constant $\gamma$ then becomes effectively the smoother
matrix for a (penalised) polynomial regression of degree 0. 

Thus, taking the limit $\flatlim$ while keeping $\gamma$ fixed does not lead to
very interesting results, since the GP fit tends in that case to a
constant function. Precise examination of the asymptotics (more on which below)
leads to the conclusion that $\gamma$ must scale as $\varepsilon^{-p}$ for the
number of degrees of freedom to stay constant as $\varepsilon$.

As an example, we may take $\gamma=\gt \varepsilon^{-1}$. There are two cases we
need to distinguish: $r=1$ and $r>1$. If $r>1$, one may check that
$\frac{\lambda_i(\varepsilon)}{\lambda_i(\varepsilon)+\frac{\sigma^2}{\gamma(\varepsilon)}}$ goes to 1 for $i=1$, and goes to 0
for $i>1$. Thus, the fit will now correspond to an \emph{unpenalised} polynomial
regression of degree 0. If $r=1$, then
$\frac{\lambda_i(\varepsilon)}{\lambda_i(\varepsilon)+\frac{\sigma^2}{\gamma(\varepsilon)}}$
goes to 1 for $i=1$, then all
subsequent terms (from $2$ to $n$) equal
$\frac{\lt_i}{\lt_i+\frac{\sigma^2}{\gt}} + \O(\varepsilon) $. From the
filtering point of view, the first eigenvector goes through unpenalised, then
everything goes through with a penalty. This is the signature of a
semiparametric model, and indeed it is. As we will see later, in this case the
parametric part is the constant function, and the non-parametric part is made up
of linear splines. 

The general pattern of the results has much in common with what we have seen so
far. $\gamma$ needs to rise as $\flatlim$, and by controlling the ratio $\gamma
/ \varepsilon$, we control how many eigenvectors go through unpenalised,
penalised, or not at all. Asymptotically we need $\gamma = \gt \varepsilon^{-p}$,
and the asymptotics will depend on the parity of $p$. 


In the theorems stated here we take $\gamma=\varepsilon^{-p} \tilde{\gamma}$.
The first result concerns the smoother matrix and follows directly from the
spectral asymptotics in \cite{BarthelmeUsevich:KernelsFlatLimit}. We introduce  $\widetilde{\bD^{(2r-1)}}$ the matrix $\bD^{(2r-1)} = \left[ \norm{\bx_i -
    \bx_j}^{2r-1} \right]_{i,j}$ with monomials of degree $< r$ projected out, {\it i.e.} obtained using a QR decomposition.
\begin{equation}
  \label{eq:Dtilde}
 \widetilde{\bD^{(2r-1)}} = (\bI - \bQ_{< r} \bQ_{< r}^\top)\bD^{(2r-1)}(\bI - \bQ_{< r} \bQ_{< r}^\top)
\end{equation}
$\bQ_{< r}$ being and orthogonal basis for $\bV_{< r}$, {\it e. g.} obtained using a QR decomposition. We can state:
\begin{theorem}
\label{eigenflat1d:th}
  Let $\kappa$ a stationary kernel with regularity order $r \in \mathbb{N}^+$,
  given $n$ observations on the real line  $\{
  x_1,\ldots,x_n \}$ the smoother matrix
  $\bM_\varepsilon = \gamma \bK (\gamma \bK + \sigma^2 \bI)^{-1}$ has the
  following expansion as $\flatlim$:
  $\bM_\varepsilon = \bU \bF_\varepsilon \bU^\top + \O(\varepsilon)$
  where $\bU$ is the matrix of limiting eigenvectors and $\bF_\varepsilon$ is a diagonal
  (``filter'') matrix with entries
  \begin{equation}
    \label{eq:filter-matrix}
    {\bF_{\varepsilon}}_{ i,i} =
    \begin{cases}
   \displaystyle    \frac{\lt_i}{\lt_i + \gt^{-1} \sigma^2
        \varepsilon^{p-2 (i-1)}} & \mbox{\ if\ } i \leq r \\
    \displaystyle   \frac{\lt_i}{\lt_i + \gt^{-1} \sigma^2
        \varepsilon^{p-(2r-1)}} & \mbox{ otherwise}
    \end{cases}
  \end{equation}
  The limiting eigenvectors are similarly partitioned as:
  \[ \bU = \begin{bmatrix}
      \bQ_{\leq \min(n,r)} & \tilde{\bU}
    \end{bmatrix} \]
  where $\bV = \bQ\bR$ is the QR decomposition of $\bV$ and $\widetilde{\bU}$ are
  the eigenvectors of $\widetilde{\bD^{(2r-1)}}$, as defined above. If $r>n$ then
  $\bU = \bQ$.
\end{theorem}
\begin{proof}
  This follows directly from noting that (by Rellich's theorem) $\gamma \bK
  (\gamma \bK + \sigma^2 I) ^ {-1} =
  \bU(\varepsilon) \bF_{\epsilon}\bU(\varepsilon)^\top $ with
  $\bU(\varepsilon)$ analytic, and
  $ {\bF_{\varepsilon}}_{ i,i}=\frac{\lambda_i(\varepsilon)}{\lambda_i(\epsilon)+\frac{\sigma^2}{\gamma(\varepsilon)}}$
  and filling in the results of \cite{BarthelmeUsevich:KernelsFlatLimit}. 
\end{proof}

Note that the theorem includes the case $r \geq n$. In this case, the result is independent of the regularity parameter $r$ and is equivalent to the infinite smooth case. 

We can now study the flat limit when $\varepsilon$ goes to 0.

\subsubsection{Smooth case, or $r>n$}
\label{proof1dsmooth:sssec}

In this case all the eigenvalues are $O(\varepsilon^{2(i-1)})$. 
 Thus we get :
 \begin{itemize}
\item $p\geq 2n -1 $, then $ {\bF_{\varepsilon}}_{ i,i} \longrightarrow_{\varepsilon \rightarrow 0} 1, \forall i=1,\ldots, n$, $\matr{M}_\varepsilon \longrightarrow \bI$ and $\Tr(\matr{M}_\varepsilon) \longrightarrow n$.

This corresponds to the polynomial interpolation of $n$ points by a polynomial of order $n-1$. 

\item $p$ odd and $p < 2n -1 $ : 
$ {\bF_{\varepsilon}}_{ i,i}\longrightarrow_{\varepsilon \rightarrow 0} 1, \forall i=1,\ldots, l=\lfloor p/2+1 \rfloor$, 0 otherwise and thus $\matr{M}_\varepsilon \longrightarrow \bQ_{\leq l} \bQ_{\leq l}^\top $ and $\Tr(\matr{M}_\varepsilon) \longrightarrow l $.

This corresponds to a (unpenalised) regression by a polynomial of order $l-1$.


\item  $p$ even  and $p < 2n-1$ :  : $ {\bF_{\varepsilon}}_{ i,i} \longrightarrow_{\varepsilon \rightarrow 0} 1, \forall i=1,\ldots, l= p/2$, 
$\matr{\Delta}_{l+1,l+1} =  \frac{\lt_{l+1}}{\lt_{l+1}+\frac{\sigma^2}{\gt}}$, 0 otherwise. 
Thus $\matr{M}_\varepsilon \longrightarrow \bQ_{\leq l} \bQ_{\leq l}^\top + \frac{\lt_{p/2+1}}{\lt_{p/2+1}+\frac{\sigma^2}{\gt}} \vect{q}_{p/2+1}\vect{q}_{p/2+1}^\top $ and $\Tr(\matr{M}_\varepsilon) \longrightarrow p/2+ \frac{\lt_{p/2+1}}{\lt_{p/2+1}+\frac{\sigma^2}{\gt}}$.

This corresponds to a penalised polynomial regression  (the order is controlled by a balance between the observation noise and the importance (as measured by the eigenvalue) of a higher order monomial). 


\end{itemize}

\subsubsection{Non-smooth case or $r<n$}
\label{proof1dnonsmooth:sssec}

In the non-smooth case, all eigenvalues have order at most $2r-1$ in
$\varepsilon$. 
\begin{itemize}
\item for $p>2r-1$, we obtain an interpolation, since all eigenvalues of the
  smoother matrix go to 1. 

\item for $p=2r-1$, $ {\bF_{\varepsilon}}_{ i,i} \longrightarrow_{\varepsilon \rightarrow 0} 1, \forall i=1,\ldots, r$ and $\frac{\lt_{i}}{\lt_{i}+\frac{\sigma^2}{\gt}} , \forall i=r+1,\ldots, n$.

Thus, $\matr{M}_\varepsilon \longrightarrow \bQ_{\leq r} \bQ_{\leq r}^\top + \tilde{\bU}  {\bF_{\varepsilon}}_{\geq r+1}( {\bF_{\varepsilon}}_{\geq r+1} +\frac{\sigma^2}{\gt}\bI)^{-1} \widetilde{\bU}^\top$

Here we obtain the smoothing splines solutions, for which the splines of order $r$ are added to the polynomial regression of order $r-1$. 


\item for $p< 2r-1$, we recover the two last cases of the smooth-case above, depending on the parity of $p$. 

\end{itemize}

\begin{remark}
  Some examples will clarify the meaning of this result. For the Gaussian
  kernel, $r=\infty$ and so for every odd $r$ the smoother matrix goes to
  $\bQ_{\leq l} \bQ_{\leq l}^\top $ ($l = (r-1)/2$), the smoother matrix of a polynomial regression of degree $l$.
  For the exponential kernel, $r=1$, and so for $p=1$ the smoother matrix goes
  to the smoother matrix of a polyharmonic spline regression of degree $1$. For $p$ larger than
  1 all eigenvalues go to 1 and so the limit is an interpolation. 
\end{remark}

\subsubsection{Final step of the proof, equivalent asymptotic models}

The final step of the proof uses the results of section
\ref{sec:prediction-equivalence}. The limit of smoother matrices, which is
always of the form $\bM_\varepsilon=\bM_0+\O(\varepsilon)$. Examining the
proof of lemma \ref{lem:equivalence-smoother} and proposition
\ref{prop:pred-eq-smoother-semipar}, we see that the predictive mean and
variance depend smoothly on $\bM_\varepsilon$ and so we have asymptotic
predictive equivalence in the sense of definition
\ref{def:asymptotic-pred-eq}. Then we obtain, using the results of the two preceding sections, and restricting to the the case $r<n$:
\begin{itemize}
\item  $p$  even: The model  converges to $(l_p ; {\cal V}_p)$ where 
the kernel associated is  $l_p(x,y)=x^{p/2+1} y^{p/2+1}$ and the basis functions given by the monomials up to order $p/2$, or 
${\cal V}_p=(1,x,x^2,\ldots,x^{p/2})$.
\item  $p$ odd:
The model  converges in this case to $(l_p ; {\cal V}_p)$ where 
the kernel associated is  $l_p(x,y)=0 $ and the basis functions given by the monomials up to order $\lfloor p/2+1 \rfloor$, or 
${\cal V}_p=(1,x,x^2,\ldots,x^{\lfloor p/2+1 \rfloor})$.
\item $p=2r-1$: 
The model   converges to $(l_p ; {\cal V}_p)$ where 
the kernel associated is  $l_p(x,y)=(-1)^r | x-y |^{2r-1}$ and the basis functions given by the monomials up to order $r-1$, or 
${\cal V}_p=(1,x,x^2,\ldots,x^{r-1})$.

\end{itemize}
This concludes the proof.

%

\subsection{Wronskians}

\subsubsection{Wronskian matrices from spectral representation}
\label{sec:wronskian-spectral}

In this section we derive some properties of the Wronskian matrices of
stationary kernel functions from their spectral representation. For stationary
kernels, we may write $k(\vect{x},\vect{y}) = \kappa(\vect{x}-\vect{y})$.

Bochner's theorem \cite{stein1999interpolation} tells us that $\kappa$ may be written as:
\begin{equation}
  \label{eq:Bochners-th}
  \kappa(\vect{x}-\vect{y}) = \int_{\R^d} \exp\left(\im \vo^\top(\vect{x}-\vect{y})\right) d\mu(\vo)
\end{equation}
where $\mu$ is a positive measure (the spectral measure). In fact, there is a
one-to-one correspondence between measures and kernel functions. Assume further
that $\mu$ has a density ({\it i.e.} is absolutely continous relative to the Lebesgue
measure on $\R^d$). Then eq. \eqref{eq:Bochners-th} can be rewritten as:
\begin{equation}
  \label{eq:Bochners-th-density}
  \kappa(\vect{x}-\vect{y}) = \int_{\R^d} \exp\left(\im \vo^\top(\vect{x}-\vect{y})\right) q(\vo) d \vo
\end{equation}
where $q = \mu'$ is the density corresponding to the spectral measure $\mu$,
called the spectral density. Eq. \eqref{eq:Bochners-th-density} implies that in
this case $\kappa$ and $q$ are Fourier transform pairs. Since integrable
functions have continuous Fourier transforms, if $\kappa$ is integrable it has a
spectral density. 

We follow the normalisation convention of \cite{stein1999interpolation} and note:
\begin{equation}
  \label{eq:spectral-density}
  q(\vo) = \frac{1}{(2\pi)^d}\int_{\R^d} \exp \left(-\im \vo^\top\vect{x} \right) \kappa(\vect{x}) d \vect{x}
\end{equation}
the spectral density of the kernel.
The spectral representation of $k$ lets us link the Wronskian matrix to moments
of $q$. Partial derivatives $ k^{(\va, \vb)}$ of $k$ may be computed as:
\begin{align*}
  \frac{\partial^{ |\va|+ |\vb|}}{\partial\vect{x}^{\va} \partial \vect{y}^{\vb} } k(\vect{x},\vect{y}) &=
                                                                                                                                                       \frac{\partial^{ |\va|+ |\vb|}}{\partial\vect{x}^{\va} \partial \vect{y}^{\vb} } \int _{\R^d} \exp(\im \vo^\top(\vect{x}-\vect{y})) d \mu(\vo)   \\
  & = \int_{\R^d} \im^{|\va|+ |\vb|}(-1)^{|\vb |} \vo^{\va+ \vb} \exp(\im \vo^\top(\vect{x}-\vect{y})) d \mu(\vo)
\end{align*}
Thus, the ($\va$,$\vb$) element of the Wronskian matrix equals:
\begin{align*}
  \matr{W}_{\va,\vb} &= \frac{1}{\va! \vb!}  k^{(\va, \vb)}(0,0) \\
                          &= \frac{1}{\va! \vb!}  \int_{\R^d} \im^{|\va|+ |\vb|}(-1)^{\vb} \vo^{\va+ \vb} d \mu(\vo) \\
  &= \frac{1}{\va! \vb!}  \im^{|\va|+|\vb|}(-1)^{|\vb |} \E_{\mu} \left( \vo^{\va+ \vb} \right)
\end{align*}
where $\E_{\mu}$ designates expectation under $q$ (note that $q$ is not normalised).
In general, not all moments of $q$ exist, which is an equivalent way of stating
that not all derivatives exist \cite{stein1999interpolation}. The fact that $q$
is symmetric around the origin immediately yields that $\E_{\mu}\left(
  \vo^{\va+ \vb} \right) \neq 0$ if and only if
$|\va+ \vb|$ is even. We obtain the following compact expression:
\begin{equation}
  \label{eq:wronskian-moments}
  W_{\va,\vb} =
  \begin{cases}
     \frac{1}{\va! \vb!}  (-1)^{p+|\vb|} \E_{\mu} \left( \vo^{\va+ \vb} \right) \mathrm{\ if\ } |\va+ \vb|= 2p, p \in \mathbb{Z} \\ 
 0 \mathrm{\ otherwise }
  \end{cases}
\end{equation}
\begin{remark}
  In certain cases further simplification is possible. If the kernel function is
separable (which is the case for instance for the squared-exponential kernel),
then $q$ is a product distribution: $q(\vo) = \prod_{j=1}^d p(\omega_j)$, and 
$\E_{\mu}(\vo^{\va+ \vb}) = \prod_{j=1}^d
\E_p(\omega_j^{\alpha_j+\beta_j})$. In the special case of Matérn kernels, $q$
is a multivariate Student's t distribution, and a simple expression for the
moments can be found in \cite{kotz2004multivariate}. 

\end{remark}
We now use eq. \eqref{eq:wronskian-moments} to prove that Wronskians are
positive-definite. The positive-definiteness is strict as long as the spectral
density exists.

\begin{lemma}
  Let $\bW_m = [W_{\va,\vb}]_{\va,\vb \in
    \mathbb{P}_{m}}$ the Wronskian (of order $m$) of a stationary kernel of
  order $r \geq m$. Then $\bW_m \geq 0$. Further, if the kernel has a spectral
  density then $\bW_m > 0$.
\end{lemma}
\begin{proof}
  We will show that $\bW_m$ is a Gram matrix, which implies
  positive-definiteness. Consider the following sequence of functions:
  $\Psi_{\gamma}(\vo)= \frac{1}{\vect{\gamma}!} (\im
  \vo)^{\vect{\gamma}}$ where $\vect{\gamma}$ runs over
  $\mathbb{P}_m$. We define a dot product from $q$:
  \begin{equation*}
    <\phi,\eta> = \int_{\R^d}  \phi(\vo) \overline{\eta(\vo)} d \mu(\vo)
  \end{equation*}
  Then:
  \begin{align*}
    <\psi_{\va},\psi_{\vb}> &= \int  \frac{1}{\va!} (\im  \vo)^{\va}  \overline{ \frac{1}{\vb!} (\im  \vo)^{\vb} } d \mu(\vo) \\
                                               &= \frac{1}{\va! \vb!} (-1)^{|\vb|} \int  (\im)^{|\va|+|\vb|} \vo^{\va+\vb} d \mu(\vo) \\
                                               &= W_{\va,\vb}
  \end{align*}
which verifies that $\matr{W}$ is indeed a Gram matrix, and implies $\matr{W}
\geq 0$. To go further and prove positive-definiteness, we now assume that the
spectral density exists. Consider the quadratic form
$\vect{p}^\top\matr{W}\vect{p}$, with $\vect{p} \neq 0$:
\begin{align*}
  \vect{p}^\top\matr{W}\vect{p} &= \sum_{\va,\vb} p_{\va} W_{\va,\vb}p_{\vb}  \\
                             &=    \sum p_{\va} <\psi_{\va},\psi_{\vb}> p_{\vb} \\
                             &=  < \sum p_{\va} \psi_{\va},\sum p_{\vb} \psi_{\vb}> \\
                             &= \norm{\sum p_{\va} \psi_{\va}}_\mu^2 \numberthis \label{eq:quad-form-as-norm}
\end{align*}
 where $\norm{f}_\mu^2$ is the norm induced by the dot product we have defined,
 {\it i.e.}:

 \begin{equation}
   \label{eq:norm-mu}
   \norm{f}_\mu^2 = \int_{\R^d} f(\vo)\overline{f(\vo)}q(\vo)d\vo
 \end{equation}
Obviously, $\norm{f}_\mu^2 = 0$ if and only if $f$ vanishes almost everywhere on
the support of $\mu$. 

Recall that the set of functions $\psi_{\va}$ are a subset of the monomials (up
to scaling). In eq. \eqref{eq:quad-form-as-norm} $\eta_p(\vo) = \sum p_{\va}
\psi_{\va}(\vo)$ is a complex-valued polynomial of degree
$m$, so $\norm{\eta_p(\vo)}_\mu^2 \neq 0$ unless
$\eta_p(\vo) = 0$ almost everywhere on the support of $\mu$. Since $\mu$ is absolutely
continuous w.r.t the Lebesgue measure on $\R^d$, the support of $\mu$ is a
$d$-dimensional subset of $\R^d$. Polynomials in $\R^d$ can only vanish on a
subspace of dimension less than $d$,
 so
$\norm{\eta_p(\vo)}_\mu^2 > 0$ for all non-zero $\vect{p}$.

\end{proof}

\subsubsection{Expressions for the Wronskian in some special cases}
\label{sec:wronskian-special-cases}

The Wronskian for the squared exponential kernel can be worked out directly from eq.
\eqref{eq:wronskian-moments} and known formulas for Gaussian moments.
The spectral density of the squared exponential kernel equals:
\begin{equation}
  \label{eq:spectral-dens-gaussian}
  q(\vo) = \frac{1}{(2\pi)^d}\int_{\R^d} \exp \left(-\im \vo^\top\vect{x} \right) \exp(-\norm{\vect{x}}^2) d \vect{x}
   = \frac{1}{(2\sqrt{\pi})^d} \exp(-\frac{\norm{\vo}^2}{4})
\end{equation}
which is the (normalised) density of a $d$-dimensional  Gaussian vector with independent entries of
 variance $\sigma^2=2$. The $p$-th moment of a univariate, centered Gaussian
variate $z$ equals $\E(z^p) = \sigma^p (p-1)!!$ if $p$ is even, and 0 otherwise.
Here $n!!$ designates the so-called ``double factorial'', which if $n$ is odd,
equals the product $n(n-2)(n-4) \ldots 1$. Since $q$ is separable, we have:
\begin{equation}
  \label{eq:gaussian-moments}
  \E_{\mu}(\omega^{\vect{\gamma}}) =
  \begin{cases}
    \prod_{i=1}^d 2^{\gamma_i/2} (\gamma_i-1)!! \mathrm{\ if \ } \gamma_1, \ldots, \gamma_d \mathrm{\ even} \\
    0 \mathrm{\ otherwise}
  \end{cases}
\end{equation}

Injecting eq.
\eqref{eq:spectral-dens-gaussian} into eq. \eqref{eq:wronskian-moments}, we
obtain:
\begin{equation}
  \label{eq:gaussian-wronskian}
  W_{\va,\vb} = 
  \begin{cases}
   \frac{\left(  \va+\vb-1\right)!! }{\va!\vb!}  (-2)^{\frac{\va+\vb}{2}} (-1)^{\vb}  \mathrm{\ if \ } \alpha_1+\beta_1, \ldots, \alpha_d+\beta_d \mathrm{\ even} \\
    0 \mathrm{\ otherwise}
  \end{cases}
\end{equation}

\subsubsection{The Wronskian matrix and orthogonal polynomials}
\label{sec:wronskian-orth-poly}

Let $h_0, \ldots, h_s$ denote the first $s$ orthogonal polynomials of the
spectral measure. There is a close relationship between the orthogonal
polynomials and the matrix of moments:
\begin{align*}
  \E(\vo^{\va+\vb}) &= \E(\vo^{\va}\vo^{\vb}) \\
                    &= \E \left( \left(\sum_{|\vg|\leq |\va|} A_{\vg,\va} h_{\vg}(\vo) \right) \left(\sum_{|\vg'|\leq |\vb|} A_{\vg',\vb} h_{\vg'}(\vo) \right) \right)\\
                    &= \sum_{\vb,\vb'} A_{\vg,\vb} A_{\vg',\vb} \E( h_{\vg} h_{\vg'}) \\
                    &= \sum_{\vb} A_{\vg,\va} A_{\vg,\vb} \E( h_{\vg}^2 ) \\
                    &= (\matr{A}\matr{H} \matr{A}^\top)_{\va,\vb}
                      \numberthis   \label{eq:moments-orth-poly}
\end{align*}
which gives an LDLt decomposition of the matrix of moments (H is diagonal and
its diagonal values are the energies of the orthogonal polynomials). Using the same
trick as in the previous section, and defining $B_{\va,\vg} =
\frac{(\im)^{\va}}{\va!}  A_{\va,\vg}$, we find the LDL* decomposition of $\bW$,
specifically:
\begin{equation}
  \label{eq:ldl_decomp_wronskian}
  \bW = \matr{B} \matr{H} \matr{B}^*
\end{equation}

In certain cases the elements of $\matr{B}$ are available in closed-form. For
example in the squared-exponential case the orthogonal polynomials of the
measure are the Hermite polynomials.

\subsubsection{Wronskian matrix of Matérn kernels}
\label{sec:wronskian-matern}
In this section we consider Matérn kernels with regularity parameter $\nu $,
with spectral density as in \cite{williams2006gaussian}, p.84:
\begin{equation}
  \label{eq:Matern-spd}
  q(\vo) = \frac{(2\sqrt{\pi})^d\Gamma(\nu+d/2)(2\nu)^\nu}{\Gamma(\nu)}\left( 2\nu+4\pi^2 \norm{\vo}^2 \right)^{-\nu-d/2}
\end{equation}
One may verify that this is also the density of a multivariate $t$-distribution,
with degrees of freedom $2\nu$. The $t$-distribution has finite moments of order
$2\nu-1$, so that the order of regularity $r$ of the kernel equals $2\nu$. For
instance, if $\nu=1/2$, then $q(\vo)$ is integrable but has no other
finite moments exist, meaning equivalently the kernel function is not
differentiable. The Matern kernel with $\nu = 1/2$ is actually the exponential
kernel. 
The moments of $q$ are given in \cite{kotz2004multivariate}, p. 11.

The Wronskian matrix for the Matèrn kernels (and others) can be obtained from the expression
for the Gaussian kernel. In the expansion:
\[ \bK(\epsilon) = f_0 \bD^0 + f_2 \epsilon^2 \bD^{(2)} + \dots + f_{2r-1}
  \epsilon^{2r-1} \bD^{2r-1} \]
the coefficients $f_0,f_2,\dots $ depend on the kernel but the distance matrices
do not, and it is the latter that can be expressed in the monomial basis.
For radial kernels it is enough to expand the radial function at 0 to obtain
the correct coefficients.

\subsubsection{Inverse and Schur complements of Wronskian matrices for separable kernels}
\label{sec:schur-compl-wronskians}

Theorem \ref{thm:equivalent-kernels-nd} involves Schur complements in the
Wronskian matrix (eq. \eqref{eq:wronskian-schur}). A result in
\cite{HeltonLasserrePutinar2008} shows that these Schur complements are diagonal for separable kernels. Indeed, rephrased for our need, we have:
\begin{theorem}[Th. 3.1 in \cite{HeltonLasserrePutinar2008}]
  \label{thm:heltonlasserre}
  Consider a random vector $X$ with values in $\R^d$ and with product measure  $p(X)=\prod_i p(X_i)$. Let $M_k\in \R^{{ \cal P}_{k,d}\times {\cal P}_{k,d}}$ be the moment matrix with entries  $M_{\alpha,\beta}=\E_p[X^\alpha X^\beta]$. 
  Note that the maximal degree of the multiindexes considered is $k$. 
  
 Then $M^{-1}_{\alpha,\beta} $ can be different from 0 if and only if $\sum_{i=1}^d \max (\alpha_i,\beta_i) \leq k$.
 
 Alternately, $M^{-1}_{\alpha,\beta} $ is necessarily zero (called a congenital zero in  \cite{HeltonLasserrePutinar2008}) iff
  $\sum_{i=1}^d \max (\alpha_i,\beta_i) > k$
\end{theorem}
As an exemple, consider $d=2$ and $k$ up to 3. The pattern provided by the theorem is depicted in the following matrices: all the entries $\star$ can take different values from zero, all the others are necessarily 0:
\[
M_1^{-1}=\left(\begin{array}{c|c|c c|c c  c}
 & (0,0) & (1,0) & (0,1)  \\\hline 
 (0,0) &\star & \star &\star \\\hline
  (1,0) & & \star & 0   \\
  (0,1) & &  & \star  \\\hline 
  \end{array}\right)
\]
\[
M_2^{-1}=\left(\begin{array}{c|c|c c|c c  c}
 & (0,0) & (1,0) & (0,1) & (2,0) & (1,1) & (0,2) \\\hline 
 (0,0) &\star & \star & \star & \star & \star& \star \\\hline
  (1,0) & & \star & \star & \star & \star & 0  \\
  (0,1) & &  & \star & 0 & \star & \star  \\\hline 
   (2,0) &  &  &  & \star & 0 & 0 \\
  (1,1) & &  &  &  & \star & 0  \\
   (0,2) &  &  &  &  &  & \star 
  \end{array}\right)
\]
\[
M_3^{-1}=\left(\begin{array}{c|c|c c|c c  c|c c c c}
 & (0,0) & (1,0) & (0,1) & (2,0) & (1,1) & (0,2) & (3,0) & (2,1) & (1,2) & (0,3) \\\hline 
 (0,0) &\star & \star & \star & \star & \star& \star & \star& \star& \star& \star\\\hline
  (1,0) & & \star & \star & \star & \star & \star & \star & \star & \star & 0 \\
  (0,1) & &  & \star & \star & \star & \star & 0 & \star & \star &\star \\\hline 
   (2,0) &  &  &  & \star & \star & 0 & \star & \star & 0 & 0 \\
  (1,1) & &  &  &  & \star & \star & 0 & \star & \star & 0 \\
   (0,2) &  &  &  &  &  & \star & 0 & 0 & \star & \star \\\hline 
 (3,0) &  &  &  &  &  &  & \star & 0 & 0 & 0 \\
  (2,1) &  &  & &  &  &  &  & \star & 0 & 0 \\
 (1,2) &  &  &  &  &  &  & &  & \star & 0 \\
 (0,3) &  &  &  &  &  &  &  &  &  & \star
  \end{array}\right)
\]
Since the Schur complements we need are diagonal subblocks $M^{-1}_{\alpha,\alpha}$ for $|\alpha|=k$, Th. 3.1 in \cite{HeltonLasserrePutinar2008} directly shows that these are diagonal! 
This result allows to proove corollary \ref{cor:flat-limit-gaussian}.

\subsubsection{Proof of corollary \ref{cor:flat-limit-gaussian}}
\label{sec:proof-corollary-gaussian}
\sloppy

We just have to proove that the flat limit kernel for the Gaussian kernel rescaled by $\varepsilon^{-p}$ is $\rho_p(\vect{x},\vect{y})=(\vect{x}^\top \vect{y})^{p/2}$ for even $p$. 

From the main result  \ref{thm:equivalent-kernels-nd} we know that the limiting kernel in this case  is 
$l(\bx,\by) = \sum_{|\ba|=m,|\bb|=m} \bar{W}_m(\ba,\bb) \bx^{\ba}\by^{\bb}$.
Since $ \bar{W}_m(\ba,\bb)$ is diagonal this reduces to $l(\bx,\by) = \sum_{|\ba|=m} {\cal H}_{m,d}\bar{W}_m(\ba,\ba) \bx^{\ba}\by^{\ba}$, where we recall that ${\cal H}_{m,d}$ is the number of monomials of degree $m$ in dimension $d$.

$\bar{W}_m(\ba,\ba)$ are the diagonal terms of the Schur complement discussed above. From elementary block inverses lemma, we know that the Schur complement needed is the inverse of the corresponding block in the inverse of $\vect{W}_m$.

Let $\vect{\Delta}_{\va}$ be a diagonal matrix with elements $ \im^{|\va|}/\va!$. Then $\vect{W}_m=\vect{\Delta}_{\va}  \vect{M}_m\vect{\Delta}_{\vb}^*$. Thus  $\vect{W}_m^{-1}=\vect{\Delta}_{\vb}^{-*}  \vect{M}_m^{-1}\vect{\Delta}_{\va}^{-1}$.
Since the block we are interested in is diagonal, we end up with $\bar{W}_m(\ba,\ba)= 1/ \left( \va!^2 \vect{M}_m^{-1}(\va,\va)\right)$. We thus need to calculate the diagonal terms in the inverse of the moment matrix. 

To do so, a trick is to use the orthonornal polynomials associated we the measure at end (here the isotropic Gaussian, see above). Writing the orthonormal polynomials as $p_{\va} = \sum_{\bga \leq \va} d_{\va,\bga} {\bx}^{\bga}$, we obtain
$\E p_{\va} p_{\vb} =\delta_{\va,\vb}=\sum_{\bga \leq \va, \brho \leq \vb} d_{\bga,\brho} \vect{M}_m (\bga,\brho) d_{\vb,\brho}$. Note that the matrix $\vect{D}$ with entries $d_{\va,\vb}$ is lower triangular (orthogonality) and invertible. 
Therefore, $I= \vect{D}  \vect{M}_m \vect{D}^\top$ so that $ \vect{M}_m^{-1}=\vect{D}^\top \vect{D}$. Thus, 
$ \vect{M}_m^{-1} (\va,\va) =\sum_{\va\leq\bga} d^2_{\bga,\va}= d_{\va,\va}^2$ since in the block of order $m$ all terms but the diagonal terms are equal to zero. Now, $d_{\va,\va}$ may be found using the norm of the polynomial $p_{\va}$ which writes $\E p_{\va} p_{\va} =1 = \sum_{\bga \leq \va} d_{\va,\bga}\E p_{\va} \bx^{\bga}= d_{\va,\va}\E p_{\va} \bx^{\va}$ again because $d$ is diagonal for the term of maximal degree. The final twist here uses the fact that the measure is a product measure $\mu = \prod \mu_i$, which implies that the orthonormal polynomials $p_{\va} $ associated  are products of the orthonormals polynomials $p_{\alpha_i}$ of the measures $\mu_i$. This implies that $d_{\va,\va}= \prod d_{\alpha_i,\alpha_i}$. In the Gaussian case considered, $d_{\alpha_i,\alpha_i}$ is thus simply the norm of the Hermite polynomial of degree $\alpha_i$.  For our case, $d_{\alpha_i,\alpha_i}^2=1/(2^{\alpha_i} \alpha_i!)$ so that $d_{\va,\va}^2=1/(2^{|\va | } \va!)$ and $\bar{W}_m(\ba,\ba)=2^{|\va | }/\va!$. Thus $l(\bx,\by) \propto \sum_{|\ba|=m} m!/\va! \bx^{\ba}\by^{\ba}=(\bx^\top\by)^m$. Since proportionality here leads to prediction-equivalence, this ends the proof of corollary \ref{cor:flat-limit-gaussian}.

\subsection{Extension to general linear measurements}
\label{sec:general-linear}
Here we sketch the extension of our results to general linear measurements of a
GP, as used in linear inverse problems. As a concrete example, suppose we take
Fourier measurements of an unknown function $f$:
\begin{equation}
  \label{eq:fourier-measurements}
  y_k = \int_0^{1} f(x) e^{2\pi i k x} dx + v_i
\end{equation}
with $v_i \sim \N(0,\sigma^2)$. The goal is to reconstruct $f$ from $y_1 \ldots
y_n$. We do so by assuming $f$ is sampled from a Gaussian process, and
estimating $f(x)$ via its posterior expectation $\E( f(x) | \by)$.

Let us first set up some notation. To generalise beyond eq.
(\ref{eq:fourier-measurements}), we assume that each measurement is the output
of a linear functional acting on $f$:
\begin{equation}
  \label{eq:fourier-measurements}
  y_i = \Braket{ \phi_i | f} + \nu_i
\end{equation}
Here the linear functional may be for instance an integral operator (as in eq.
(\ref{eq:fourier-measurements})), or a differential operator. We recover the
usual setup (point evaluation of $f$) with $\phi_i = \delta_{x_i}$, the Dirac delta at $x_i$. 

Assume that $f \sim GP(0,k)$ with $k(x,y)$ a kernel function. Then one can show \cite{sarkka2011linear}
that the vector $[\Braket{\phi_i,f}]_i$ has a multivariate Gaussian distribution
with mean $\E \Braket{\phi_i | f} = \Braket{\phi_i | \E(f)}) = 0$ and covariance
\begin{equation}
  \label{eq:covariance-gen}
  M_{ij} = \Braket{ \phi_j | \Braket{\phi_i | k }_x}_y = \Braket{\phi_i | k | \phi_j} 
\end{equation}
where $\Braket{\phi | f(x,y)}_x = \Braket{\phi | f(\cdot,y)}$ evaluates the
operator along $x$ with $y$ fixed, and similarly for $\Braket{\phi | f(x,y)}_x$.

Note for later that
\[ \Braket{\phi_i | f + g | \phi_j} = \Braket{\phi_i | f |
    \phi_j}+\Braket{\phi_i | g | \phi_j} \]
and if $h(x,y)=f(x)g(y)$ is a separable function, then
\begin{equation}
  \label{eq:separability}
  \Braket{\phi_i | h  | \phi_j} = \Braket{\phi_i | f }\Braket{g | \phi_j}
\end{equation}

The expectation of $f(x)$ given $\by$ by noting that $f,\by$ are jointly
Gaussian again, and by the usual conditioning formulas
\begin{equation}
  \label{eq:conditional-gen}
  \E(f(x)|\by) = \mm_x^t (\bM+\sigma^2\bI)^{-1} \by
\end{equation}
with $\mm_x^t = \left[ \Braket{\phi_1|k(x,\cdot)} \ldots
  \Braket{\phi_n|k(x,\cdot)} \right]$.

On the other hand, we may also try to solve the problem using polynomials (of
fixed degree $s \leq n$). In this case the natural thing to do is to solve:
\begin{equation}
  \label{eq:solve-invprob-poly}
  \underset{f \in \mathcal{P}_s }{\argmin}{ \sum_i(y_i-\Braket{\phi_i|f})^2}
\end{equation}
where $\mathcal{P}_s$ is the set of polynomials of degree $\leq s$. That is a
simple, finite-dimensional least-squares problem. Noting
\begin{equation}
  \label{eq:invprob-poly}
  S_{ij} = \Braket{\phi_i|x^j}
\end{equation}
we have the estimate
$\ft(x) = \bv_x^t (\bS \bS^t)^{-1}\bS^t \by $ with $\bv_x^t=\left[ x^0 \ldots
  x^s \right]$. 

We can now proceed with the flat limit expansion. For brevity, we do this for
$d=1$ and $r=\infty$ but the general case proceeds along the same lines. 
If we inject the expansion
\begin{equation}
  \label{eq:kernel-exp}
  k_\epsilon(x,y) = \beta_0 + \beta_2 \epsilon^2 (x-y)^2 + \beta_4 \epsilon^4 (x-y)^4 + \ldots
\end{equation}
into the definition of the matrix $\bM$ (eq. (\ref{eq:covariance-gen})) we find
\begin{align*}
  M_{ij}(\epsilon) &=  \Braket{\phi_i | k_\epsilon | \phi_j} \\
                   &= \Braket{\phi_i | \beta_0 + \beta_2 \epsilon^2 (x-y)^2 + \beta_4 \epsilon^4 (x-y)^4 | \phi_j} \\
                   &= \beta_0 \Braket{\phi_i|1|\phi_j} +
                     \epsilon^2 \beta_2\Braket{\phi_i|(x-y)^2|\phi_j}+
                     \epsilon^4 \beta_2\Braket{\phi_i|(x-y)^4|\phi_j}+\ldots
\end{align*}
If we now use the binomial expansion, along with eq. (\ref{eq:separability}), we
find:
\begin{align*}
  \label{eq:kernel-exp-modified}
  M_{ij}(\epsilon) &=  \beta_0 \Braket{\phi_i|1}\Braket{1|\phi_j} +
  \epsilon^2 \beta_2 \left\{ \Braket{\phi_i|x^2}\Braket{\phi_j|1}
    -2\Braket{\phi_i|x}\Braket{\phi_j|y} + \Braket{\phi_i|1}\Braket{\phi_j|y^2}
  \right\} 
  + \ldots \\
  &= \beta_0 S_{i,0}S_{j,0} + \beta_2\epsilon^2 \left\{ S_{i,2}S_{j,0}- 2S_{i,1}S_{j,1}+S_{i,0}S_{j,2} \right\} + \ldots
\end{align*}
This allows us to write $\bM$ in the form
\begin{equation}
  \label{eq:kernel-exp-matrix-form}
  \bM(\epsilon) = \bS\matr{\Delta}(\epsilon)\bW\matr{\Delta}(\epsilon)\bS^t + \ldots
\end{equation}
With this in hand, all the results from \cite{BarthelmeUsevich:KernelsFlatLimit}
follow in a modified
form, with $\bS$ playing the role of the Vandermonde matrix $\bV$. For instance, $\bM(\epsilon)$ will have
eigenvalues of order $0,\epsilon^2,\epsilon^4,\ldots$ and its limiting
eigenvectors are given by Gram-Schmidt applied to $\bS$.

If we now apply the expansion to the conditional mean (eq.
(\ref{eq:conditional-gen})), we find:
\[ \mm_x = \bv_x^t \Ds \bW \Ds \bS + \ldots\]
Assuming $\bS$ has full rank, we may write $\bv_x^t = \bm{\nu}_x^t \bS $. Eq.
(\ref{eq:conditional-gen}) takes the following form
\begin{equation}
  \label{eq:conditional-modified}
  \E(f(x)|\by) = \bm{\nu}_x^t \bS \Ds \bW \Ds \bS^t ( \bS \Ds \bW \Ds \bS^t+\sigma^2\bI)^{-1} \by + \ldots
\end{equation}
where we recognise a regularised inverse (smoother matrix). This generalises the
smoother matrices we have already studied, with $\bS$ generalising $\bV$. By
scaling the kernel matrix appropriately, we recover all the usual results. In
the flat limit of inverse problems, the GP model becomes equivalent to the usual
parametric or semi-parametric models. See section \ref{sec:non-gaussian} for a
detailed example.

\subsection{Extension to non-Gaussian likelihoods}
\label{sub:non-gaussian}
Although the results are entirely analoguous, going beyond Gaussian likelihoods
requires a very substantial change to the proof techniques. In the Gaussian
case, the posterior mean of $f(x)$ given the data $\by$ is a linear function of
$\by$. This allows us to rely on linear algebra for the proofs, but if we now
wish to generalise to non-Gaussian likelihoods this avenue is no longer open. 

Instead, what we may do is directly express posterior expectations as integrals
and use asymptotic formulae for these integrals. In the Gaussian case, we relied
on a reduction (in section \ref{sec:prediction-equivalence}) that let us study
smoothing matrices only. Here we rely on a more general result, in the form of a
Bayesian variant of the representer theorem, that lets us focus on the finite
dimensional posterior for the values of $f$ at the measured location. That
posterior can itself be tackled using asymptotic integrals, as we show.

We give only the barest sketch here, with apologies to the reader. There are
many tedious details to be worked out, in particular in specifying the exact
class of likelihood functions for which the results hold (some form of
regularity is required).

\subsubsection{Assumptions}

An example of a non-Gaussian likelihood is the Bernoulli likelihood introduced
in section \ref{sec:non-gaussian}. Each datapoint $y_i$ is a class label in
$\{0,1\}$ and we assume
\begin{equation}
  \label{eq:probit-model}
  p(y_i = 1 | f(\bx_i)) = \Phi(f(\bx_i))
\end{equation}
where $f$ is our Gaussian
process. If $f(\bx) > 0$ then points at $\bx$ are more likely to be in class
$1$, if $f(\bx) < 0$ it's the opposite and the separating surface is at the
level set $f(\bx) = 0$.

Again we have $\X = \{x_1 \ldots x_n\}$ the measurement locations and $\fX =
[f(x_1) \ldots f(x_n)]$ the vector of function values. The goal of GP
classification is to form predictions at an arbitrary location, from the
posterior predictive distribution $p(f(x) \vert \by)$. To perform prediction, we
compute the posterior expectation of $f(x)$ given the data $\by$, and if that
expectation is positive we predict class $1$. Contrary to the Gaussian case,
computing these expectations cannot be done analytically.

We can assume a model in a more general form, where the measurements $y_i$ at
each location $x_i$ only depend on the value of $f$ at $x_i$, and the
measurements are independent, i.e.:
\begin{equation}
  \label{eq:likelihood}
  p(\by \vert f ) = \prod_i l_i(y_i\vert f(x_i)) = \Ly(\fX)
\end{equation}
We can also extend this to general linear observations on $f$, as in the
previous section, but that requires burdensome notation so we stick to pointwise
evaluation for simplicity.

The gist of the proof is as follows: first, we reduce the problem from looking
at the posterior predictive at $f(x)$ for any $x$ to just looking at the
finite-dimensional posterior $p(\fX\vert\by)$. Next, we compute asymptotics of
the moments of $p(\fX\vert\by)$ essentially by doing multivariate integrals
in the right basis.

\subsubsection{The Bayesian representer theorem}
\label{sec:bayesian-representer}

The paper \cite{csato2002} contains a result which can be seen as 
a Bayesian counterpart of the representer theorem; and indeed you can use it
to prove the classical representer theorem. The representer theorem shows that
a functional (i.e., infinite-dimensional) optimisation problem actually
collapses to a finite-dimensional optimisation problem. The Bayesian representer
theorem shows something similar for the posterior expectation of $f(x)$ given
the data, in a very general setting. 

Consider $\hat{f}(x) = \E(f(x) \vert \by)$ as a function of $x$. We know that
$f$ belongs to the RKHS generated by the kernel, but can we be more specific?
The results in \cite{csato2002} show that we can:
\begin{equation}
  \label{eq:bayesian-rep}
  E(f(x) \vert \by) = \bk(x,\X)\bK^{-1}\E(\fX\vert\by) = \sum_{i=1}^n k(x,x_i) \alpha_i(\by) 
\end{equation}
so that $\hat{f}$ actually belongs to an $n$ dimensional subspace spanned by the
kernel functions $\{ k(x,x_1) \dots k(x,x_n)\}$. The kinship with the classical
representer theorem should be obvious. 

Importantly, by eq. (\ref{eq:bayesian-rep}), we can think of the posterior
expectation at $x$ as an interpolation at $x$ from observations at $\X$ with
value $\E(\fX\vert\by)$ (to see that it is an interpolation, consider $x \in
\X$).
In addition, a similar result holds for conditional variance:
\begin{equation}
  \label{eq:bayesian-rep-var}
  Var(f(x) \vert \by) = \bk(x,\X)\bK^{-1}\Var(\fX\vert\by)
\end{equation}
which again is an interpolation of the posterior variance at the measured
locations. 

Since the flat limit of interpolation is well-characterised, we can focus on the
behaviour of $\E(\fX\vert\by)$.

\subsubsection{Asymptotic integrals}
\label{sec:asymptotic-int}

We would like to find asymptotic formulas for posterior expectations in the flat
limit, of the form
\begin{equation}
  \label{eq:post-exp-eps}
  \E(\fX \vert\by) = \frac{\int \ff \Ly(\ff) \exp(- \frac{1}{2}\ff^t (\gamma_0\epsilon^{-p}\bKe)^{-1} \ff) d\ff}{\int \Ly(\ff) \exp(- \frac{1}{2}\ff^t (\gamma_0\epsilon^{-p}\bKe)^{-1} \ff) d\ff}
\end{equation}

Recall that the role of the $\gamma_0\epsilon^{-p}$ scaling (with $p \leq 0$) is
to get a nontrivial limit. To look at what happens to eq.
(\ref{eq:post-exp-eps}) in $\flatlim$, we will
perform a linear change of variable, after which some dimensions will drop out
and others will simplify.

\subsubsection{Some asymptotic formulas}
\label{sec:asympt-formulas}

Consider the following (univariate) integral:
\begin{equation}
  \label{eq:tilted-gaussian}
  I(\epsilon) = \int s(x) \exp(-\frac{\epsilon^\beta}{2} x^2) dx
\end{equation}
which we seek to evaluate in $\flatlim$ for different values of the exponent
$\beta$. 

If $\beta > 0$, and $s(x)$ is integrable, then
\begin{equation}
  \label{eq:tilted-positive}
  I(\epsilon) = \int s(x) dx + \O(\epsilon^\beta)
\end{equation}
If $\beta < 0$, and $s(x)$ is differentiable at $0$, then
\begin{equation}
  \label{eq:tilted-negative}
  I(\epsilon) = \sqrt{2\epsilon^{-\beta}}(f(0) + \O(\epsilon^{-\beta}))
\end{equation}
which can be obtained from Laplace's method or just by thinking of
\ref{eq:tilted-gaussian} as an expectation under an (unnormalised) Gaussian and
Taylor expanding $f$ at 0.

We can combine these formulas in multivariate integrals, such as
\begin{align}
  J(\epsilon) &= \int f(x,y,z) \exp(-\frac{1}{2}(\epsilon x^2 + y^2 +
                \frac{z^2}{\epsilon}) )dx dy dz \\
              &\approx \int
                \sqrt{2\epsilon^{-\beta}}(f(x,y,0)\exp(-\frac{1}{2}(\epsilon x^2 + y^2 ))dx dy
              &\approx
                \sqrt{2\epsilon^{-\beta}}(f(x,y,0)\exp(-\frac{1}{2}( y^2 ))dx dy
\end{align}

Assume that $f(x,y,z) \leq 0$ represents a likelihood function, and that
$f(x,y,z) \exp(-\frac{1}{2}(\epsilon x^2 + y^2 + \frac{z^2}{\epsilon}) )dx dy
dz$ is a posterior density. We can compute the expectation of $x,y,z$ via a
modification of $J(\epsilon)$, similar to eq. (\ref{eq:post-exp-eps}). If we
apply the same process of asymptotic simplification and take the ratio, we'll
see that as $\flatlim$, the expectation is the same as if we had (a) clamped $z$
at 0 and (b) treated $x$ as unpenalised. That is in a sense, all that is going on in the
flat limit: under the right change of basis, some dimensions are clamped at
0, some are unpenalised, and the remaining still have some penalisation by the
prior.

\subsection{Putting it all together}
\label{sec:all-together}

Let us derive the right change of basis so that the prior in eq.
(\ref{eq:post-exp-eps}) can be treated as a product of Gaussians with variances
in different orders in $\epsilon$.

For now, consider the $d=1$, smooth case. It will come as no surprise that the
correct change of basis is to go to the orthonormal
polynomials, so that $\ff = \bQ \bg$ where $\bV = \bQ \bR$ is the QR
decomposition of the full Vandermonde matrix (degree $n-1$). Note that the
change-of-variable has determinant one (it's a rotation). We also note $\bg =
[g_0 \ldots g_{n-1}]$, starting the indexing at 0 because of the association
with monomial degrees. 

As usual we can write $\bKe$ as
\begin{equation}
  \label{eq:expansion-kernel}
  \bKe = \bV \Ds \bW \Ds \bV^t + \ldots
\end{equation}
where $\Ds$ is diagonal with $\Ds_{i,i}=\epsilon^i$.
A quadratic form in $\bK^{-1}$ can be expanded as: 
\begin{equation}
  \label{eq:quadform}
  \ff^t \bKe^{-1} \ff = \bg \bQ^t \bV^{-t} \Ds^{-1} \bW^{-1}\Ds^{-1} \bV^{-1} \bQ \bg + \ldots = \bg \bR^{-t} \Ds^{-1} \bW^{-1}\Ds^{-1} \bR^{-1} \bg + \ldots
\end{equation}
The matrix $\bR^{-1}$ is lower-triangular, and so with some algebra you can
convince yourself that in the quadratic form 
\[ \bg \bQ^t \bV^{-t} \Ds^{-1} \bW^{-1}\Ds^{-1} \bV^{-1} \bQ \bg \]
the term in $g_0^2$ has valuation $0$ in $\epsilon$, the term in $g_1^2$ has
valuation $-2$, the term in $g_2^2$ has valuation $-4$, etc. This means we can rewrite the prior as effectively proportional to:
\[ \exp(-\frac{\gamma_0 \epsilon^{p}}{2} (c_0g_0^2 + \epsilon^{-2} c_1g_1^2 +
  \ldots  ))\]

At this stage we can apply our asymptotic integral formulas and we are done.
Depending on $p$ some dimensions become unpenalised, some are clamped to 0 and
the one that (potentially) remains is penalised. We have a semiparametric prior,
the exact nature of which depends on the expansion of the kernel matrix at 0. To
identify this model, we repeat the arguments of the linear-Gaussian case. 

\subsection{Non-analytic kernels}
\label{sec:non-analytic-kernels}

As it stands our results are limited to kernels that are analytic at 0. This
limitation is frustrating because, e.g. not all Matérn kernels are analytic. In
particular, numerical evidence suggests that thin-plate splines
\cite{wood2006generalized} do appear in the flat limit of Matérn kernel, but
this is slightly beyond what we can prove using the current tools. In this
section we explain the problem and sketch a direction for the proof.

To explain the issue, let us consider the Matérn kernel in dimension 1, as
defined in \cite{stein1999interpolation}, eq. 14. Let $k(x,y) = \psi(x-y)$, with \begin{equation}
  \label{eq:matern-1d}
  \psi(t) = \frac{\sqrt{\pi}}{2^{\nu-1}\Gamma(\nu+\frac{1}{2})}|t|^{\nu}\mathcal{K}_\nu(|t|)
\end{equation}
where $\mathcal{K}_\nu$ is a modified Bessel function. Recall that $\nu$ determines
the m.s. differentiability of the process: a sample from a GP with Matérn
covariance is $m$ times differentiable if and only if $\nu > m$. Using the tools
in this article, we can only handle the cases where $\nu$ is half-integer,
because in these cases $\psi(t)=\exp(-|t|)p(t)$ where $p(t)$ is a polynomial
\cite{stein1999interpolation}.

Stein \cite{stein1999interpolation} gives asymptotic series for the Matérn
kernel (eqs. 15 and 16, p. 32) which shed light on the behaviour of $\psi(t)$ in
small $t$. For simplicity, we only consider $\nu \leq 1$.
If $0 < \nu < 1$, then $\psi$ has an expansion of the form:
\[ \psi(t) = b_0 - b_1 t^{2\nu} + \O(|t|^2) \] where $b_0$ and $b_1$ are
coefficients depending on $\nu$. Note that $2\nu$ is not an integer, so $\psi$
is not asymptotic to a power series.
If $\nu = 1$, then $\psi$ has expansion:
\[ \psi(t) = 2 + t^{2}\log|t| + \O(t^2) \] 
where $t^{2}\log|t|$ is not a monomial. 

In terms of kernel matrices, this implies that if $0<\nu<1$, $\bKe$ can be
expanded as:
\[ \bKe = b_0 \ones\ones^t - \varepsilon^{2\nu} b_1 \bD^{(2\nu)} + \O(\varepsilon^2) \]
where $\bD^{(2\nu)} = [|x_i-x_j|^{2\nu}]_{i,j}$. If $\nu = \frac{1}{2}$, we
recover the special case of the exponential kernel. If $\nu$ is any other real
number, we need to deal with a non-analytic matrix perturbation. This creates a
problem because we use Rellich's theorem, which assumes that the perturbation is
analytic. The same problem arises when $\nu=1$, in which case:
\[ \bKe = 2 \ones\ones^t + (\varepsilon^{2}\log \varepsilon)  \bL + \O(\varepsilon^2) \]
where $\bL = [|x_i-x_j|^{2}\log |x_i-x_j|]_{i,j}$. Again we cannot use Rellich's
theorem. 

It is relatively easy to see what should happen here: the scaling of $\gamma$
needs to be adapted to scale as $\varepsilon^{-2\nu}$ or
$\frac{1}{\varepsilon^{2}\log \varepsilon}$ when $\nu=1$. What we need is a
version of the results in \cite{BarthelmeUsevich:KernelsFlatLimit} that does not
rely on Rellich's theorem. There is a possibility of obtaining the same results
using linear algebra in the field of transseries \cite{van2006transseries}, and
we hope to pursue this in the future.

\bibliographystyle{plainnat}
\bibliography{gps}

\end{document}